\documentclass[pdftex,12pt,a4paper]{amsart}
\usepackage[pdftex]{graphicx}
\DeclareGraphicsExtensions{.png,.pdf}
\usepackage{thumbpdf}
\usepackage{comment}
\usepackage[export]{adjustbox}
\usepackage{mathrsfs}
\graphicspath{{figs/}}
\usepackage[english]{babel}
\usepackage[utf8]{inputenc}
\usepackage[export]{adjustbox}
\usepackage{amssymb}
\usepackage{amsmath}
\usepackage{amsthm}
\usepackage{color}
\usepackage{hyperref}
\usepackage{url}
\usepackage{bm}
\usepackage{enumerate}
\usepackage{tikz}
\newcommand{\R}{{\mathbb R}}
\newcommand{\N}{{\mathbb N}}
\newcommand{\T}{{\mathbb T}}
\newcommand{\Z}{{\mathbb Z}}
\newcommand{\C}{{\mathbb C}}
\newcommand{\K}{{\mathcal{K}}}
\newcommand{\W}{{\mathcal{W}}}
\newcommand{\Or}{\mathcal{O}}

\newcommand{\bmi}{\bm{i}}

\def\ee{\mathrm{e}}

\def\deltalambda{{\Delta\lambda}}

\newcommand{\err}{\mathop{\rm err}\nolimits}

\theoremstyle{definition}
\newtheorem{Def}{Definition}[section]
\newtheorem{remark}[Def]{Remark}

\theoremstyle{plain}

\newtheorem{lemma}{Lemma}[section]

\theoremstyle{remark}
\newtheorem{algo}{Algorithm}[subsection]

\begin{document}

\title[Flow map methods for quasi-periodic systems]{
Flow map parameterization methods for invariant tori in Quasi-Periodic Hamiltonian systems
}
\author{Álvaro Férnandez}
\address{Departament de Matem\`atiques i Inform\`atica, Universitat de Barcelona,
Gran Via 585, 08007 Barcelona, Spain.}
\email[Corresponding author]{alvaro.fernandez@ub.edu}
\author{Àlex Haro}
\address{Departament de Matem\`atiques i Inform\`atica
Universitat de Barcelona, Gran Via 585, 08007 Barcelona, Spain \& Centre de Recerca Matemàtica, Edifici C, Campus Bellaterra, 08193 Bellaterra (Barcelona), Spain.}
\email{alex@maia.ub.es}
\author[J.M. Mondelo]{J.M. Mondelo}
\address{Departament de Matemàtiques \& CERES-IEEC, Universitat Autònoma de Barcelona, Av.~de l'Eix Central, Edifici C, 08193 Bellaterra (Barcelona), Spain.}
\email{jmm@mat.uab.cat}

\thanks{This work has been supported by the Spanish grants MCINN-AEI PID2020-118281GB-C31 and PID2021-125535NB-I00,
the Catalan grant 2017 SGR 1374, by the Spanish State Research Agency, through the
Severo Ochoa and María de Maeztu Program for Centers and Units of
Excellence in R\&D (CEX2020-001084-M), the European Union's Horizon 2020 research and innovation program
under the Marie Sk\l{}odowska-Curie grant agreement No 734557, the Secretariat for Universities and Research
of the Ministry of Business and Knowledge of the Government of Catalonia, and by the European Social Fund.
}

\date{}

\begin{abstract}
The purpose of this paper is to present a method to compute parameterizations of invariant tori and bundles in non-autonomous quasi-periodic Hamiltonian systems. We generalize flow map parameterization methods to the quasi-periodic setting. To this end, we introduce the notion of fiberwise isotropic tori and sketch definitions and results on fiberwise symplectic deformations and their moment maps. These constructs are vital to work in a suitable setting and lead to the proofs of ``magic cancellations" that guarantee the existence of solutions of cohomological equations.

We apply our algorithms in the Elliptic Restricted Three Body Problem and compute non-resonant $3$-dimensional invariant tori and their invariant bundles around the $L_1$ point. 
\smallskip

\end{abstract}
 % abstract.tex

\maketitle

\section{Introduction}
The study of invariant manifolds constitutes the center piece in understanding dynamical systems. It is a rather natural first approach ---and often the only hope--- to unveil the qualitative behavior of a time-evolving system. Besides the intrinsic interest of invariant manifolds, such structures have found their ``real-world'' analogues in celestial mechanics, astrodynamics and mission design, plasma physics, semi-classical quantum theory, magnetohydrodynamics, neuroscience, and the list goes on. In particular, celestial mechanics has a long-held tradition in considering such objects, especially periodic orbits and invariant tori carrying quasi-periodic motion, and is in fact one of the main fields that promoted their rigorous and numerical study. Astronomers have used perturbative techniques for centuries in the form of formal series of dubious convergence due to the existence of the so-called small divisors problems. Progress had to wait until the pioneering works \cite{Kolmogorov54,Arnold63a,Moser62} that gave birth to the celebrated KAM theory and the subsequent proofs on the persistence of invariant tori under small enough perturbations of integrable systems.

In later works, KAM theory was carried out far from the perturbative regime, without the need of action-angle variables \cite{GonzalezJLV05}, and rigorous results and algorithms were developed for hyperbolic invariant tori and their whiskers \cite{HaroL06b,HaroL06a,HaroL07}. The methodology is based on the solution of functional equations in the spirit of the parametrization method introduced in \cite{CabreFL03a,CabreFL03b,CabreFL05} for invariant manifolds of fixed points. For the solution of the functional equations, the parameterization method constructs a Newton-like sequence of functions in a scale of Banach spaces that converges to the solution starting from an initial approximation. The results are stated following a posteriori formulation: if there is an approximate solution of the invariance equation satisfying some non-degeneracy conditions, then there exist a true solution nearby. Rather rapidly, the parameterization method lead to a plethora of rigorous results \cite{CallejaCL13a,CanadellH17a,FontichLS09,GonzalezHL13,HuguetLS12,LuqueV11} and numerical explorations \cite{CallejaCGLl22,callejanontwist,CanadellH17b,gonzalez2022,HM21,KumarALl22}, in different contexts, to name a few. See \cite{book_alex} for a  survey.

Our objective is to design a flow map parameterization method in the spirit of \cite{HM21} to compute non-resonant partially hyperbolic invariant tori and their invariant bundles in quasi-periodic Hamiltonian systems. Time-dependent Hamiltonians appear naturally in astrodynamics and celestial mechanics as improvements of the Circular Restricted Three Body Problem (CRTBP). There is a hierarchy of models of increasing complexity that provide a closer behavior to the real solar system dynamics, which is generally accepted to be given by the Newtonian attraction of the celestial bodies described according to the JPL ephemeris \cite{dei2017advantages}. The improvements of the CRTBP include the Elliptic Restricted Three Body Problem, the Quasi Bicircular problem in the thesis of \cite{andreu1998quasi}, and the frequency models of \cite{gomez2002solar}, to mention a few. These improved models enable the consideration of advanced space missions concepts and can serve as a seed to compute bounded motion for several decades in the JPL ephemeris model \cite{andreu1999,andreu2003}. The application of the parameterization method ideas to non-autonomous complex models in astrodynamics and celestial mechanics is our main motivation.

Flow maps methods allow the reduction of the torus dimension to be computed by one  \cite{GomezM01,HM21}. The operation count to manipulate functions grows exponentially with the number of variables of the parameterization. Therefore, the reduction allowed by flow map methods is computationally advantageous. This comes at the expense of numerical integration, which can be easily parallelized. A similar argument can be made for using the parameterization method instead of following a normal form approach. Normal forms require to manipulate functions of the same number of variables as the dimension of the phase space. Instead, the parameterization method requires to manipulate functions with as many variables as the dimension of the invariant manifold.

The parameterization method leads to very efficient algorithms. If the parameterization is approximated with either $N$ sample values in a regular grid or $N$ Fourier coefficients, the Newton-like step requires $\Or(N)$ storage and $\Or(N\log N)$ operations as opposed to $\Or(N^2)$ storage and $\Or(N^3)$ operations of classical Newton methods applied directly to discretized versions of the functional equations. The gain in efficiency comes from the geometrical properties of the phase space (i.e.~symplectic geometry), the systems (i.e.~exact symplecticity) and the tori (i.e.~isotropicity, Lagrangianity). These properties lead to a Newton step that is decomposed into substeps that require $\Or(N)$ operations either in grid space or Fourier space where the $\Or(N\log N)$ cost comes from the FFT performed in order to switch representation spaces. See e.g. \cite{book_alex} and references therein.

An important ingredient in the design of algorithms based on the parameterization method is the presence of ``magic cancellations''.  These cancellations come as well from the geometrical properties and they allow the solution of the so-called small divisors equations. Such equations appear naturally in the algorithm and are the hallmark of KAM theory. Even though in this paper we will not provide a convergence proof of the algorithm, we will provide all the elements to produce such a proof with KAM techniques. See e.g. \cite{GonzalezJLV05, FiguerasHL17, FontichLS09, book_alex, LuqueV11}. We emphasize that these works, as well as our prequel \cite{HM21}, hold for autonomous Hamiltonian systems. A standard practice when working with non-autonomous systems is to consider an extended phase space by defining extra angle variables and conjugated fictitious variables to make the system autonomous. Although this is mathematically equivalent, this incurs in increasing the dimension of the phase space which leads to less efficient algorithms. We avoid this practice by considering appropriate functional equations. Another standard practice to
deal with non-autonomous systems, in particular periodic, is to use flow map methods for a time-$T$ map that make the discrete system autonomous, i.e., choosing for $T$ the period of the system. Instead, we take $T$ as one of the internal periods of the torus sought for. This enables continuation from an autonomous approximation of quasi-periodic models starting directly from tori of the autonomous approximation computed via flow map methods.

%When a periodic system is homotopic to an autonomous system, the starting point to compute invariant manifolds is often to consider the manifolds of the autonomous problem. Then, by continuation methods, it is possible to compute the objects in the periodic system given certain regularity conditions of the manifolds with respect to parameters.

%If the program is to compute invariant manifolds under flow maps in an autonomous system and to continue them to a periodic system, these objects will in principle not be invariant for $T$ the period of the system. Consequently, additional action is required which makes the approach not optimal. Instead, if $T$ is the period associated to a frequency of an autonomous torus, one can simply add frequencies and dimensions to the invariant torus which is a more natural and convenient approach for such a program.

%In \cite{HM21} we developed a flow map parameterization method for (partially hyperbolic) invariant tori in  autonomous Hamiltonian systems, and proved the presence of the ``magic cancellations'' in such a context.

In order to generalize the parameterization method to the quasi-periodic Hamiltonian setting, we do not make use of an autonomous reformulation. Instead, we introduce new geometrical machinery that generalizes the ideas of \cite{HM21} and allow us to prove the ``magic cancellations'' in the quasi-periodic context. In geometrical jargon, we consider the extended phase space as a symplectic bundle, the invariant tori are fiberwise isotropic, and we sketch definitions and results on fiberwise symplectic deformations. We also construct the corresponding so-called moment maps which can be seen as generating Hamiltonians of the deformations.

The geometrical constructs are crucial for our method and could be of independent interest. They are also crucial for the eventual proofs on convergence.  In spite of their importance, we present them in the appendices for ease of exposition of the method.

 % intro.tex

%\section{Preliminaries}
%\label{sec:prem}
%\input{preliminaries}  % preliminaries.tex

\section{Setting}
\label{sec:setting}
We assume all objects to be sufficiently smooth, even real analytic.

\subsection{Hamiltonian systems}
Let us assume we have an {\em exact symplectic form} $\boldsymbol{\omega}$ on an open set $U\subset\R^{2n}$ endowing $U$ with an exact symplectic structure and let $\Omega:U\to\R^{2n\times2n}$ be the matrix representation of $\boldsymbol{\omega}$. Let us also assume we have a smooth function $H:U\times\R\to\R$ on the extended phase space that depends explicitly on time. Since the symplectic form is bilinear and non-degenerate, $\boldsymbol{\omega}$ sets a fiberwise linear isomorphism between 1-forms and vector fields. Therefore, there is a unique vector field $X_H:U\times\R\to \R^{2n}$ obtained from the differential of the Hamiltonian as
\begin{equation}\label{eq:hamsys}
   \dot z=X_H(z,t):=\Omega(z)^{-1} {\rm D}_z H(z,t)^\top,
\end{equation}
where $X_H$ is the {\em Hamiltonian vector field} of the {\em non-autonomous Hamiltonian system} generated by $H$.

In the present work, we focus on a subset of time-dependent Hamiltonians that frequently arise in physical models such as those in celestial mechanics. We will consider the quasi-periodic case where the time-dependent Hamiltonian is a quasi-periodic function with frequencies $\hat{\alpha}\in\R^\ell$. Let $\T^m=\R^m/\Z^m$ be the standard $m$-torus. We can define the angle variables $\varphi:=\hat\alpha t\in\T^\ell$, and with a slight abuse of notation, we consider the Hamiltonian as a quasi-periodic smooth function $H:U\times\T^\ell\to\R$. Analogously, we consider the corresponding Hamiltonain vector field $X_H:U\times\T^\ell\to\R^{2n}$. Then, on the extended phase space, we have the following vector field $\tilde X_H:U\times \T^\ell \to \R^{2n}\times\R^\ell $ constructed as
\begin{equation}
\label{eq:diff_eq}
\begin{pmatrix}
\dot{z} \\
\dot{\varphi}
\end{pmatrix} = \tilde X_H(z,\varphi) = \begin{pmatrix}
X_H(z,\varphi)\\
\hat{\alpha}
\end{pmatrix}.
\end{equation}
Note that when $\varphi\in\T$, our system reduces to the periodic case. We look at $\tilde X_H$ as a quasi-periodic vector field defined on the total space $U\times\T^\ell$ of a bundle with base $\T^\ell$. On each fiber $\mathcal{F}_\varphi = \pi^{-1}(\varphi)$, where $\pi$ is the bundle projection, we have a symplectic vector structure and a Hamiltonian vector field $X_H(\cdot,\varphi):U\to\R^{2n}$. Note that all these vector fields are coupled by the phase equation $\dot\varphi=\hat\alpha$ according to \eqref{eq:diff_eq}. The vector field $\tilde X_H$ is then a fiberwise Hamiltonian vector field---we will see our objects inherit this fiberwise structure in different contexts.

We denote the flow associated to $\tilde X_H$ by $\tilde\phi:D\subset \R \times U\times\T^\ell\to U\times\T^\ell$, where $D$ is the open set domain of definition of the flow. Then, for every $(z,\varphi)$, we have the maximal interval of existence $I_{z,\varphi}$ such that the domain of the flow can be expressed as
\[
D=\{(t,z,\varphi)\in \R\times U\times\T^\ell\mid t\in I_{z,\varphi}  \}.
\]
The flow adopts the form
\begin{equation}
\label{eq:ext_flow}
  \tilde{\phi}(t,z,\varphi)=\begin{pmatrix}\phi(t,z,\varphi)\\\varphi + \hat{\alpha}t \end{pmatrix},
\end{equation}
where the evolution operator $\phi$ satisfies
\begin{gather*}
    \frac{d}{dt}  \phi(t,z,\varphi) = X_H\big(\phi(t,z,\varphi),\varphi + \hat \alpha t\big),\\
    \phi(0,z,\varphi)=z.
\end{gather*}
From now on, we will adopt the standard notations 
\begin{equation*}
    \begin{gathered}
    \tilde\phi(t,z,\varphi) = \tilde\phi_t(z,\varphi),\\
    \phi(t,z,\varphi) = \phi_t(z,\varphi).
    \end{gathered}
\end{equation*}
Since $\boldsymbol{\omega}$ is exact, the matrix representation of the 2-form is given by
\begin{equation}
\label{eq:exactness}
\Omega(z)={\rm D} a(z)^\top-{\rm D} a(z),
\end{equation}
where $a:U\to\R^{2n}$ and $a(z)^\top$ is the matrix representation at $z\in U$ of the {\em action form} $\bm{\alpha}$ defined on $U$. For fixed $t$, $\phi_t$ is fiberwise exact symplectic: for each $\varphi\in\T^\ell$, $\phi_t$ satisfies symplecticity,
\begin{equation}
\label{eq:symplecticity}
  	{\rm D}_z\phi_t(z,\varphi)^\top \Omega\bigl(\phi_t(z,\varphi)\bigr) {\rm D}_z\phi_t(z,\varphi)= \Omega(z),  
\end{equation}
 and exactness,
\[
	a\bigl(\phi_t(z,\varphi)\bigr)^\top {\rm D}_z\phi_t(z,\varphi) - a(z)^\top = {\rm D}_z p_t(z,\varphi),
\]
for some {\em primitive function} $p_t:U\times\T^\ell\to \R$, see Appendix \ref{ap:primitive} for a explicit form of $p_t$. The existence of the primitive function for the fiberwise exact symplectomorphism $\phi_t$ allows certain cancellations that are crucial in our iterative scheme for the computation of parameterizations of invariant tori and bundles.  

We will also assume we have an almost-complex structure $\bm J$ on $U$ compatible with the symplectic structure, i.e., we have a matrix map $J:U\to\R^{2n\times2n}$ that is anti-involutive and symplectic; that is
\begin{gather*}
    J(z)^2=-I_{2n},\\
    J(z)^\top\Omega(z)J(z)=\Omega(z).
\end{gather*}
The almost-complex structure induces a {\em Riemannian metric} $\bm g$ with a matrix representation $G:U\to\R^{2n\times2n}$ defined as $G(z):=-\Omega(z)J(z)$ at each $z\in U$.

Note that we assume the symplectic form to be independent of $\varphi$ but not constant in $U$. In the standard case, the symplectic structure is given by
\[
\Omega_0(z)=\begin{pmatrix} O_n & -I_n\\
I_n & O_n \end{pmatrix}
\]
and the action form, the almost-complex structure, and the metric adopt the matrix representations
\[
	a_0(z)= \tfrac12 \begin{pmatrix} O_n & I_n \\ -I_n & O_n \end{pmatrix} z,\ 
	J_0(z)=    \begin{pmatrix} O_n & -I_n \\ I_n & O_n \end{pmatrix},\ 
	G_0(z)=  \begin{pmatrix} I_n & O_n \\ O_n & I_n \end{pmatrix}.
\]

%Let us introduce $\mathrm{z}=(z,\varphi)$ and $\mathcal{Z}=(X_H,\hat{\alpha})$. We can then rewrite the system given by \eqref{eq:diff_eq1}-\eqref{eq:diff_eq2} in a more compact form as
%\begin{equation}
%    \label{eq:diff_eq3}
%    \dot{\mathrm{z}}=\mathcal{Z}(\mathrm{z});
%\end{equation}

%We denote the Hamiltonian flow associated to \eqref{eq:diff_eq3}:diff_eq1} by $\phi:U\times\T^\ell\times\R\to U$ and the flow associated to \eqref{eq:diff_eq3} by $\tilde{\phi}:U\times\T^\ell\times\R\to U\times\T^\ell$. We can then define the time-$\tau$ maps obtained for fixed $\tau$ from the flows of \eqref{eq:diff_eq1} and \eqref{eq:diff_eq3} by $\phi_\tau:U\times\T^\ell\to U$ and $\tilde{\phi}_\tau:U\times\T^\ell\to U\times\T^\ell$, respectively, such that

\subsection{Invariance equations for invariant tori}\label{sec:invK}
We are interested in quasi-periodic solutions for the system given by \eqref{eq:diff_eq}, with frequencies $\hat{\omega}\in\R^{d}$ and $\hat{\alpha}\in\R^\ell$---which correspond to the frequencies of the Hamiltonian. We will refer to $\hat\omega$ and to $\hat\alpha$ as internal and external frequencies, respectively. Geometrically speaking, quasi-periodic solutions lie in $(d+\ell)$-dimensional tori $\tilde{\K}\subset U\times\T^\ell$. In the light of the parameterization method, we consider suitable parameterizations $\tilde{K}:\T^d\times\T^\ell\to U\times \T^\ell$ that conjugate the dynamics in $\tilde \K$ to a linear flow in $\T^d\times\T^\ell$
with frequency $(\hat\omega,\hat\alpha)$. In particular, the frequencies $\hat\omega$ and $\hat\alpha$ need to be, at least, non-resonant or ergodic. That is, 
\[k\cdot \hat \omega + j \cdot \hat \alpha \neq 0 \quad \text{for}\quad (k,j)\in \Z ^{d}\times\Z^\ell\setminus \{0\},
\] 
where $\cdot$ is the standard scalar product. Then, for the range of $\tilde K$ to be an invariant torus, the parameterization is required to satisfy the functional equation
\begin{equation} \label{eq:inv_ext}
\tilde\phi_t\circ \tilde K(\hat\theta,\varphi) - \tilde K (\hat\theta + \hat\omega t, \varphi + \hat\alpha t) = 0,
\end{equation}
where $\hat\theta\in\T^d, \varphi \in \T^\ell$, and $t\in\R$.
\begin{remark}
 By differentiating \eqref{eq:inv_ext} with respect to $t$, we obtain the following vector field version of the invariance equation 
\[
\tilde X_H\circ \tilde K(\hat\theta,\varphi)={\rm D}_{\hat\theta} \tilde K(\hat\theta,\varphi)\hat\omega + {\rm D}_\varphi \tilde K(\hat\theta,\varphi)\hat \alpha.
\]
\end{remark}
Recall that the extended phase space, $U\times\T^\ell$, has a bundle structure with $\T^\ell$ as the base space. Because of this bundle structure, we consider parameterizations for $\tilde \K$ of the form
\begin{equation}
\tilde K(\hat\theta,\varphi) = \begin{pmatrix}
\hat K(\hat\theta,\varphi)\\
\varphi
\end{pmatrix},
\end{equation}
where $\hat K : \T^d\times\T^\ell\to U$ parameterizes a $(d+\ell)-$dimensional invariant torus $\hat\K\subset U$. Then, for $\tilde K$ to satisfy \eqref{eq:inv_ext}, it suffices that $\hat K$ satisfies the invariance equation
\begin{equation}
    \label{eq:inv_cont}
    \phi_t\bigl(\hat{K}(\hat\theta,\varphi),\varphi\bigr) - \hat{K}(\hat\theta + \hat{\omega}t,\varphi + \hat{\alpha}t)=0.
\end{equation}

From a computational point of view, the cost to compute parameterizations rapidly increases with the dimension of the torus. We therefore follow the trick from \cite{GomezM01,HM21} and look for a parameterization  $K:\T^{d-1}\times\T^\ell\to U$ of a $(d+\ell-1)$-dimensional torus $\K\subset\hat\K$, invariant under time-$T$ maps where $T$ is the period associated to one
of the internal frequencies of $\hat \K$.

Let us first define $\hat{\omega}=:\tfrac{1}{T}(\omega,1)$ and $\hat{\alpha}=:\tfrac{1}{T}\alpha$, where $\omega\in\R^{d-1}$ and $\alpha\in\R^\ell$. In what follows, we will require of $(\omega,\alpha)$ stronger non-resonance conditions. We will assume $(\omega,\alpha)$ to be Diophantine, meaning that there exists $\gamma>0$ and $\tau\geq d + \ell-1$ such that for all $ n\in\Z$ 
\[
|k\cdot\omega+j\cdot\alpha-n|\geq \frac{\gamma}{\left(|k|_1+|j|_1\right)^\tau}\quad \text{for}\quad (k,j) \in \Z^{d-1}\times\Z^\ell \setminus \{ 0\}, 
\]
where $|\cdot|_1$ is the $\ell^1$-norm. See Remark \ref{remark:diophantine}.

We can then assume $\hat\theta=(\theta,\theta_d)$, for some fixed $\theta_d\in\T$ and with $\theta\in\T^{d-1}$, and consider a parameterization for $\K$ of the form $K(\theta,\varphi)=\hat K(\theta,\theta_d,\varphi)$. If we look at the invariance equation \eqref{eq:inv_cont} for $t=T$ and some fixed $\theta_d$, we observe that 
\[
\phi_T\big(\hat K(\theta,\theta_d,\varphi),\varphi\big) - \hat K(\theta+\omega,\theta_d+1,\varphi+\alpha)=0,
\]
from where we obtain the following invariance equation for $K
$
\begin{equation}\label{eq:inv_flowmap}
\phi_T\big(K(\theta,\varphi),\varphi\big) - K(\theta+\omega,\varphi+\alpha) = 0.
\end{equation}
We refer to $\omega$ and $\alpha$ as the internal and external rotation vectors, respectively. We can recover a parameterization for $\hat{\K}$, and consequently for $\tilde \K$, from the parameterization of $\K$ as
\[
\hat K(\hat\theta,\varphi)=\phi_{\theta_d T}\big( K(\theta-\theta_d\omega,\varphi-\theta_d\alpha),\varphi-\theta_d\alpha\big)
\]
We refer to $\mathcal{K}$ as the generator of $\hat{\K}$ and to $T$ as the flying time of $\mathcal{K}$. We emphasize that, although we began by considering a $(d+\ell)-$dimensional invariant torus $\tilde \K$ living in the extended phase space $U\times\T^\ell$, this torus is completely determined by $(d+\ell-1)-$dimensional tori $\K$ living in $U$. Consequently, with this formulation, we not only manage to reduce the dimension of the phase space but also of the invariant tori to be computed.

\begin{remark}\label{remark:underK}
For invariant tori $\hat{\mathcal{K}}$ and $\mathcal{K}$, the parameterizations $\hat K$ and $K$ are not unique. For $\hat a \in \R^d$, if $\hat K$ satisfies Eq. \eqref{eq:inv_cont}, $\hat K(\hat\theta + \hat a,\varphi)$ is also a solution parameterizing the same torus. Similarly, for $a \in \R^{d-1}$ and $\tau\in\R$, if $K$ satisfies Eq. \eqref{eq:inv_flowmap}, $\phi_\tau \big(K(\theta + a,\varphi),\varphi\big)$ is also a solution generating the same torus as $K$. Consequently, Eq. \eqref{eq:inv_cont} determines $\hat K$ up to a $d-$dimensional phase shift and Eq. \eqref{eq:inv_flowmap} determines $K$ up to a $(d-1)$-dimensional phase shift and up to a translation within the torus $\hat{\mathcal{K}}$. Hence, for both $\hat {\mathcal{K}}$ and $\mathcal K$, we have $d$ degrees of freedom in their parameterizations.
\end{remark}
 
\begin{remark} \label{remark:freq}
The frequency and the rotation vectors $\hat\omega$ and $\omega$, respectively, are defined up to unimodular matrices. To the frequency vector $\hat A\hat\omega$, with $\hat A\in\Z^{d\times d}$, corresponds the reparameterization $(\hat\theta,\varphi)\mapsto\hat K(\hat A^{-1}\hat\theta,\varphi)$. Similarly, to the rotation vector $A\omega$, with $A\in\Z^{(d-1)\times (d-1)}$,
corresponds the reparameterization $(\theta,\varphi)\mapsto K(A^{-1}\theta,\varphi)$.
\end{remark}

\begin{remark}
When the Hamiltonian is periodic with period $T_H$, it is somewhat standard to look for invariant tori under time-$T_H$ maps, see e.g. \cite{CastellaJ00}. Let us illustrate our motivation for taking time$-T$ maps, where $T$ is the period associated to one of the internal frequencies, with a simple example commonly found in applications. Assume we have a $T_H-$periodic Hamiltonian of the following form,
\begin{equation*}
H(z,\varphi)=H_0(z) + \varepsilon H_1(z,\varphi),
\end{equation*}
where $\varepsilon\in\R$ is some parameter not necessarily small. Also assume that we have computed a parameterization $K_0:\T^{d-1}\to U$ of a $(d-1)$-dimensional generating torus for the autonomous system given by $H_0$. It is then natural to consider a parameterization $K_\varepsilon:\T^{d-1}\times\T\to U$ of a $d$-dimensional generating torus for the system given by $H$ that can be obtained from $K_0$ by continuation methods in $\varepsilon$. If we take time-$T$ maps where $T$ is the period associated to one of the frequencies of the torus parameterized by $K_0$, we can construct $K_\varepsilon$ at $\varepsilon=0$ as
\begin{equation}\label{eq:K0}
K_\varepsilon(\theta,\varphi) = K_0(\theta),
\end{equation} 
from where we can directly start the continuation in $\varepsilon$. See Section \ref{sec:cnte}. The same argument also holds for the quasi-periodic case.
\end{remark}

\subsection{Invariant bundles of rank 1}\label{sec:invW}
In the present work, we focus on $(d+\ell)$-dimensional partially hyperbolic invariant tori with $d=n-1$. Our choice is motivated by applications such as quasi-periodic perturbations of the Restricted Three Body Problem---this will be the test case explored in Section \ref{sec:application}. The results can easily be adapted for Lagrangian tori $(d=n),$ and for the lower dimensional case $(d<n-1)$. In the later case, a complication is that partially hyperbolic invariant tori are not necessarily reducible. Nonetheless, there exist strategies \cite{HaroL06b,HaroL07,HuguetLS12}. Following the results from Section \ref{sec:invK}, we will directly consider rank 1 bundles of $\hat\K$ and $\K$ instead of bundles of $\tilde\K$.

Let us first consider bundles of rank 1, $\hat{\mathcal{W}}$, with base $\hat{\mathcal{K}}$, invariant under the linearization of $\phi_t$ on $\hat{\mathcal{K}}$, where the dynamics on the fibers is contracting or expanding. We then look for a parameterization $\hat{W}:\T^{d}\times\T^\ell\to\R^{2n}$ satisfying \begin{equation}
    {\rm D}_z\phi_t\bigl(\hat{K}(\hat\theta,\varphi),\varphi\bigr)\hat{W}(\hat\theta,\varphi)=
	e^{t\chi}\hat{W}(\hat\theta+\hat{\omega}t,\varphi+\hat{\alpha}t),
\end{equation}
with $\chi\in \R$. If $\chi<0$, $\hat{\mathcal{W}}=\hat{\mathcal{W}}^s$ is the stable bundle and if $\chi>0$, $\hat{\mathcal{W}}=\hat{\mathcal{W}}^u$ is the unstable bundle. 

We again reduce the dimension of the object by using time$-T$ maps and look for a parameterization $W:\T^{d-1}\times\T^\ell\to\R^{2n}$ of a bundle $\mathcal{W}$ with base $\mathcal{K}$ satisfying
\begin{equation}\label{eq:invW}
   {\rm D}_z\phi_T\bigl(K(\theta,\varphi),\varphi\bigr)W(\theta,\varphi)
   =
  W(\theta+\omega,\varphi+\alpha)\lambda,
\end{equation}
with $\lambda=e^{T \chi}$. We can recover the parameterization of the bundle $\mathcal{\hat W}$ as
\begin{align*}
\hat W(&\hat\theta,\varphi) =\\ 
&e^{-\theta_d T\chi}{\rm D}_z\phi_{\theta_d T}(K(\theta-\theta_d\omega,\varphi-\theta_d\alpha),\varphi-\theta_d\alpha)W(\theta-\theta_d\omega,\varphi-\theta_d\alpha).
\end{align*}
Note that because $\hat{\mathcal{W}}$ and $\mathcal{W}$ have rank 1, the dynamics on the fibers is a uniform contraction when $\lambda<1$ and a uniform expansion when $\lambda>1$. We will also refer to $\mathcal{W}$ as the generator of $\hat{\mathcal{W}}$, to $\lambda$ as the Floquet multiplier of $\mathcal{K}$, and to $\chi$ as the Floquet exponent of $\mathcal{\hat{K}}$.
\begin{remark}
Rank 1 stable and unstable bundles are reducible (to constant $\lambda$). In general, one could consider rank $m$ stable and unstable bundles that might not be reducible \cite{HaroL06b,HaroL07,HuguetLS12}.
\end{remark}
\begin{remark}
We are implicitly assuming the bundles are trivial or that can be trivialized (e.g.,  by using the double covering trick \cite{HaroL07}).
\end{remark}
\begin{remark}
The invariant bundle $\hat{\mathcal{W}}$, and analogously $\mathcal{W}$, is the linearization on $\hat{\mathcal{K}}$ of certain manifolds $\widehat {\mathcal{W}}$ defined on the annulus: the whiskers. If $\widehat W:\T^{d}\times\T^{\ell}\times\R\to U$ is a parameterization of a whisker, then it satisfies the invariance equation
\begin{equation*}
  \phi_t\big(\widehat W(\hat\theta,\varphi,s),\varphi\big) = \widehat W(\hat\theta + \hat\omega t,\varphi+\hat\alpha t,e^{t\chi} s),
\end{equation*}
where $s\in\R$. Also note that at $s=0$, $\widehat W(\hat\theta,\varphi,0) = \hat K(\hat\theta,\varphi)$.
\end{remark}

\subsection{Geometric properties of invariant tori}
Until now, we have only considered the geometric properties of the phase space. Nevertheless, invariant tori of exact symplectic maps under non-resonance conditions carry certain geometric properties that are of vital importance in our constructions.

For each $\varphi$, let us consider tori $\mathcal{K}_\varphi\subset\mathcal{K}$ with parameterizations $K_\varphi:\T^{d-1}\to U$ constructed as $K_\varphi(\theta) = K(\theta,\varphi)$. If we think of $U\times\T^\ell$ as the total space of a bundle with base $\T^\ell$ and projection $\pi$, each fiber $\mathcal{F}_\varphi = \pi^{-1}(\varphi)$ contains the torus $\mathcal{K}_\varphi$, see Fig. \ref{fig:fiberwise} for a sketch. 

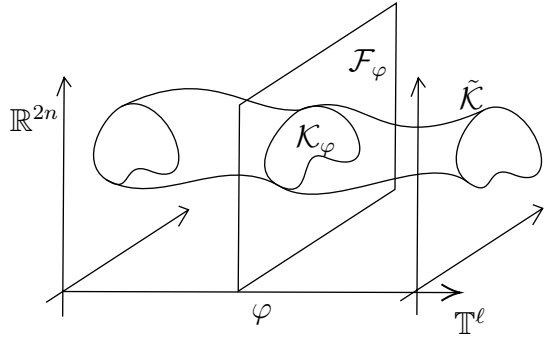
\begin{figure}[h]
\centering
\tikzset{every picture/.style={line width=0.5pt}} %set default line width to 0.75pt        

\begin{tikzpicture}[x=0.75pt,y=0.75pt,yscale=-0.85,xscale=0.85]
%uncomment if require: \path (0,310); %set diagram left start at 0, and has height of 310

%Shape: Axis 2D [id:dp8639308836493105] 
\draw [color={rgb, 255:red, 0; green, 0; blue, 0 }  ,draw opacity=1 ][line width=0.5]  (61.89,198.55) -- (145.54,143.75)(72.52,65.38) -- (71.63,206.98) (139.72,141.58) -- (145.54,143.75) -- (139.65,153.59) (68.28,76.52) -- (72.52,65.38) -- (76.64,71.04)  ;
%Straight Lines [id:da5499005117400797] 
\draw    (71.72,192.12) -- (303,192.74) ;
\draw [shift={(305,192.75)}, rotate = 180.16] [color={rgb, 255:red, 0; green, 0; blue, 0 }  ][line width=0.75]    (10.93,-4.9) .. controls (6.95,-2.3) and (3.31,-0.67) .. (0,0) .. controls (3.31,0.67) and (6.95,2.3) .. (10.93,4.9)   ;
%Shape: Polygon Curved [id:ds7830525147824037] 
\draw   (97,125.75) .. controls (81,114.75) and (95,88.75) .. (106,82.75) .. controls (117,76.75) and (131,86.75) .. (137,100.75) .. controls (143,114.75) and (140,127.75) .. (123,120.75) .. controls (106,113.75) and (113,136.75) .. (97,125.75) -- cycle ;
%Shape: Axis 2D [id:dp322536053977714] 
\draw  (269.89,197.69) -- (353.54,142.89)(280.51,65.69) -- (279.63,205.98) (347.72,140.72) -- (353.54,142.89) -- (347.65,152.73) (276.27,76.84) -- (280.51,65.69) -- (284.64,71.36)  ;
%Shape: Polygon Curved [id:ds022158474355022828] 
\draw   (310,127.75) .. controls (294,116.75) and (308,90.75) .. (319,84.75) .. controls (330,78.75) and (344,88.75) .. (350,102.75) .. controls (356,116.75) and (353,129.75) .. (336,122.75) .. controls (319,115.75) and (326,138.75) .. (310,127.75) -- cycle ;
%Shape: Polygon Curved [id:ds8491860492681872] 
\draw   (199,129.75) .. controls (178,119.75) and (200,90.75) .. (211,84.75) .. controls (222,78.75) and (238,84.75) .. (244,98.75) .. controls (250,112.75) and (243,119.75) .. (226,112.75) .. controls (209,105.75) and (220,139.75) .. (199,129.75) -- cycle ;
%Curve Lines [id:da9885101731215843] 
\draw    (106,82.75) .. controls (158,56.75) and (179,86.75) .. (211,84.75) ;
%Curve Lines [id:da6573813599724072] 
\draw    (211,84.75) .. controls (251,69.75) and (255,113.75) .. (319,84.75) ;
%Curve Lines [id:da047336844064339245] 
\draw    (97,125.75) .. controls (129,143.75) and (165,109.75) .. (199,129.75) ;
%Curve Lines [id:da18087364908660564] 
\draw    (199,129.75) .. controls (227,147.75) and (285,107.75) .. (310,127.75) ;
%Shape: Rectangle [id:dp3909747040232413] 
\draw   (175.49,81.74) -- (267.56,21.49) -- (266.79,131.69) -- (174.72,191.94) -- cycle ;

% Text Node
\draw (39,80.4) node [anchor=north west][inner sep=0.75pt]    {$\mathbb{R}^{2n}$};
% Text Node
\draw (300,200.4) node [anchor=north west][inner sep=0.75pt]    {$\mathbb{T}^{\ell }$};
% Text Node
\draw (181,198.4) node [anchor=north west][inner sep=0.75pt]    {$\varphi $};
% Text Node
\draw (207,92.4) node [anchor=north west][inner sep=0.75pt]    {$\mathcal{K}_{\varphi }$};
% Text Node
\draw (303,65.4) node [anchor=north west][inner sep=0.75pt]    {$\tilde{\mathcal{K}}$};
% Text Node
\draw (238,47.4) node [anchor=north west][inner sep=0.75pt]    {$\mathcal{F}_{\varphi }$};
\end{tikzpicture}

\caption{Sketch of the torus $\K_\varphi$ within the torus $\tilde{\mathcal{K}}$. }
\label{fig:fiberwise}
\end{figure}

Notice that, by differentiating Eq. \eqref{eq:inv_flowmap} with respect to $\theta$, 
we obtain 
\begin{equation}\label{eq:invDK}
   {\rm D}_z\phi_T\bigl(K(\theta,\varphi),\varphi\bigr){\rm D}_\theta K(\theta,\varphi)
   =
  {\rm D}_\theta K(\theta+\omega,\varphi+\alpha).
\end{equation}
We rephrase \eqref{eq:invDK} by saying that the tangent bundle of $\mathcal{K}_\varphi$, 
parameterized by the column vectors of ${\rm D}_\theta K_\varphi$, is transported by the differential ${\rm D}_z \phi_T$ as:
\[
   {\rm D}_z\phi_T\bigl(K_\varphi(\theta),\varphi\bigr){\rm D}_\theta K_\varphi(\theta)
   =
  {\rm D}_\theta K_{\varphi+\alpha} (\theta+\omega).
\]
We show in Appendix \ref{ap:isotropy} that, from this invariance property, $\mathcal{K}_\varphi$ is an isotropic torus for each $\varphi$---that is, 
$K_\varphi^*\boldsymbol{\omega}=0$. In coordinates, this property reads 
\begin{equation*}
%\label{eq:isoK}
{\rm D}_\theta K_\varphi(\theta)^\top\Omega\big(K_\varphi(\theta)\big){\rm D}_\theta K_\varphi(\theta) = 0.
\end{equation*}
Hence, we think of $\mathcal{K}$ as a $\varphi$-parameterized family of $(d-1)$-dimensional isotropic tori and we say that $\mathcal{K}$ is fiberwise isotropic.

As we will see in Section \ref{sec:frm}, because of the invariance of $\K$ and $\W$ and the fiberwise isotropy of $\K$ there exists a set of coordinates given by a symplectic frame $P: \T^{d-1}\times \T^\ell \to \R^{2n\times 2n}$. This frame is partly generated by ${\rm D}_\theta K$ and $W$ and that reduces the linearized dynamics to upper triangular form, constant along the diagonal, as  
\[
P(\theta+ \omega,\varphi +\alpha)^{-1}{\rm D}_z\phi_T\bigl(K(\theta,\varphi),\varphi\bigr) P(\theta, \varphi) = \hat \Lambda, 
\]
with
\[
	\hat \Lambda=    \begin{pmatrix}
            \Lambda & S(\theta,\varphi) \\
                O_n    & \Lambda^{-\top}
         \end{pmatrix}, \quad
       \Lambda= \begin{pmatrix} I_{n-1} & 0 \\ 0 & \lambda \end{pmatrix}
   ,\quad
   S(\theta,\varphi)=
   \begin{pmatrix} S^1(\theta,\varphi) & 0 \\ 0 & 0 \end{pmatrix}.
\]
Where the zero blocks correspond to zero matrices of suitable dimensions and $S^1$ is a $(n-1)\times(n-1)$ symmetric matrix. With some extra work, $S^1$ can be further reduced to a constant matrix. Hence, the linearized dynamics is automatically reduced to a block triangular matrix $\hat \Lambda$. We will use this form of the linearized dynamics at each iteration step to efficiently compute corrections to the parameterizations of $\mathcal{K}$ and $\mathcal{W}$, see Section \ref{sec:method} for the details.

The reducibilty of the linearized dynamics is commonly known as automatic reducibility and it is  an important property both in theory and applications. See e.g. 
\cite{CallejaCL13a,CanadellH14,GonzalezJLV05,FontichLS09,GonzalezHL13,HaroL06b,HaroL06a,HuguetLS12,LuqueV11} for several references on the parameterization method and reducibility in different contexts relatively close to the one of the present paper and \cite{CallejaCGLl22,HM21,KumarALl22} for recent applications to Celestial Mechanics and Astrodynamics.

\subsection{Cohomological equations} \label{sec:coho}
In our iterative scheme we will encounter cohomological equations. We dedicate this section to such equations that frequently appear in KAM theory. The material presented here is rather standard, see e.g. \cite{Llave01,Russmann76a}, but we adapted the formulation to be better suited to our purposes.

In what follows, for functions $\zeta:\T^{d-1}\times\T^\ell\to\R$, that are 1-periodic in each variable, we will often consider their Fourier series
\[
\zeta(\theta,\varphi)=\sum_{k\in\Z^{d-1}}\sum_{j\in\Z^\ell} \hat\zeta_{kj} \ee^{\bmi 2\pi (k \theta + j \varphi)},
\]
where $k\theta:= \sum_{u= 1}^{d-1} k_{u} \theta_{u}$, $j\varphi:= \sum_{v= 1}^\ell j_{v} \varphi_{v}$, and $\bmi$ is the
imaginary unit. Their average is the zero term 
\[
	\hat\zeta_{00} = \langle\zeta\rangle:=\iint_{\T^{d-1}\T^\ell} \zeta(\theta,\varphi)\, d\theta\, d\varphi.
\]

\begin{remark} \label{remark:decay}
Note that, for analytic $\zeta$, the Fourier coefficients go to zero exponentially fast when
\[|(k,j)|_1:=\sum_{u= 1}^{d-1} |k_u| + \sum_{v=1}^\ell|j_v|\] goes to infinity.
\end{remark}

\subsubsection{Non-small divisors cohomological equations} 
Let $\xi,\eta:\T^{d-1}\times\T^\ell\to\R$ be analytic functions and let $\omega\in\R^{d-1}$ and $\alpha\in\R^\ell$ be fixed rotation vectors. For ease of notation, let us define $\bar\theta:=\theta+\omega$ and $\bar\varphi:=\varphi+\alpha$.

We first consider functional equations of the form 
\begin{equation}
\label{eq:nsdcoho}
\lambda  \xi(\theta,\varphi)-\mu\xi(\bar\theta,\bar\varphi)=\eta(\theta,\varphi),
\end{equation} 
with $\lambda,\mu \in \R$ such that $|\lambda|\neq|\mu|$ and where $\eta$, $\omega$, and $\alpha$ are known and $\xi$ is to be found. If we express $\xi$ and $\eta$ as Fourier series
\begin{equation*}
\label{eq:series}
\xi(\theta,\varphi) = \sum_{k\in\Z^{d-1}}\sum_{j\in\Z^\ell}\hat{\xi}_{kj}\ee^{\bmi2\pi(k\theta + j\varphi)}, \quad \eta(\theta,\varphi) = \sum_{k\in\Z^{d-1}}\sum_{j\in\Z^\ell}\hat{\eta}_{kj}\ee^{\bmi2\pi(k\theta + j\varphi)}
\end{equation*}
the solution of \eqref{eq:nsdcoho} is formally given by
\begin{equation*}
\hat{\xi}_{kj} = \frac{\hat\eta_{kj}}{\lambda-\mu \ee^{\bmi 2\pi  (k\omega+j\alpha)}}\mbox{\quad $\forall k,j$}.
\end{equation*}

\subsubsection{Small-divisors cohomological equations}
We will also consider functional equations of the form
\begin{equation}
\label{eq:sdcoho}
\xi(\theta,\varphi)-\xi(\bar\theta,\bar\varphi)=\eta(\theta,\varphi),
\end{equation}
where $\eta$, $\omega$, and $\alpha$ are known and $\xi$ is to be found. If $\langle\eta\rangle=0$, the solution of \eqref{eq:sdcoho} is formally given by 
\[
\hat{\xi}_{00} \in \R \quad \text{free},
\]
\[\hat{\xi}_{kj} = \frac{\hat\eta_{kj}}{1-\ee^{\bmi 2\pi  (k\omega+j\alpha)}}\quad \mbox{ for $k,j\neq 0$}.
\]
% \begin{remark}
% Note that a necessary condition to solve \eqref{eq:sdcoho} is that the average of the right hand-side needs to be zero.\jmm{\sl But, written as in (15), it already is.}
% \end{remark}
\begin{remark} \label{remark:diophantine}
Note that $1-\ee^{\bmi 2\pi  (k\omega+j\alpha)}$ can become arbitrarily small even if $\omega$ and $\alpha$ are non-resonant---this is the so called {\em small divisors problem}. For analytic $\eta$, the convergence of the series for $\xi$ is guaranteed by requiring stronger non-resonant conditions---that is, the Diophantine condition. This is standard in KAM theory
\end{remark}

 % setting.tex

\section{Flow map parameterization methods}
\label{sec:method}
In this section we develop the methodology for computing parameterizations of generating tori $\K$ and generating bundles $\W$ for $(d+\ell)$-dimensional partially hyperbolic invariant tori $\hat{\K}$ and their invariant bundles $\hat{\W}$, respectively, where $d=n-1$. 

\subsection{Adapted frames}
\label{sec:frm}
We proceed to construct symplectic frames $P:\T^{d-1}\times\T^\ell\to\R^{2n\times2n}$ in order to leverage the automatic reducibility of the tori. We look for a vector bundle map over the identity such that, in suitable coordinates, the linear dynamics reduce to an upper triangular matrix as
\begin{equation}
    \label{eq:reducibility}
    P(\bar{\theta},\bar{\varphi})^{-1}{\rm D}_z\phi_T\big(K(\theta,\varphi),\varphi\big)P(\theta,\varphi) =       \left(
         \begin{array}{c|c}
            \Lambda & S(\theta,\varphi) \\
            \hline
              O_n      & \raisebox{0pt}[2.2ex][0ex]{$\Lambda^{-\top}$}
         \end{array}
      \right),
\end{equation}
with
\begin{equation}\label{eq:defLmb}
   \Lambda= \begin{pmatrix} I_{n-1} & 0 \\ 0 & \lambda \end{pmatrix}
   ,\quad
   S(\theta,\varphi)=
   \begin{pmatrix} S^1(\theta,\varphi) & 0 \\ 0 & 0 \end{pmatrix},
\end{equation}
where each $0$ block corresponds to a zero matrix of suitable dimensions and $S^1$ is the $(n-1)\times(n-1)$ symmetric matrix known as the torsion matrix. 

The construction of the frame $P$ follows from first constructing a subframe $L:\T^{d-1}\times\T^{\ell}\to \R^{2n\times n}$ that is invariant under the differential of time-$T$ maps on $\K$; that is, $L$ is required to satisfy  
\begin{equation}
\label{eq:L_inv}
    {\rm D}_z\phi_T\big(K(\theta, \varphi),\varphi\big)L(\theta,\varphi)=L(\bar{\theta},\bar{\varphi})\Lambda.
\end{equation}
This is a necessary condition for the frame $P$ to satisfy \eqref{eq:reducibility}.

Note that according to \eqref{eq:invDK} and \eqref{eq:invW}, ${\rm D}_\theta K$ and $W$ are invariant under the differential of time-$T$ maps and they are therefore suitable to partly generate the subframe $L$. For autonomous Hamiltonian systems, the Hamiltonian vector field is invariant under the differential of time-$T$ maps and, consequently, this suffices to construct the subframe $L$, see \cite{HM21}. For periodic and quasi-periodic Hamiltonians, this is no longer the case due to the time dependency. We need to construct a new object invariant under ${\rm D}_z\phi_T$---this is the key element to apply the same ideas of flow map parameterization methods for autonomous systems to our setting.

Let us first derive Eq. \eqref{eq:inv_flowmap} with respect to $\varphi$ to obtain
\begin{equation}
    \label{eq:D_varphi}
    {\rm D}_{z}\phi_T\big(K(\theta,\varphi),\varphi\big){\rm  D}_\varphi K(\theta,\varphi) + {\rm D}_\varphi \phi_T\big(K(\theta,\varphi),\varphi\big) = {\rm D}_{\varphi} K(\bar{\theta},\bar{\varphi}).
\end{equation}
For the flow $\tilde\phi$, let us consider  
\begin{equation}
\label{eq:lie}
  \tfrac{d}{dt}\tilde\phi_T\circ\tilde\phi_t(z,\varphi)
\end{equation}
at $t=0$, from where we obtain the identity
\begin{equation}
\label{eq:transport}
    \begin{pmatrix}
            {\rm D}_z\phi_T(z,\varphi) & {\rm D}_\varphi\phi_T(z,\varphi) \\
            O_{\ell\times 2n} & I_\ell 
    \end{pmatrix}
    \begin{pmatrix}
            X_H(z,\varphi) \\ \hat\alpha
            \end{pmatrix} = 
            \begin{pmatrix}
            X_H\big(\phi_T(z,\varphi),\bar\varphi\big)\\ \hat\alpha
            \end{pmatrix}.
\end{equation} 
Then, using \eqref{eq:D_varphi} and evaluating \eqref{eq:transport} on the torus, i.e., at $z=K(\theta,\varphi)$, it is easy to show that
\begin{equation}
\label{eq:inv_zgeo}
{\rm D}_z\phi_T\big(K(\theta,\varphi),\varphi\big)\mathcal{X}_H(\theta,\varphi)=\mathcal{X}_H(\bar{\theta},\bar{\varphi}),
\end{equation}
where
\begin{equation}
\label{eq:defXH}
\mathcal{X}_H(\theta,\varphi):=X_H\big(K(\theta,\varphi),\varphi\big) - {\rm D}_\varphi  K(\theta,\varphi)\hat{\alpha}    
\end{equation}
is a geometric object defined on the torus that is invariant under ${\rm D}_z\phi_T$. From \eqref{eq:invW},\eqref{eq:invDK}, and \eqref{eq:inv_zgeo} it is immediate to see that 
\begin{equation}
\label{eq:L_def}
     L(\theta, \varphi):=\Big({\rm D}_{\theta}K(\theta,\varphi) \quad \mathcal{X}_H(\theta,\varphi) \quad W(\theta,\varphi) \Big)
\end{equation}
satisfies \eqref{eq:L_inv}. 

Since we require for $P$ to be a symplectic frame, this adds another layer of structure to the subframe $L$. More specifically, the subframe $L$ needs to be fiberwise Lagrangian. That is, for each $\varphi\in\T^\ell$, $L_\varphi(\theta):=L(\theta,\varphi)$ is required to generate a Lagrangian subspace on $T_{\mathcal{K}_\varphi} U$. We prove in Appendix \ref{ap:fiberwise_lag} that the subframe $L$, constructed as in \eqref{eq:L_def}, is a fiberwise Lagrangian subframe.

%Because of the formulation that we have adopted, from this point onwards we can apply the same ideas of flow map parameterization methods in autonomous systems presented in \cite{HM21} to our case.

We now proceed to complement the subframe $L$ with a subframe $N:\T^{d-1}\times\T^\ell\to \R^{2n\times n}$ such that the frame $P$ is symplectic with respect to the standard symplectic form. Note that the symplecticity of $P$ implies the subframe $N$ also needs to be fiberwise Lagrangian. There exists several ways to construct the subframe $N$, see e.g. \cite{book_alex} for details. We will use the almost complex structure and the Riemannian metric to construct 
\begin{align}\label{eq:defN}
   \hat N(\theta,\varphi)&:= J\bigl(K(\theta,\varphi)\bigr)   L(\theta,\varphi)G_K(\theta,\varphi)^{-1},
   \\
   G_K(\theta,\varphi)&:= L(\theta,\varphi)^\top G\big(K(\theta,\varphi)\big) L(\theta,\varphi)
   .
\end{align}
Then, it follows that the frame $\hat{P}(\theta,\varphi)=\left(L(\theta,\varphi)\vert \hat{N}(\theta,\varphi)\right)$ is symplectic and satisfies
\begin{equation}
\label{eq:first_reduction}
   \hat P(\bar{\theta},\bar{\varphi})^{-1} {\rm D}_z\phi_{T}\big(K(\theta,\varphi),\varphi\big)\hat P(\theta,\varphi)
   =
   \left(
   \begin{array}{c|c}
   \Lambda & \hat S(\theta,\varphi) \\
   \hline
   	  O_n        & \raisebox{0pt}[2.2ex][0pt]{$\Lambda^{-\top}$}
   \end{array}
   \right),
\end{equation}
where 
\begin{equation}\label{eq:defSi}
   \hat S(\theta,\varphi)
   =
   \hat N(\bar{\theta},\bar{\varphi})^\top
   \Omega\big(K(\bar{\theta},\bar{\varphi})\big)
   {\rm D}_z\phi_{T}\big(K(\theta,\varphi),\varphi\big)
   \hat N(\theta,\varphi),
\end{equation}
due to symplecticity, satisfies $\hat S(\theta,\varphi)\Lambda^\top = \Lambda \hat S(\theta,\varphi)^\top$.

Note that the frame $\hat{P}$ does not yet satisfy \eqref{eq:reducibility} as the torsion matrix given by $\hat{S}$ needs to be transformed to adopt the reduced form given in \eqref{eq:defLmb}.
In doing so, we get another invariant bundle generated by the last column of $P$. We construct a new symplectic frame by considering symplectic transformations
\begin{equation}\label{eq:defQ}
   Q(\theta,\varphi)=
   \left(
      \begin{array}{c|c}
         I_n & B(\theta,\varphi) \\
         \hline
          O_n   & I_n
      \end{array}
   \right)
      ,
 \end{equation}
with $B$ symmetric, such that the frame
\begin{equation}
\label{eq:defP}
    P(\theta,\varphi):=\hat{P}(\theta,\varphi)Q(\theta,\varphi)
\end{equation}
satisfies \eqref{eq:reducibility}. This translates into the matrix 
\begin{equation}\label{eq:S}
   S(\theta,\varphi)=
               \Lambda B(\theta,\varphi)
               +\hat S(\theta,\varphi)
               -B(\bar{\theta},\bar{\varphi})\Lambda^{-\top}
\end{equation}
adopting the required form which, in turn, determines the matrix $B$. Let us define splittings in blocks of sizes $(n-1)\times(n-1), (n-1)\times 1, 1\times(n-1)$ and $1\times1$ for $\hat{S}$ as 
\begin{equation}\label{eq:splittinghatS}
\hat S(\theta,\varphi)
   =
   \begin{pmatrix}
      \hat S^1(\theta,\varphi) & \hat S^2(\theta,\varphi) \\
      \hat S^3(\theta,\varphi) & \hat S^4(\theta,\varphi)
   \end{pmatrix},
\end{equation}
and define analogous splittings for the matrices $S$ and $B$. Then, expression \eqref{eq:S} reads 
\begin{align}
   S^1(\theta,\varphi) &= \hat S^1(\theta,\varphi) +B^1(\theta,\varphi)\phantom{\lambda}- B^1(\bar\theta,\bar\varphi) ,\label{eq:S1}\\
   S^2(\theta,\varphi) &= \hat S^2(\theta,\varphi)
                     + B^2(\theta,\varphi)\phantom{\lambda}-B^2(\bar \theta,\bar\varphi) \lambda^{-1},
                     \label{eq:S2}\\
   S^3(\theta,\varphi) &= \hat S^3(\theta,\varphi)
                        +B^3(\theta,\varphi)\lambda-B^3(\bar\theta, \bar\varphi),
                     \label{eq:S3}\\                
   S^4(\theta,\varphi) &= \hat S^4(\theta,\varphi)
                     + B^4(\theta,\varphi)\lambda-B^4(\bar\theta,\bar\varphi) \lambda^{-1}.
                     \label{eq:S4}
\end{align}
We require that $S^3(\theta,\varphi)^\top= S^2(\theta,\varphi)= 0$ and $S^4(\theta,\varphi)= 0$, whereas no restriction is applied to $S^1$. Consequently, our requirement on the frame $P$ translates into   
\begin{align}
   B^2(\theta,\varphi)\phantom{\lambda}-B^2(\bar\theta,\bar\varphi) \lambda^{-1}
      &= -\hat S^2(\theta,\varphi),
      \label{eq:b}
      \\
    B^4(\theta,\varphi)\lambda-B^4(\bar\theta,\bar\varphi)\lambda^{-1}
      &= -\hat S^4(\theta,\varphi),
      \label{eq:D}
\end{align}
and $B^3(\theta,\varphi)^\top=
B^2(\theta,\varphi)$. We then choose $B^1(\theta,\varphi)= I_{n-1}$, which results in $S^1(\theta,\varphi)= \hat S^1(\theta,\varphi)$. Equations \eqref{eq:b} and \eqref{eq:D} are non-small divisors cohomological equations that can be solved as detailed in Section \ref{sec:coho}. 

The construction of adapted frames is summarized with the following algorithm:
\begin{algo}\label{algo:frame}
Given $(K,W,\lambda)$ satisfying
\eqref{eq:inv_flowmap} and \eqref{eq:invW}, compute the adapted frame
$P$ and the reduced dynamics by following these steps:
\begin{enumerate}
\item Compute $\mathcal{X}_H(\theta,\varphi)$ from \eqref{eq:defXH}.
\item
   Compute $L(\theta,\varphi)$ from \eqref{eq:L_def}.
\item Compute 
   $\hat N(\theta,\varphi)$ from \eqref{eq:defN}.
\item\label{en:frnumi}
   Compute $\hat S(\theta,\varphi)$ from \eqref{eq:defSi}, let $B^1(\theta,\varphi)=I_{n-1}$, and $S^1(\theta,\varphi)=\hat S^1(\theta,\varphi)$.
\item
   Compute $B^2(\theta,\varphi)$ from  \eqref{eq:b} and let $B^3(\theta,\varphi)= B^2(\theta,\varphi)^\top$.
\item 
   Compute $B^4(\theta,\varphi)$ from \eqref{eq:D}. 
\item
   Compute $P(\theta,\varphi)$ from \eqref{eq:defP} with $Q(\theta,\varphi)$ given by \eqref{eq:defQ}.
\end{enumerate}
\end{algo}

\subsection{Description of a Newton step}
\label{sec:nw}
Given $(K,W,\lambda)$ that approximately satisfy equations \eqref{eq:inv_flowmap} and \eqref{eq:invW}, our aim is to improve such approximations with an iterative scheme. Let us define the error in the invariance of $\mathcal{K}$ and $\mathcal{W}$ as the functions $E^K,E^W:\T^{d-1}\times\T^\ell\to\R^{2n}$ given by 
\begin{gather}
   \label{eq:EK}
E^K(\theta,\varphi):=\phi_T\big(K(\theta,\varphi),\varphi\big)-K(\bar\theta,\bar\varphi),  \\
\label{eq:EW}
E^W(\theta,\varphi):= {\rm D}_z\phi_T\big(K(\theta,\varphi),\varphi\big)W(\theta,\varphi)-W(\bar\theta,\bar\varphi)\lambda.
\end{gather}
Since $\mathcal{K}$ and $\mathcal{W}$ are approximately invariant, $\mathcal{K}$ is approximately reducible. That is, there is an error in the reducibility of the linearized dynamics controlled by $E^K$ and $E^W$ as
\begin{equation*}
    P(\bar{\theta},\bar{\varphi})^{-1}{\rm D}_z\phi_T\big(K(\theta,\varphi),\varphi\big)P(\theta,\varphi) =
    \begin{pmatrix} 
            \Lambda & S(\theta,\varphi) \\
            O_n & \Lambda^{-\top}\\
    \end{pmatrix} + \mathcal{O}(\|E^K\|,\|E^W\|)
\end{equation*}
for some suitable norms. Also, the matrix $P$ is approximately symplectic. Therefore, instead of computing its inverse, we can use that
\[
P(\bar{\theta},\bar{\varphi})^{-1} =- \Omega_0 P(\bar{\theta},\bar{\varphi})^\top\Omega\big(K(\bar{\theta},\bar{\varphi})\big) +\mathcal{O}(\|E^K\|,\|E^W\|)
\]
to compute an approximate inverse. In order to improve the parameterizations for $\mathcal{K}$ and $\mathcal{W}$, we add corrections to their parameterizations such that the linearized invariance equations vanish at first order. Therefore, we can neglect the error in the reducibility of the linearized dynamics and in the inverse of $P$ as long as their contributions is of second order or higher.

%In {\em autonomous} Hamiltonian systems, it is common to include an equation to fix the value of the Hamiltonian on the torus. In the context of quasi-periodic Hamiltonian systems it could be possible to require some average value of the Hamiltonian of $\mathcal{K}$. We however choose not to do so due to the dependency of this average on the parameterization and the lack of dynamical meaning thereof. Additionally, freeing $T$ and $\omega$ in the Newton method could be considered but certain restrictions do not allow such considerations.
\subsubsection{A Newton step on the torus}
\label{sec:nwttr}
We proceed by adding a correction $\Delta K:\T^{d-1}\times\T^\ell\to\R^{2n}$ to the parameterization $K$. The invariance equation on the corrected torus then reads 
\begin{equation}
\label{eq:E_lin1}
\phi_T\big(K(\theta,\varphi)+\Delta K(\theta,\varphi),\varphi\big) - K(\bar\theta,\bar\varphi)- \Delta K(\bar\theta,\bar\varphi)=0.
\end{equation}
We linearize the equation around the approximate torus in order to find the correction $\Delta K$ that makes equation \eqref{eq:E_lin1} vanish at first order. Let us use the frame $P$ and write the correction in coordinates such that $\Delta K(\theta,\varphi)=P(\theta,\varphi)\xi(\theta,\varphi)$. Expanding around the approximated torus and retaining only terms up to first order results in
\begin{equation}
\label{eq:E_lin2}
{\rm D}_z\phi_T\big(K(\theta,\varphi),\varphi\big)P(\theta,\varphi)\xi(\theta,\varphi) - P(\bar\theta,\bar\varphi)\xi(\bar\theta,\bar\varphi)=-E^K(\theta,\varphi).
\end{equation}
We then left-multiply by $P(\bar\theta,\bar\varphi)^{-1}$ and neglect higher order terms to obtain
\begin{equation}
\label{eq:eq_nwK}
\begin{pmatrix} \Lambda & S(\theta,\varphi) \\ 0 & \Lambda^{-\top} \end{pmatrix}\xi(\theta,\varphi)- \xi(\bar\theta,\bar\varphi)=\eta(\theta,\varphi),
\end{equation}
where 
\begin{equation}
\label{eq:defeta}
\eta(\theta,\varphi)=\Omega_0 P(\bar\theta,\bar\varphi)^\top\Omega\big(K(\bar\theta,\bar\varphi)\big)E^K(\theta,\varphi).    
\end{equation}
Let us write $\xi$ and $\eta$ into  $(n-1)\times1\times(n-1)\times1$ components as
\begin{equation*}
   \xi(\theta,\varphi)=
      \begin{pmatrix}
         \xi^1(\theta,\varphi)\\ \xi^2(\theta,\varphi)\\ \xi^3(\theta,\varphi)\\ \xi^4(\theta,\varphi)
      \end{pmatrix}
   ,\quad
   \eta(\theta,\varphi)=
      \begin{pmatrix}
         \eta^1(\theta,\varphi)\\ \eta^2(\theta,\varphi)\\ \eta^3(\theta,\varphi)\\
         \eta^4(\theta,\varphi)
      \end{pmatrix}
   .
\end{equation*}
Observing that $\langle\eta^3\rangle$ is quadratically small---see Appendix \ref{ap:quadratiically small averages}---we neglect again higher order terms to write \eqref{eq:eq_nwK} as
\begin{align}
   \xi^1(\theta,\varphi)+S^1(\theta,\varphi)\xi^3(\theta,\varphi)
         -\xi^1(\bar\theta,\bar\varphi) &= \eta^1(\theta,\varphi),\label{eq:corrtr1} \\
   \lambda\xi^2(\theta,\varphi)-\xi^2(\bar\theta,\bar\varphi) 
   \phantom{-\xi^1(\bar\theta,\bar\varphi)}
      &= \eta^2(\theta,\varphi),\label{eq:corrtr2} \\
   \xi^3(\theta,\varphi)-\xi^3(\bar\theta,\bar\varphi)
    \phantom{-\xi^1(\bar\theta,\bar\varphi)}
      &= \eta^3(\theta,\varphi)-\langle\eta^3\rangle,\label{eq:corrtr3} \\
   \lambda^{-1}\xi^4(\theta,\varphi)-\xi^4(\bar\theta,\bar\varphi)
    \phantom{-\xi^1(\bar\theta,\bar\varphi)}
      &= \eta^4(\theta,\varphi),\label{eq:corrtr4}
\end{align}
Note that equations \eqref{eq:corrtr2} and \eqref{eq:corrtr4} are non-small divisors cohomological equations whereas equation \eqref{eq:corrtr3} is a small divisors cohomological equation. Once $\xi^3$ is solved for, equation \eqref{eq:corrtr1} is also a small divisors cohomological equations. Since \eqref{eq:corrtr3} is solvable (assuming Diophantine conditions), with $\langle\xi^3\rangle$ free, we adjust its value in order to adjust averages in \eqref{eq:corrtr1}. That is, we will use this freedom to solve equation \eqref{eq:corrtr1} as a small divisors cohomological equation. See Section \ref{sec:coho}. Let us consider
\[
\xi^3(\theta,\varphi)=\xi^3_0 + \tilde{\xi}^3(\theta,\varphi),
\]
where $\tilde\xi^3$, solves \eqref{eq:corrtr3} with $\langle\tilde\xi^3\rangle=0$. Then, equation \eqref{eq:corrtr1} becomes
\begin{equation*}
%\label{eq:corrtr11}
\xi^1(\theta,\varphi)-\xi^1(\bar\theta,\bar\varphi)=\eta^1(\theta,\varphi)-S^1(\theta,\varphi)\tilde\xi^3(\theta,\varphi)-S^1(\theta,\varphi)\xi^3_0.
\end{equation*}
We now choose $\xi^3_0$ such that
\begin{equation}
\label{eq:xi3sys}
\langle S^1 \rangle\xi^3_0 = \langle \eta^1 - S^1\tilde\xi^3 \rangle.
\end{equation}  
Hence, we can now solve equation \eqref{eq:corrtr1} as a small divisors cohomological equation with $\langle \xi^1 \rangle$ free; a simple choice is to take $\langle \xi^1 \rangle=0$. This underdeterminacy reflects the underdeterminacy of the parameterization of the generating torus $\mathcal{K}$ under phase shifts in $\T^{d-1}$ and translations within $\hat{\mathcal K}$, see Remark \ref{remark:underK}.

The Newton step on the torus is summarized with the following algorithm:
\begin{algo}\label{algo:nwttr}
Let $(K,W,\lambda)$ satisfy equations
\eqref{eq:inv_flowmap} and \eqref{eq:invW} approximately. Obtain the corrected generating
torus by following these steps:
\begin{enumerate}
\item
   Compute $P(\theta,\varphi)$ and $S^1(\theta,\varphi)$ by following
   Algorithm \ref{algo:frame}.
\item
   Compute the error $E^K(\theta,\varphi)$ from \eqref{eq:EK}.
\item
   Compute $\eta(\theta,\varphi)$, the  right-hand side of the cohomological equations given in \eqref{eq:eq_nwK}, from \eqref{eq:defeta}.
\item
   Solve \eqref{eq:corrtr2} and \eqref{eq:corrtr4} as non-small divisors
   cohomological equations in order to obtain $\xi^2(\theta,\varphi)$ and
   $\xi^4(\theta,\varphi)$.
\item
   Solve \eqref{eq:corrtr3} as small divisors cohomological equations
   in order to obtain its zero-average solution $\tilde\xi^3(\theta,\varphi)$.
\item
   Compute $\langle S^1\rangle$
   and the right-hand side of the linear system \eqref{eq:xi3sys}
    and solve it in order to obtain $\xi^3_0$.
\item
   Solve \eqref{eq:corrtr1} as small divisors cohomological equations in order to  obtain $\xi^1(\theta,\varphi)$ with $\langle \xi^1\rangle=0$.
\item
   Set $K(\theta,\varphi)\leftarrow K(\theta,\varphi) +P(\theta,\varphi)\xi(\theta,\varphi)$. 
\end{enumerate}
\end{algo}

\subsubsection{A Newton step on the bundle}
\label{sec:nwtbd}
Once we have corrected the parameterization of the generating torus in the previous step, we proceed in an analogous manner and add corrections $\Delta W:\T^{d-1}\times\T^\ell\to\R^{2n}$ and $\Delta \lambda\in\R$ to the parameterization of the generating bundle and to the Floquet multiplier, respectively. Then, on the corrected bundle, we obtain from  \eqref{eq:invW} 
\begin{equation*}
\begin{split}
{\rm D}_z\phi_T\big(K(\theta,\varphi),\varphi \big)\big(W(\theta,\varphi) &+ \Delta W(\theta,\varphi)\big) \\
 	&-\left(W(\bar\theta,\bar\varphi) + \Delta 				W(\bar\theta,\bar\varphi)\right)(\lambda + \Delta\lambda)=0.
\end{split}
\end{equation*}
We choose the corrections such that the previous equation vanishes at first order. Again, we write the correction term for the bundle in coordinates such that $\Delta W(\theta,\varphi)=P(\theta,\varphi)\xi(\theta,\varphi)$ and expand the previous expression retaining terms up to first order to obtain 
\begin{equation*}
\begin{split}
{\rm D}_z\phi_T\bigl(K(\theta,\varphi),\varphi\bigr)P(\theta,\varphi)\xi(\theta,\varphi) &- P(\bar\theta,\bar\varphi)\xi(\bar\theta,\bar\varphi)\lambda \\
 &- W(\bar\theta,\bar\varphi)\Delta \lambda = -E^W(\theta,\varphi).
\end{split}
\end{equation*}
We then left-multiply by $P(\bar\theta,\bar\varphi)^{-1}$, use equation \eqref{eq:reducibility}, and neglect second order terms to obtain 
\begin{equation}
\label{eq:eq_nwW}
\begin{pmatrix} \Lambda & S(\theta,\varphi) \\ O_n & \Lambda^{-\top} \end{pmatrix}\xi(\theta,\varphi)- \lambda\xi(\bar\theta,\bar\varphi)-e_n\Delta \lambda=\eta(\theta,\varphi),\quad 
\end{equation}  
with 
\[
   e_n = \begin{pmatrix} 0_{n-1}\\ 1 \\ 0_{n-1} \\ 0 \end{pmatrix}
   ,
\]
and 
\begin{equation}
\label{eq:defetabd}
    \eta(\theta,\varphi)=\Omega_0 P(\bar\theta,\bar\varphi)^\top\Omega\big(K(\bar\theta,\bar\varphi)\big)E^W(\theta,\varphi).
\end{equation}

We can split $\xi$ and $\eta$ as in the previous section and rewrite \eqref{eq:eq_nwW} as
\begin{align}
   \xi^1(\theta,\varphi)+S^1(\theta,\varphi)\xi^3(\theta,\varphi)
         -\lambda\xi^1(\bar\theta,\bar\varphi)
         \phantom{\ \,-\deltalambda}
      &= \eta^1(\theta,\varphi),\label{eq:corrbd1} \\
   \lambda\xi^2(\theta,\varphi)-\lambda\xi^2(\bar\theta,\bar\varphi)
   -\deltalambda
      &= \eta^2(\theta,\varphi),\label{eq:corrbd2} \\
   \xi^3(\theta,\varphi)-\lambda\xi^3(\bar\theta,\bar\varphi)
      \phantom{\ \,-\deltalambda}
      &= \eta^3(\theta,\varphi),\label{eq:corrbd3} \\
   \lambda^{-1}\xi^4(\theta,\varphi)-\lambda\xi^4(\bar\theta,\bar\varphi)
      \phantom{\ \,-\deltalambda}
      &= \eta^4(\theta,\varphi).\label{eq:corrbd4}
\end{align}
Equations \eqref{eq:corrbd3} and \eqref{eq:corrbd4} can be solved as non-small divisors cohomological equations and, once $\xi^3$ is known, equation \eqref{eq:corrbd1} can also be solved as a non-small divisors cohomological equation. On the contrary, equation \eqref{eq:corrbd2} can be solved as a small divisors cohomological equation with $\langle \xi^2 \rangle$ free once we adjust the average of the right hand side by taking $\Delta \lambda= -\langle \eta^2\rangle$. The freedom in $\langle \xi^2 \rangle$ is related to the freedom in the length of $W$, analogous to the underdeterminacy of the lengths of eigenvectors in an eigenvalue problem. We then take the simplest choice for this average, i.e., $\langle \xi^2 \rangle = 0$.

The Newton step on the bundle is summarized with the following algorithm:
\begin{algo}\label{algo:nwtbd}
Let $(K,W,\lambda)$ satisfy equations \eqref{eq:inv_flowmap} and \eqref{eq:invW} approximately. Obtain the corrected generating
bundle and Floquet multiplier by following these steps:
\begin{enumerate}
\item
   Compute $P(\theta,\varphi)$ and $S^1(\theta,\varphi)$ by following
   Algorithm \ref{algo:frame}.
\item
   Compute the error $E^W(\theta,\varphi)$ from \eqref{eq:EW}.
\item
   Compute $\eta(\theta,\varphi)$, the right-hand side of the cohomological equations given in \eqref{eq:eq_nwW}, from \eqref{eq:defetabd}.
\item
   Solve \eqref{eq:corrbd3}, \eqref{eq:corrbd4}, and \eqref{eq:corrbd1}
    as non-small divisors cohomological equations in order to obtain $\xi^3(\theta,\varphi)$, $\xi^4(\theta,\varphi)$, and $\xi^1(\theta,\varphi)$,
   respectively.
\item
   Take $\Delta\lambda=-\langle\eta^2\rangle$.
\item
   Solve \eqref{eq:corrbd2} as a small divisors cohomological equations in order to obtain $\xi^2(\theta,\varphi)$ with $\langle\xi^2\rangle=0$,
\item
   Set $W(\theta,\varphi)\leftarrow W(\theta,\varphi)+P(\theta,\varphi)\xi(\theta,\varphi)$ and $\lambda\leftarrow\lambda+\Delta\lambda$.
\end{enumerate}
\end{algo}

\subsection{Continuation with respect to external parameters}
\label{sec:cnte}
Let us assume we have a family of Hamiltonians $H_\varepsilon$ that depends analytically on some parameter $\varepsilon\in\R$. Given $(K,W,\lambda)_\varepsilon$ for certain value of $\varepsilon$ that satisfies equations \eqref{eq:inv_flowmap} and \eqref{eq:invW}, we want to compute parameterizations of a generating torus and generating bundle (with the corresponding Floquet multiplier) for a different Hamiltonian $H_{\varepsilon'}$. As commonly done in continuation methods, see e.g. \cite{allgower-georg}, we will provide a methodology to compute the tangent to the continuation curve with respect to $\varepsilon$ from where we obtain a first order approximation of the invariant objects for $H_{\varepsilon'}$. 

Let us assume equations \eqref{eq:inv_flowmap} and \eqref{eq:invW} define implicitly $(K,W,\lambda)_\varepsilon$ as functions of $\varepsilon$. Then, we want to compute $\partial_\varepsilon K, \partial_\varepsilon W$, and $\partial_\varepsilon \lambda$. In the following, for the sake of notation clarity, we will omit the dependency of $(K,W,\lambda)$ and of $\phi_T$ on $\varepsilon.$ Let us begin by taking Eq. \eqref{eq:inv_flowmap} and differentiate it with respect to $\varepsilon$ to obtain
\[
{\rm D}_z \phi_T\big(K(\theta,\varphi),\varphi\big)\partial_\varepsilon K(\theta,\varphi) - \partial_\varepsilon K(\bar\theta,\bar\varphi) + \partial_\varepsilon \phi_T\big(K(\theta,\varphi),\varphi\big)=0,
\]
where $\partial_\varepsilon \phi_T\bigl(z,\varphi\bigr)$ is the variation of the map $\phi_T$ with respect to $\varepsilon$ that can be computed through variational equations. We now express the derivatives in the frame such that $\partial_\varepsilon K(\theta,\varphi)= P(\theta,\varphi)\xi(\theta,\varphi)$. The previous expression then reads
\begin{equation*}
\begin{split}
{\rm D}_z\phi_T\big(K(\theta,\varphi),\varphi\big)P(\theta,\varphi)\xi(\theta,\varphi) &- P(\bar\theta,\bar\varphi)\xi(\bar\theta,\bar\varphi) \\
 &+ \partial_\varepsilon \phi_T\big(K(\theta,\varphi),\varphi\big)=0
\end{split}
\end{equation*}
and, after left-multiplication by $P(\bar\theta,\bar\varphi)^{-1}$, we have
\begin{equation}
\label{eq:syspartialK}
\begin{pmatrix} \Lambda & S(\theta,\varphi) \\ 0 & \Lambda^{-\top} \end{pmatrix}\xi(\theta,\varphi)- \xi(\bar\theta,\bar\varphi)=\eta(\theta,\varphi),
\end{equation}
where
\begin{equation}
\label{eq:defetacontK}
\eta(\theta,\varphi) = \Omega_0 P(\bar\theta,\bar\varphi)^\top\Omega\big(K(\bar\theta,\bar\varphi)\big)\partial_\varepsilon \phi_T\big(K(\theta,\varphi),\varphi\big).    
\end{equation}
The cohomological equations given in \eqref{eq:syspartialK} can be solved exactly as described in Section \ref{sec:nwttr} but with $\eta$ given by \eqref{eq:defetacontK}.

\begin{remark}
For the system given by \eqref{eq:syspartialK} to be solvable, we encounter again small divisors cohomological equations for $\xi^3$ (recall the splittings defined in Section \ref{sec:nwttr}) that require $\langle\eta^3\rangle=0$. Using fiberwise symplectic deformations we prove in Appendix \ref{ap:zero_average} that, for $\eta$ defined as in \eqref{eq:defetacontK}, $\langle\eta^3\rangle=0$.

\end{remark}

In order to obtain $\partial_\varepsilon W$ and $\partial_\varepsilon \lambda$, we differentiate equation \eqref{eq:invW} with respect to $\varepsilon$ to obtain
\begin{equation*}
\begin{split}
\partial_\varepsilon\Big( {\rm D}_z \phi_T\bigl(K(\theta,\varphi),\varphi\bigr)\Big) W(\theta,\varphi) &+ {\rm D}_z \phi_T\bigl(K(\theta,\varphi),\varphi\bigr)\partial_\varepsilon W(\theta,\varphi)\\
 &- W(\bar\theta,\bar\varphi)\partial_\varepsilon \lambda - \partial_\varepsilon W(\bar\theta,\bar\varphi)\lambda = 0,
\end{split}
\end{equation*} 
which can be rewritten as
\begin{equation}
\label{eq:DeW}
\begin{split}
{\rm D}_z \phi_T\bigl(K(\theta,\varphi),\varphi\bigr)\partial_\varepsilon W(\theta,\varphi)&-\partial_\varepsilon W(\bar\theta,\bar\varphi)\lambda\\
&- W(\bar\theta,\bar\varphi)\partial_\varepsilon \lambda = -E^{\partial_\varepsilon W}(\theta,\varphi),
\end{split}
\end{equation}
with
\begin{equation}
\begin{split}
\label{eq:EpartialW}
E^{\partial_\varepsilon W}(\theta,\varphi) &= \partial_\varepsilon\Big( {\rm D}_z \phi_T\bigl(K(\theta,\varphi),\varphi\bigr)\Big) W(\theta,\varphi) \\
&=\partial_\varepsilon{\rm D}_{z}\phi_T\bigl(K(\theta,\varphi),\varphi\bigr)W(\theta,\varphi)\\
 &\phantom{=} + {\rm D}_z^2\phi_T\bigl(K(\theta,\varphi),\varphi\bigr)[W(\theta,\varphi),\partial_\varepsilon K(\theta,\varphi)]. 
\end{split}
\end{equation}
The form ${\rm D}_z^2\phi_T\bigl(z,\varphi\bigr)[\cdot,\cdot]$ is the bilinear form given by the second differential of $\phi_T$ with respect to $z$ that can be computed through variational equations. Then, by using the frame such that $\partial_\varepsilon W(\theta,\varphi)=P(\theta,\varphi)\xi(\theta,\varphi)$, and after left-multiplication by $P(\bar\theta,\bar\varphi)^{-1}$, expression \eqref{eq:DeW} reads
\begin{equation}\label{eq:DeWcoho}
\begin{pmatrix} \Lambda & S(\theta,\varphi) \\ 0 & \Lambda^{-\top} \end{pmatrix}\xi(\theta,\varphi)- \lambda\xi(\bar\theta,\bar\varphi)-e_n\partial_\varepsilon \lambda=\eta(\theta,\varphi),
\end{equation} 
with
\begin{equation}
\label{eq:defetapartialW}
\eta(\theta,\varphi)=\Omega_0 P(\bar\theta,\bar\varphi)^\top\Omega\big(K(\bar\theta,\bar\varphi)\big)E^{\partial_\varepsilon W}(\theta,\varphi).    
\end{equation}
These cohomological equations can be solved exactly as described in Section \ref{sec:nwtbd} but with $\eta$ given by \eqref{eq:defetapartialW}.

\begin{algo}\label{algo:cnte}
Let $(K,W,\lambda)_\varepsilon$ be implicit functions of $\varepsilon$ by
equations \eqref{eq:inv_flowmap} and \eqref{eq:invW}. Find
$\partial_\varepsilon K$, $\partial_\varepsilon W$, and $\partial_\varepsilon\lambda$ by following these steps:
\begin{enumerate}
\item
   Compute $P(\theta,\varphi)$ and $S^1(\theta,\varphi)$ by following
   Algorithm \ref{algo:frame}.
\item
  Compute $\eta(\theta,\varphi)$, the right-hand side of the cohomological equations given in \eqref{eq:syspartialK}, from \eqref{eq:defetacontK}.
  \item
   Solve \eqref{eq:corrtr2} and \eqref{eq:corrtr4} as non-small divisors
   cohomological equations in order to obtain $\xi^2(\theta,\varphi)$ and
   $\xi^4(\theta,\varphi)$.
\item
   Solve \eqref{eq:corrtr3} as small divisors cohomological equations
   in order to obtain its zero-average solution $\tilde\xi^3(\theta,\varphi)$.
\item
   Compute $\langle S^1\rangle$,
   and the right-hand side of the linear system, \eqref{eq:xi3sys}
    and solve it in order to obtain $\xi^3_0$.
\item
   Solve \eqref{eq:corrtr1} as small divisors cohomological equations in order to  obtain $\xi^1(\theta,\varphi)$ with $\langle \xi^1\rangle=0$.
\item
   Set $\partial_\varepsilon K(\theta,\varphi)\leftarrow P(\theta,\varphi)\xi(\theta,\varphi)$.
\item\label{en:cntTnumi}
   Compute $E^{\partial_\varepsilon W}(\theta,\varphi)$ from \eqref{eq:EpartialW}.
\item
   Compute $\eta(\theta,\varphi)$, the right-hand side of the cohomological equations given in \eqref{eq:DeWcoho}, from \eqref{eq:defetapartialW}.
\item
   Solve \eqref{eq:corrbd3}, \eqref{eq:corrbd4}, and \eqref{eq:corrbd1} as non-small divisors cohomological equations in order to obtain
   $\xi^3(\theta,\varphi)$, $\xi^4(\theta,\varphi)$, and $\xi^1(\theta,\varphi)$, 
   respectively. 
\item
   Take $\partial_\varepsilon\lambda=-\langle\eta^2\rangle$.
\item
    Solve 
    \[
    \lambda \xi^2(\theta,\varphi)-\lambda\xi^2(\theta,\varphi) - \partial_\varepsilon\lambda = \eta^2(\theta,\varphi)
    \]
    as a small divisors cohomological equation with $\langle\xi^2\rangle=0$
\item
   Set $\partial_\varepsilon W(\theta,\varphi)\leftarrow P(\theta,\varphi)\xi(\theta,\varphi)$.
\end{enumerate}
\end{algo}

\subsection{Comments on implementations}
\label{sec:compimpl}
 For every function $\zeta:\T^{d-1}\times\T^\ell\to\R$, we use its Fourier representation and their grid representation. The Fourier representation is given in terms of Fourier coefficients $\{\hat\zeta_{kj}\}\in\C$, with $k\in\Z^{d-1}$ and $j\in\Z^\ell$, and the grid representation is given by the values $\{\zeta_{kj}\}$ of $\zeta$ in an equally spaced grid of $\T^{d-1}\times\T^\ell$. We can switch between both representations with the linear one-to-one map provided by the Discrete Fourier Transform (DFT). Operations such as phase shifts, differentiation, or solving cohomological equations can be done efficiently in Fourier space whereas numerical integration or evaluation of vector fields can be done more efficiently in the grid representation. We switch between both representations according to the operations to be performed. See e.g. \cite{HM21} for details.
 
In practice, we can only work with truncated series. Furthermore, the DFT only provides approximate coefficients, that cannot in general be directly identified with the Fourier coefficients. Instead, for each coefficient $\hat\zeta_{kj}$, we use the DFT coefficient that best approximates it, see e.g. \cite{Henrici79} for details. 

The unstable Floquet multiplier of $\mathcal{K}$ can be large, which would compromise the convergence of the method. In the numerical explorations of Section \ref{sec:application}, instead of solving Eqs. \eqref{eq:inv_flowmap} and \eqref{eq:invW}, we follow a multiple shooting approach. We consider multiple tori $\left\{\mathcal{K}_i\right\}_{i=0}^{m-1}$ and bundles $\left\{\mathcal{W}_i\right\}_{i=0}^{m-1}$ parameterized by $\left\{K_i\right\}_{i=0}^{m-1}$ and $\left\{W_i\right\}_{i=0}^{m-1}$ with $K_i:\T^{d-1}\times\T^\ell\to U$ and $W_i:\T^{d-1}\times\T^\ell\to \R^{2n}$ such that
\begin{align}
\label{eq:multi1}
 \phi_{T/m}\big(K_i(\theta,\varphi),\varphi\big) &- K_{i+1}(\bar\theta_m,\bar\varphi_m) = 0,
 \\
 \label{eq:multi2}
 {\rm D}_z \phi_{T/m}\big(K_i(\theta,\varphi),\varphi\big)W_i(\theta,\varphi) &- \lambda^{\frac{1}{m}} W_{i+1}(\bar\theta_m,\bar\varphi_m) = 0
\end{align}
for $i=1,...,m-1$, where $\bar\theta_m:=\theta + \omega/m$ and $\bar\varphi_m:=\varphi + \alpha/m$. The subindex in $K$ and $W$ is defined modulo $m$. We choose $m$ such that $|\lambda|^{\frac{1}{m}}$ is small enough. The methodology described in Section \ref{sec:method} generalizes with some extra work for Eqs. \eqref{eq:multi1} and \eqref{eq:multi2}. For the sake of clarity, we described the method for $m=1$, see \cite{HM21} for the details in the autonomous case. 

The evaluation of flow maps and the variational equations require numerical integration that can be costly for high-dimensional tori. Nonetheless, numerical integration is easily parallelizable by assigning trajectories to different threads.

To prevent numerical instabilities, we implement a lowpass filter for the approximate Fourier series of $K$ and $W$ after each iterate. For simplicity, assume $\mathcal{K}$ is a 2-dimensional torus with a grid of size $N_1\times N_2$---which is the test case of Section \ref{sec:application}. The filtering strategy consists of setting to zero the coefficients for $|k|> k_f$ and $|j|>j_f$ for the values $k_f = r_f\cdot N_1$ and $j_f = r_f\cdot N_2$, where $r_f\in\left[\frac{1}{4},\frac{1}{2}\right)$ is some filtering factor. Because of this filtering strategy, the effective number of approximate Fourier coefficients is not $N_1\times N_2$. 

Since we are working with truncated Fourier series, we need to decide on the number of Fourier coefficients. This number needs to be large enough so the series allow the parameterizations to meet their required error tolerance in the invariance equations. The necessary number of coefficients might change throughout the continuations so we also need a strategy to adjust the grid sizes of the invariant objects. Since our objects are real analytic, their Fourier coefficients decay exponentially fast according to Remark \ref{remark:decay}. Our strategy is based on controlling this decay. In order to do so, we compute the tails of the truncated series. For functions $\zeta:\T\times\T\to\R$, we compute the tails in the internal phase $\theta$ and the external phase $\varphi$ as
\begin{equation*}
    t_\theta(\zeta) = \sum_{\substack{|k|>k_t\\|j|<j_t}}|\hat\zeta_{kj}| \qquad 
   t_\varphi(\zeta) = \sum_{\substack{|k|<k_t\\|j|>j_t}}|\hat\zeta_{kj}|,
\end{equation*}
for $k_t = r_t \cdot N_1$, $ j_t = r_t\cdot N_2$, where $r_t$ is some tail factor. For the case where $\zeta$ is a multidimensional function, we consider $t_\theta$ and $t_\varphi$ to act component-wise.

Furthermore, throughout the continuations on $\varepsilon$, the necessary step size of the continuation procedure needs to be adjusted. We conclude this section with a proposal in Algorithm \ref{algo:implementation} for the continuation of generating tori, bundles, and Floquet multipliers that is an adaptation for our case of Algorithm $3.6.1$ in \cite{HM21}.
The main difference is that in the algorithm of \cite{HM21}, a strategy to reduce the number of Fourier coefficients in the parameterizations is necessary whereas in Algorithm \ref{algo:implementation}, we require a strategy to control the decay of the coefficients in the different phases of $K$ and $W$ to adjust the grid sizes of the parameterizations.

Let us represent the parameterizations of the invariant objects and the Floquet multiplier for some $\varepsilon$ as
\[
\mathcal{T}_\varepsilon=(\varepsilon,K,W,\lambda)
\]
and let us also consider the following norm for $\mathcal{T}_\varepsilon$ that we will use for the step size control
\[
\|\mathcal{T}_\varepsilon\|:=\left(\varepsilon^2+ \lambda^2 + \langle\|K\|^2_2\rangle + \langle\|W\|^2_2\rangle\right)^{\frac{1}{2}},
\]
where $\|K\|_2$ stands for the map $(\theta,\varphi)\mapsto\|K(\theta,\varphi)\|_2$ and $\|W\|_2$ is the analogous map.

The error in the torus and bundle are estimated as follows
\begin{equation*}
\begin{split}
   \err^K(\mathcal{T_\varepsilon})&=\max_{\substack{
   0\leq k<N_1\\
   0\leq j<N_2
   }}\|E^K(k/N_1,j/N_2)\|_\infty, \\
   \err^W(\mathcal{T}_\varepsilon)&=\max_{\substack{
   0\leq k<N_1\\
   0\leq j<N_2
   }}\|E^W(k/N_1,j/N_2)\|_\infty.
   \end{split}
\end{equation*}

\begin{algo}\label{algo:implementation}
Let $\mathcal{T}_\varepsilon$ be an approximately invariant torus, bundle, and Floquet multiplier for some value of the external parameter $\varepsilon$ and some values of the flying time $T$ and rotation vectors $\omega$ and $\alpha$. Let $K$ and $W$ have grid representations of size $N_1\times N_2$, let $\epsilon_K,\epsilon_W,\epsilon_t$ be tolerances, $r_t$ some tail factor, and $n_{max}, n_\varepsilon,n_{des},n_t$ be integers. Assume we have a suggested continuation step $\Delta\varepsilon$. Compute a new torus, bundle, and Floquet multiplier $\mathcal{T}_{\varepsilon'}$ for a new value of the external parameter $\varepsilon'$ and fixed $T$, $\omega$, and $\alpha$ as follows:
\begin{enumerate}
    \item \label{en:cnt:start}Compute $\partial_\varepsilon K, \partial_\varepsilon W$, and $\partial_\varepsilon \lambda$ following algorithm \ref{algo:cnte} and set the continuation direction $\delta \leftarrow (1, \partial_\varepsilon K, \partial_\varepsilon W, \partial_\varepsilon\lambda)$.
    
    \item\label{en:cnt:lbl2} Set $\Delta\mathcal{T}_\varepsilon\leftarrow\delta/\|\delta\|$ and ${\mathcal{T}}_{\varepsilon'}\leftarrow\mathcal{T}_\varepsilon+\Delta\varepsilon \Delta\mathcal{T}_\varepsilon$.
    \item Perform Newton steps on $\mathcal{T}_{\varepsilon'}$ by following algorithms \ref{algo:nwttr} and \ref{algo:nwtbd} until ${\rm err}^K(\mathcal{T}_{\varepsilon'})<\epsilon_K$ and ${\rm err}^W(\mathcal{T}_{\varepsilon'})<\epsilon_W$ or up to $n_{max}$ times.
      \item If ${\rm err}^K(\mathcal{T}_{\varepsilon'})<\epsilon_K$ and ${\rm err}^W(\mathcal{T}_{\varepsilon'})<\epsilon_W$, let $n_{it}$ be the number of Newton iterations and go to step \ref{en:cnt:success}.
    \item Go to step \ref{en:cnt:lbl2} with $\Delta\varepsilon \leftarrow \frac{1}{2}\Delta\varepsilon$ up to $n_{\varepsilon}$ times. 
    
    \item Compute $\|t_\theta(K)\|_{\infty}$ and $\|t_\varphi(K)\|_{\infty}$. 
    
    \item If $\|t_\theta(K)\|_{\infty}>\epsilon_t$ and $\|t_\varphi(K)\|_{\infty}>\epsilon_t$, set $N_1\leftarrow 2N_1$ and $N_2\leftarrow 2 N_2$. Otherwise, set $N_1\leftarrow 2N_1$ if $\|t_\theta(K)\|_{\infty}>\|t_\varphi(K)\|_{\infty}$ and $N_2\leftarrow 2N_2$ if $ \|t_\varphi(K)\|_{\infty}>\|t_\theta(K)\|_{\infty}$. Go to step \ref{en:cnt:start}; try up to $n_t$ times.
    \item \label{en:cnt:success} Set $\mathcal{T}_\varepsilon\leftarrow\mathcal{T}_{\varepsilon'}, \Delta\varepsilon \leftarrow \frac{n_{des}}{n_{it}} \Delta\varepsilon$ and go to step \ref{en:cnt:start}.
\end{enumerate}
\end{algo}

 % method.tex

\section{An application: the elliptic restricted three body problem}
\label{sec:application}

In this section, we apply our method to the periodic Hamiltonian system given by the Elliptic Restricted Three Body Problem (ERTBP). The strategy we follow consists of taking families of 2D partially hyperbolic invariant tori in the Circular Restricted Three Body Problem (CRTBP) and lift them to the elliptic problem as 3D partially hyperbolic invariant tori through continuation in the eccentricity as described in Section \ref{sec:method}. In practice, we compute parameterizations of 2D generating tori together with their stable, unstable, and center generating bundles.

For the numerical explorations, we have used a Fujitsu Siemens CELSIUS R930N workstation with two 12-core Intel Xeon E5-2630v2 at 2.60GHz running Debian GNU/Linux 11. The algorithms were written in C, compiled with GCC 10.2.1 and linked against Glibc 2.31, LAPACK 3.9.0, and FFTW 3.3.8. The numerical integration was parallelized using OpenMP 4.5.

\subsection{Dynamics}
The Elliptic Restricted Three Body Problem models the motion of a small body of negligible mass in the gravitational vector field generated by two other massive bodies, known as primaries, moving in elliptical orbits according to two-body dynamics. 

Let us denote by $m_1$ and $m_2$ the masses of the primary bodies and let us also define the parameter $\mu:=\frac{m_2}{m_1+m_2}$. It is possible to define a rotating and pulsating frame after a suitable rescaling in space and time where the primaries of mass $m_1$ and $m_2$ are at $(\mu,0,0)$ and $(\mu-1,0,0)$, respectively, and their period of revolution is $2\pi$, see \cite{szebehely} for details. The dynamics for the third body are then given by the periodic Hamiltonian
\begin{equation*}
\begin{split}
H(x,p,\varphi)=\frac{1}{2}[(p_1 & +x_2)^2 + (p_2-q_1)^2 + p_3^2 + x_3^2] \\ 
& - \frac{1}{1+e \cos f}\left[\frac{1}{2}(x_1^2+x_2^2+x_3^2) + \frac{1-\mu}{r_1}+\frac{\mu}{r_2}  \right],\end{split}
\end{equation*}
with $x\in\R^3$ positions, $p\in\R^3$ momenta, $r_1^2=(x_1-\mu)^2+x_2^2 + x_3^2, r_2^2=(x_1+1-\mu)^2+x_2^2+x_3^2$, and $e\in\R$ the eccentricity of the elliptic orbit of the primaries. The true anomaly $f:=2\pi\varphi$, with $\varphi\in\T$, is the angle that parameterizes the orbit of the primaries and moves according to the frequency $\dot\varphi=\hat\alpha=1/2\pi$.

The ERTBP has five fixed points, known as the Lagrange points, denoted by $L_i$ with $i={1,2,...,5}$. The coordinates of these equilibrium points coincide with the coordinates of the fixed points in the circular problem. The collinear solutions $L_1, L_2,$ and $L_3$ are unstable for any combination of $\mu$ and $e$ whereas the stability of the triangular points, $L_4$ and $L_5$, depends on these two parameters. For their values in the Sun-Earth system, that is $\mu=3.040357143\cdot 10^{-6}$ and $e=0.01671123$, the triangular Lagrange points are linearly stable (see e.g. \cite{szebehely} for details). 

The persistence of the fixed points in the elliptic problem does not generalize to other invariant objects. Periodic solutions only exist in resonance with the frequency of the primaries.
Furthermore, the frequency of the primaries, i.e. the external frequency $\hat\alpha$, is added to invariant objects such as periodic orbits and invariant tori of the CRTBP increasing the dimension by one of their counterparts in the elliptic case. That way, (non-resonant) periodic orbits survive as 2D tori while the classical Lissajous and quasi-halo orbits become 3D tori---this will be our case study.

\subsection{From the CRTBP to the ERTBP}
As we mentioned previously, we will lift invariant tori and bundles from the CRTBP to the ERTBP. In this section, we include some details on the families of tori that we computed in the circular problem: tori in the center manifold of the $L_1$ point.

The $L_1$ point in the CRTBP is of type center$\times$center$\times$saddle; that is
\[
   {\rm Spec} {\rm D}X_H(L_1) = \{
      \bmi 2\pi\hat\omega_p^0, -\bmi2\pi\hat\omega_p^0,
      \bmi 2\pi\hat\omega_v^0, -\bmi2\pi\hat\omega_v^0,
      \lambda^0, -\lambda^0
   \},
\]
with a value of the Hamiltonian $h^0$.
Thus, there is a 4-dimensional center manifold generated by the central part of $L_1$ that contains invariant objects.

The Lyapunov center theorem (see e.g. \cite{ meyer-hall-offin, siegel-moser}) ensures that there exists two 2-dimensional manifolds inside the center manifold filled with families of periodic orbits: the planar and the vertical Lyapunov families. These families present bifurcations that give birth to new families such as the halo family and other more exotic orbits (see e.g. \cite{2003aDiDoPa}). We can parameterize the orbits in these families in a large neighborhood of the $L_1$ point by their value of the Hamiltonian.

%Let $e^{\pm\bmi 2\pi\nu^h}$ be the multiplier of modulus one of the monodromy matrix of an orbit with Hamiltonian $h$ within certain family and $\nu^h\in[0,1/2]$; that is
%\[
%e^{\pm\bmi 2\pi\nu^h}\in {\rm Spec} {\rm %D}\phi_{T^h}(p^h), \quad \nu^h\in[0,1/2],
%\]
%where $T^h$ is the period of the orbit and $p^h$ a point in that orbit. We can then obtain a curve $\gamma:U^h\subset\R\to U^{\nu^h}\subset\R$ relating the Hamiltonian $h$ and $\nu^h$ throughout a family of periodic orbits; this curve will be particularly useful later on.

Besides the 2-dimensional manifolds of periodic orbits, the center manifold is also filled with 2-dimensional tori that exist around periodic orbits with a central part. Let us focus on the tori around vertical Lyapunov orbits. Let $\hat\omega_p$ and $\hat\omega_v$ be the frequencies of any torus such that $\hat\omega_p\rightarrow\hat\omega_p^0$ and $\hat\omega_v\rightarrow\hat\omega_v^0$ when $h\rightarrow h^0$, see Remark \ref{remark:freq} on the non-uniqueness of the frequencies. Following \cite{HM21}, we define the rotation number as
\[
\rho := \hat\omega_p/\hat\omega_v -1.
\]
We can then use the Hamiltonian $h$ and $\rho$ to represent the family of 2-dimensional tori around the vertical Lyapunov orbits. These two parameters uniquely determine each torus.
%The corresponding $\gamma$ curve for the vertical Lyapunov family bounds the region of existence of tori around vertical Lyapunov orbits, see Fig. \ref{fig:enrho} (left). 
Analogous representations can be obtained for the tori around planar Lyapunov and halo orbits, see \cite{GomezM01,HM21} for a full description.

For $\mu=3.040357143\cdot10^{-6}$ and $e=0$, we selected 77 equally spaced values of $\rho$ between 0.03565 and 0.0961 for the Sun-Earth system and ``nobilized''\footnote{A
noble number is one whose continued fraction expansion coefficients are equal
to one from a position onward.} them with an absolute tolerance of $1.6 \times 10^{-4}$. Then, we performed continuations in the flying time $T$ for each value of $\omega=\rho$, following the methodology from \cite{HM21}, and obtained the family of tori around vertical Lyapunov orbits in the CRTBP. The results are gathered in Fig. \ref{fig:enrho} where we plot the rotation number and the Hamiltonian as well as the grid size of the parameterizations in the color map for each of the 8971 tori computed.

%Analogously, we selected 60 values of $\rho$ equally spaced between 0.0115 and 0.4809 and obtained the tori around halo orbits represented in Fig. \ref{fig:enrho} (right). 
\begin{figure}[htbp]
\includegraphics{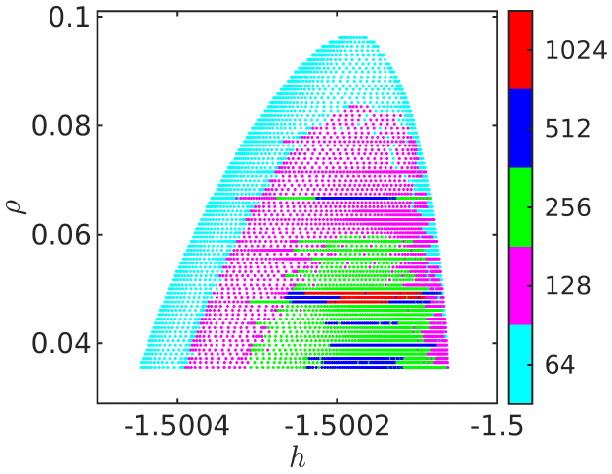}
\caption{\label{fig:enrho}Energy-rotation number representation in the CRTBP for the Lissajous family around vertical orbits around $L_1$ for the Sun-Earth mass parameter. The color map represents the grid size of the parameterizations.} 
\end{figure}

\subsection{Numerical explorations in the ERTBP}
\label{sec:numexp}
In this case study, the generating tori $\K$ are 2-dimensional and the generated tori $\hat\K$ are 3-dimensional. To initialize the algorithms, we use the autonomous tori from the CRTBP with a grid size in the internal phase as given in Fig. \ref{fig:enrho}. For the initial grid size in the external phase, we took $N_2=16$. Since the generating tori in the CRTBP are 1-dimensional, we obtain the 2-dimensional generators by setting the approximate Fourier coefficients $\hat\xi_{kj}$ to zero for $|j|>0$. Equivalently, we can construct the generators as in \eqref{eq:K0}. In algorithm \ref{algo:implementation}, we use $\epsilon^K=10^{-9}, \epsilon^W=10^{-5}, \epsilon^t=10^{-9}, r_t=1/5, n_{max}=6, n_\varepsilon=3, n_{des}=4,$ and $n_t=2$. We set a maximum of $1024$ Fourier coefficients in each phase and  for the multiple shooting approach, we take $m=4$. We set the continuations to reach the eccentricity of the Sun-Earth system; that is, $e=0.01671123$. It is worth mentioning that we set $\epsilon^K=10^{-9}$ for such an exhaustive numerical exploration, but we manage to get errors in the invariance equations of the order of $10^{-15}$ for some tori.

The results of our numerical explorations are gathered in Fig. \ref{fig:toriERTBP}. Out of the 8971 tori computed in the CRTBP, 4457 reached the Sun-Earth eccentricity. Each torus in the Figure is labeled by its value of the internal rotation number and the value of the Hamiltonian of the torus in the CRTBP used to initiate the continuations. Note that we use $h$ simply as a label with no dynamical implications. In the color map of Fig. \ref{fig:toriERTBP}, we represent the grid size in the internal phase $\theta$ (left) and the external phase $\varphi$ (right) only of the tori that reached the eccentricity of the Sun-Earth system.

\begin{figure}[htbp]
\includegraphics{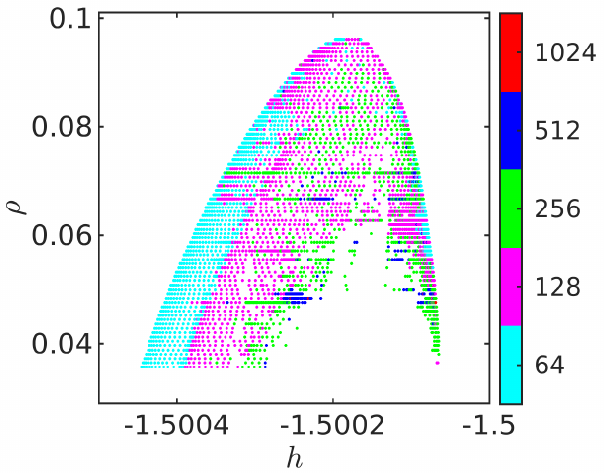}\includegraphics{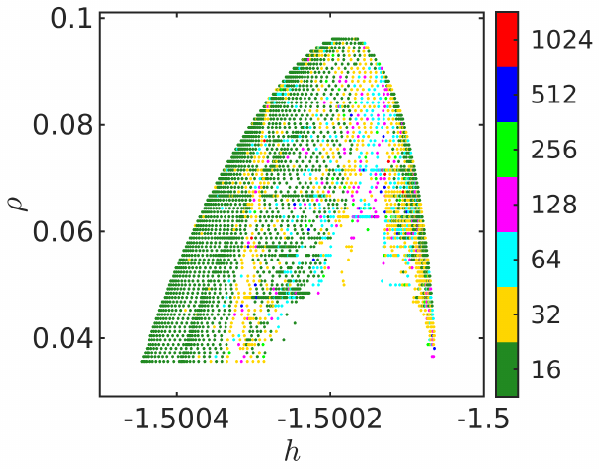}
\caption{\label{fig:toriERTBP}Energy-rotation number representation in the ERTBP for the Lissajous family around vertical orbits around $L_1$ for the Sun-Earth mass parameter. The color map represents the grid size of the internal phase (left) and the external phase (right).} 
\end{figure}

We can observe that not all continuations reach the required value of the eccentricity. Certain empty ``lines", where the continuations fail, are present, which suggests the existence of a dynamical barrier. It turns out that these lines correspond to resonances between the frequencies of the 3-dimensional tori, see Section \ref{sec:reso}. For values of $\rho<0.06$, there is also a gap in the energy-rotation number representation.
% space where very few tori converged\jmm{\sl Yo no diria que han convergido muy pocos, sino que ``there is a large gap''. De $\rho<0.06$ parece que hay mas de la mitad de la region cubierta}
% ---but seemingly, in a different manner than for resonant tori---suggesting a different obstruction mechanism.
The lack of convergence in this region is related to resonances and also to the existence of homoclinic and heteroclinic tangles. We will come back to this issue in Section \ref{sec:poincare}. 

From the numerical explorations we can see that, generally, tori in the ERTBP do not need a large number of approximate Fourier coefficients in the external phase, see Fig. \ref{fig:toriERTBP} (right), which suggests that our tori are more analytic in the phase $\varphi$. Note that $\varphi$ is essentially time and tori tend to be more analytic in the temporal direction, which is in agreement with our results.

Close to resonances, the number of coefficients needed increases. We also observe an increase in the number of coefficients for the tori surrounding the region where homoclinic and heteroclinic tangles are present. Resonances and homoclinic and heteroclinic tangles are responsible for the breakdown of tori, see \cite{Chirikov79,LlaveO06,OlveraS87}. The fact that more coefficients are necessary for the parameterizations close to such cases reveals that the tori are losing regularity because they are breaking down.

% In order to lift the autonomous tori to the elliptic problem, we see that it does not suffice to simply add coefficients in the external phase. When we compare Fig. \ref{fig:enrho} and Fig. \ref{fig:toriERTBP} (left), we see that the Energy-rotation number representation is not preserved.\jmm{\sl Te refieres a que el numero de nodos en la fase interna cambia? Yo lo diria explicitamente} In other words, the regularity in the phase $\theta$ changes from the circular to the elliptical problem.\jmm{\sl Hablar de cambio de regularidad quizas es muy atrevido. Se pueden necesitar mas armonicos porque los toros son mas curvosos pero igual de regulares ($C$-lo-que-sea)} Throughout the continuations, it was seen that for a given torus, sometimes it was necessary to increase the number of coefficients in the phase $\theta$ and sometimes in the phase $\varphi$. It is therefore key to control the decay of the coefficients and adapt the parameterization of each torus according to their individual regularity.\jmm{\sl Habria que buscar una palabra diferente de regularity}

In order to lift the autonomous tori to the elliptic problem, we see that it
does not suffice to simply add coefficients in the external phase. When we
compare Fig. \ref{fig:enrho} and Fig. \ref{fig:toriERTBP} (left), the number of
coefficients in the internal phase changes.  Throughout the continuations, it
was seen that for a given torus, sometimes it was necessary to increase the
number of coefficients in the phase $\theta$ and sometimes in the phase
$\varphi$.  It is therefore key to control the decay of the Fourier
coefficients as it has been described in Section \ref{sec:compimpl}.

Lastly, we include in Fig. \ref{fig:3Dplots} some plots in the configuration space of 3-dimensional tori (red) and their 2-dimensional generators (black) for a family with $\rho=0.071461$. They can be seen as ``fattened" with respect to their CRTBP. Such ``fattening" is clearly seen in the results from Section \ref{sec:poincare}.

% \jmm{\sl Frase tipo que las paredes tienen grosor y eso se refleja en que el generador es 2D en vez de una curva? Pero mejor dicho. Lo intento pensar}

\begin{figure}[htbp]\label{fig:3Dplots}
\includegraphics{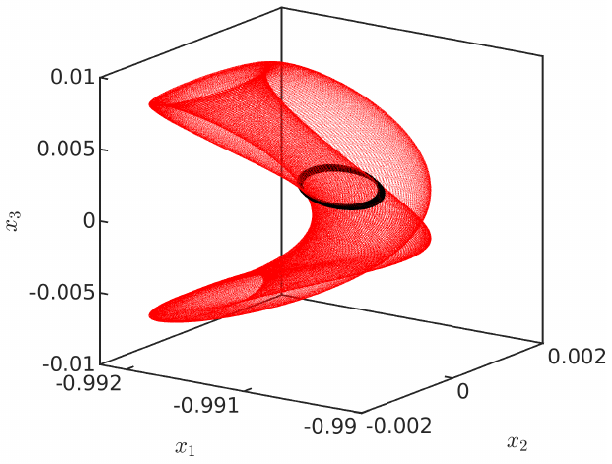}
\includegraphics{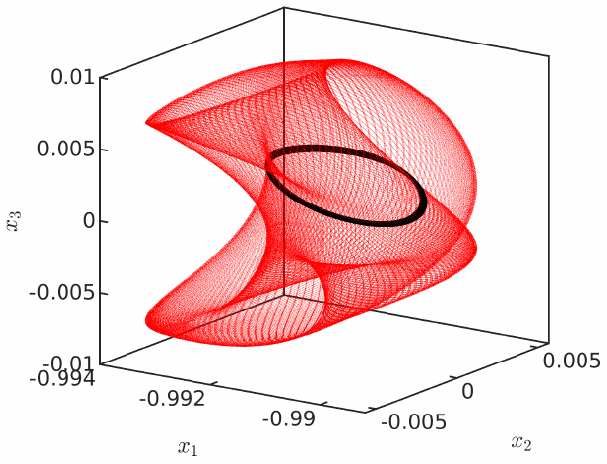}

\includegraphics{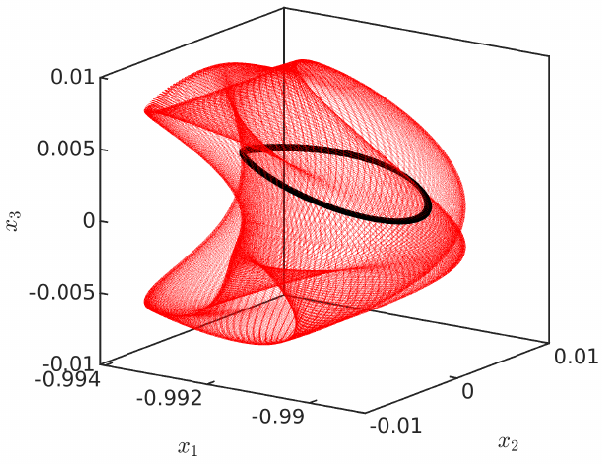}\includegraphics{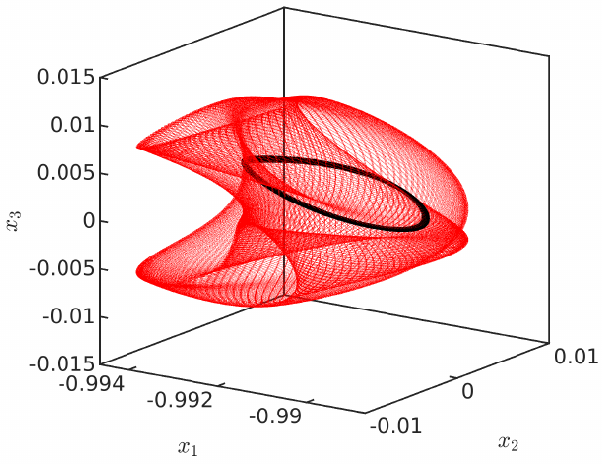}
\includegraphics{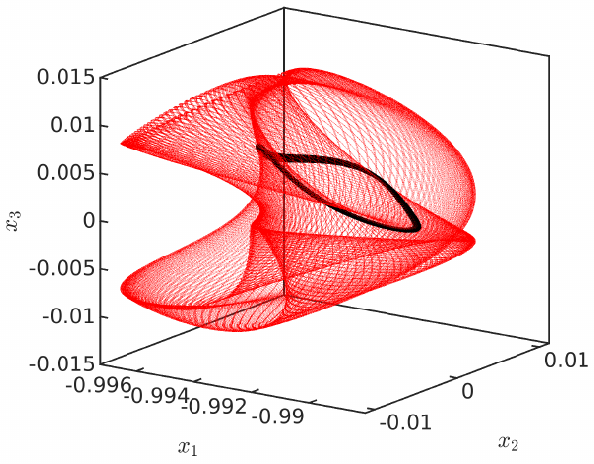}\includegraphics{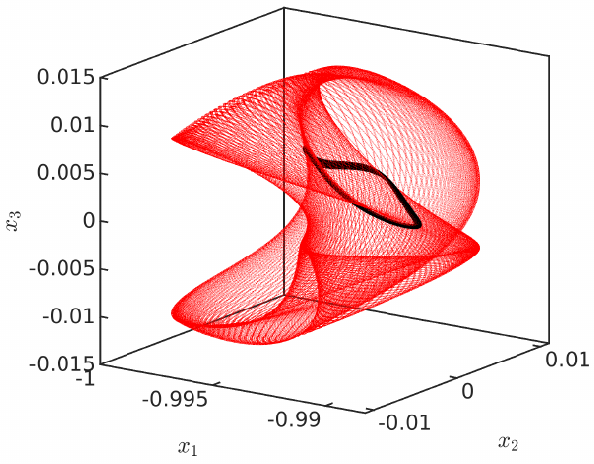}
\includegraphics{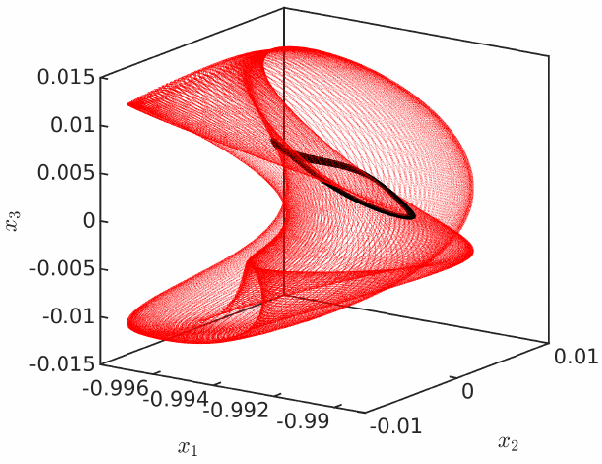}\includegraphics{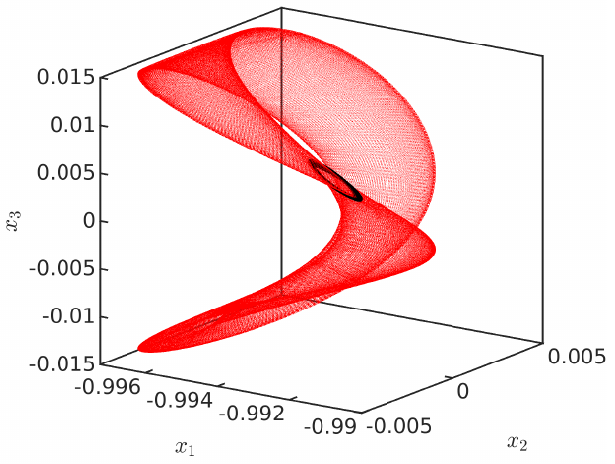}
\caption{Projection onto the configuration space of 3-dimensional generated tori (red) and 2-dimensional generating tori (black) for the family $\rho = 0.071461$.} 
\end{figure}

The behavior of the family of tori for fixed $\rho$ is qualitatively very similar to their autonomous counterparts in the CRTBP, see \cite{HM21}. The family begins with a torus close to a vertical Lyapunov orbit and, with increasing energy, the tori increase in size and start to bend. Then, tori approach a vertical Lyapunov orbit of higher energy where the family collapses.

\begin{remark}
To obtain the results presented in Fig. \ref{fig:toriERTBP}, we did several runs of Algorithm \ref{algo:implementation} for different values of the filtering factor for the lowpass filter described in Section \ref{sec:compimpl}. To give an idea of the computing time, for the family $\rho=0.089837$, we performed continuations in $e$ of 30 tori from the CRTBP to the Sun--Earth ERTBP in $4725.57$ seconds. This is the total computing time, the wall-clock time is roughly the total computing time divided by the number of threads; 24 in our case, so each of the previous continuations was done in roughly 6.6 seconds. Note that not all tori take the same time to compute. Computations close to resonances or close to homoclinic connections can take significantly longer.
\end{remark}

\subsection{Resonances} \label{sec:reso}
The vector of internal frequencies $\hat\omega$ and the vector of external frequencies $\hat\alpha$ were assumed to be sufficiently non-resonant; that is, Diophantine. When this assumption does not hold, there is a dynamical obstruction to the existence of invariant tori. 

Let $\kappa=(\kappa_1,\kappa_2,\kappa_3)$ be a 3-tuple. The vectors of frequencies $\hat\omega=(\frac{\rho}{T},\frac{1}{T})$ and $\hat\alpha=\frac{1}{2\pi}$ are in $p$-order resonance when $|\kappa|_1=p$ and
\begin{equation}
\label{eq:res}
\mathcal{R}_{\kappa}:=\kappa_1\frac{\rho}{T} + \kappa_2\frac{1}{T}+\kappa_3\frac{1}{2\pi},
\end{equation}
becomes zero. The numerical explorations were done for fixed values of $\rho$ but throughout the continuations, the flying time $T$ varies. For certain $\kappa,\rho,$ and $T$, $\mathcal{R}_k$ might become small, revealing that the continuations are approaching resonances. Note that for given $\kappa$, with $|\kappa|_1=p$, $\mathcal{R}_k=0$ defines $p$-resonant lines. To visualize these resonances in the energy-rotation number representation, we first examine the tori computed in the CRTBP. We look for 3-tuples $\bar \kappa$ such that $\mathcal{R}_{\bar \kappa}<\epsilon^{\mathcal{R}}$ up to a maximum order $\bar p$. We took $\epsilon^{\mathcal{R}}=10^{-4}$ and $\bar p =10$. Once we obtain the 3-tuples $\bar \kappa$, we compute the lines $\mathcal{R}_{\bar \kappa}=0$ and obtain $\rho_{\bar{\kappa}} = \rho_{\bar{\kappa}}(T)$. Then, we use the values of $h,\rho$, and $T$ of the grid of tori computed in the CRTBP to obtain through inverse cubic interpolation, for each line $\mathcal{R}_{\bar{\kappa}}=0$, the curve $\rho_{\bar\kappa}=\rho_{\bar{\kappa}}\circ T(h)$. The results are shown in Fig. \ref{fig:res}. 
\begin{figure}[htbp]
    \centering
    \includegraphics{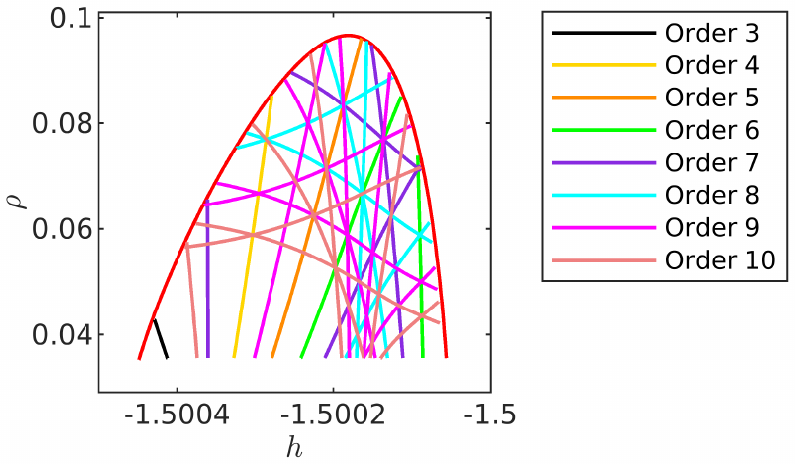}
    \caption{Resonant curves up to order 10 for the tori in the ERTBP for the Sun-Earth mass parameters.}
    \label{fig:res}
\end{figure}

It is clear that when lifting the tori from the CRTBP to the ERTBP, the frequencies can become resonant for a large subset of the autonomous tori. We observe that for $h\gtrsim-1.5002$, there is an accumulation of resonant curves in the energy-rotation number representation which explains the gap where few tori converged. The accumulation of resonant curves also explains why, when comparing Fig. \ref{fig:enrho} and Fig \ref{fig:toriERTBP} (left), more coefficients are necessary for the parameterizations of the tori that reached the Sun-Earth eccentricity.

Note the correspondence between the resonant curves and the  results from Fig. \ref{fig:toriERTBP}---the absence of tori in the region surrounded by tori with $N_2=32$ in Fig. \ref{fig:enrho} (right) corresponds to an order 4 resonance. Other resonant curves can be appreciated but are more subtle---the lower the order of the resonance, the bigger the obstruction to convergence. Due to the presence of resonances, performing continuations for each torus in the CRTBP is a more robust approach than performing them in the ERTBP. When the method tries to compute a resonant torus, it will simply fail and move to the next, whereas if continuations where done in the ERTBP, the method would have to jump through all the resonances it encounters.

\subsection{Poincaré representation}
\label{sec:poincare}
As we pointed out, the Hamiltonian in the ERTBP is not constant. Therefore, we cannot obtain the isoenergetic sections of the center manifold commonly seen in studies for the CRTBP, see e.g. \cite{JorbaM99,GomezM01,ESA3}. Nonetheless, in this section we show some analogous results. In addition to the numerical results already presented, we have also taken tori in the CRTBP around vertical and halo orbits within a level set of the Hamiltonian, lifted them to the Sun-Earth ERTBP, and computed a Poincaré section with $\Sigma = \{x,p\in\R^3 |\  x_3=0, p_3>0\}$.

We explore the level set $H=-1.5002$, value that crosses the region with the large gap, see Fig. \ref{fig:toriERTBP}. The results are gathered in Fig. \ref{fig:2Dsec}. On the left, we show (in red) the intersections of the tori of the CRTBP that, when used as seeds of continuations in the eccentricity $e$, reach the Sun-Earth ERTBP. The sections of the 3-dimensional tori at the end of each of the previous continuations are shown in Fig.~\ref{fig:2Dsec} right. In order to have some reference of where the families of tori end, we show in black the intersections of the last tori computed in the CRTBP around vertical and halo orbits.
\begin{figure}[htbp]
    \centering
    \includegraphics{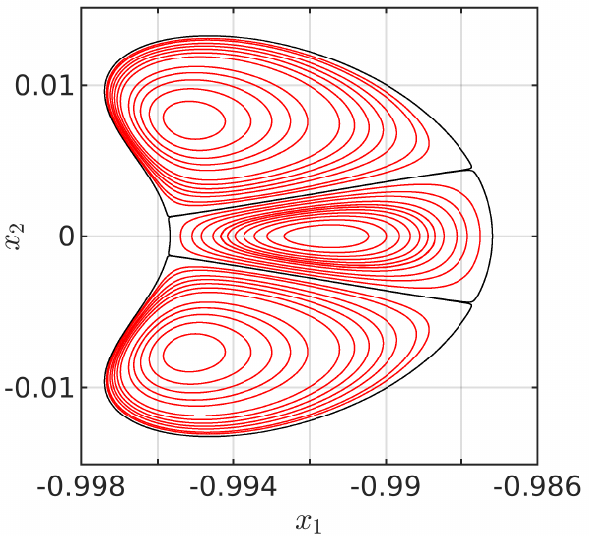}\includegraphics{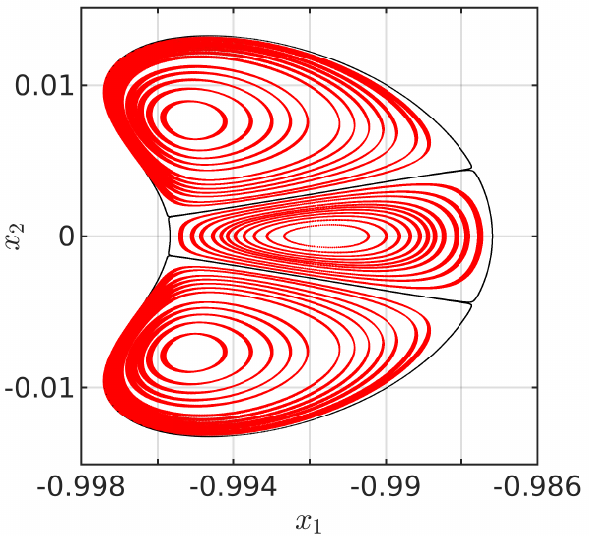}
    \caption{Intersections in red of Lissajous and quasi-halo tori with $\Sigma$ in the CRTBP (left) and in the ERTBP (right). In black, we show the intersections of the last tori computed for each family in the CRTBP. The tori were computed from the level set $H=-1.5002$ in the CRTBP.}
    \label{fig:2Dsec}
\end{figure}

We first observe that since the tori in the ERTBP are 3-dimensional, the intersections with $\Sigma$ are 2-dimensional. This allows us to somewhat see the ``fattening'' of each torus due to the time dependency of the Hamiltonian. This fattening has important implications for the convergence of our method in the region of the energy-rotation number where tori approach homoclinic and heteroclinic connections.

In the CRTBP, there is a range of the Hamiltonian where there exist tori that approach double homoclinic connections of planar Lyapunov orbits. For larger energy values, there exist heteroclinic connections between the so called ``axial'' orbits. These connections act as separatrices between the Lissajous and quasi-halo families. An example of a torus approaching such connections is shown in Fig. 8 of \cite{HM21}, where a torus around a vertical orbit approaches two vertically symmetric quasi-halo tori. The existence of transverse homoclinic and heteroclinic points is one known mechanism for breakdown of invariant tori, see \cite{Chirikov79,LlaveO06,OlveraS87}, and the tori in the CRTBP that approach connections were found for small values of $\rho$, see \cite{HM21}.

Non-resonant periodic orbits in the CRTBP survive as 2-dimensional tori in the ERTBP, so there might be even more connections than in the CRTBP. Together with the fact that tori in the ERTBP have been ``fattened'', suggests that there might be a larger set of tori in the ERTBP that are sufficiently close to homoclinic and heteroclinic connections for the method to fail. 

Lastly, in Fig. \ref{fig:homoclinic_tori}, we show in the configuration space plots of the largest Lissajous (red) and quasi-halo (blue) tori, and their projections, computed for Fig. \ref{fig:2Dsec}. That is, the largest Lissajous and quasi-halo tori that reached the Sun-Earth ERTBP from the level set $H=-1.5002$ in the Sun-Earth CRTBP. For easier visualization, of the two symmetric quasi-halo tori of the energy level only one is shown. We can observe the proximity between the tori in the configuration space (similar plots can be obtained with other projections) and how they almost merge.
\begin{figure}[htbp] 
    \centering
    \includegraphics[scale=0.95]{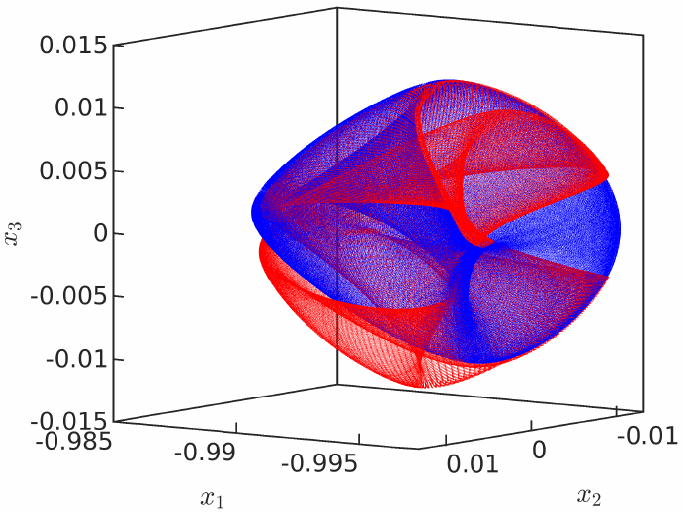}\includegraphics[scale=0.95]{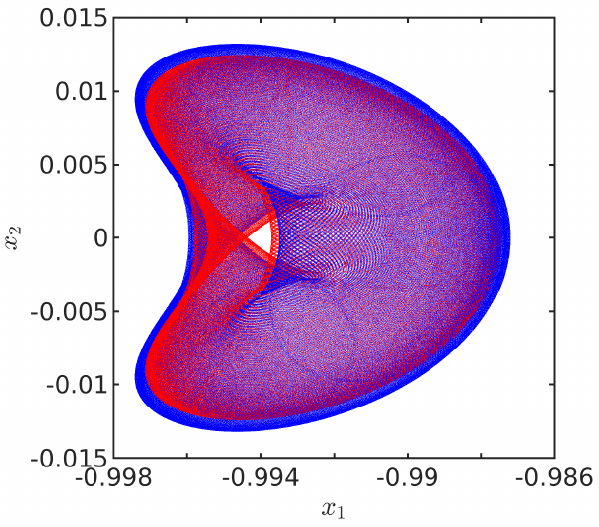} \includegraphics[scale=0.95]{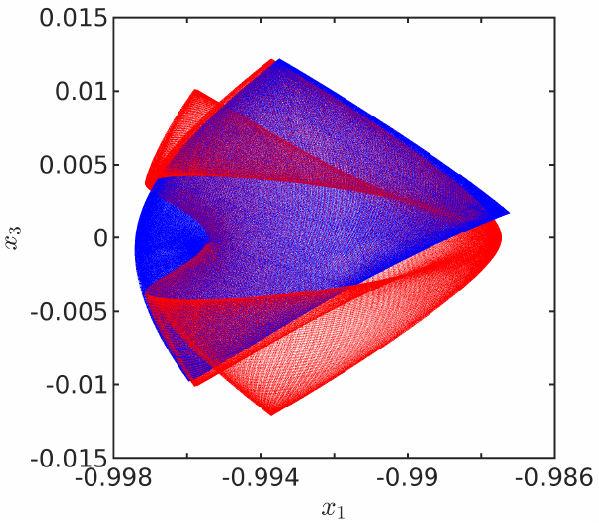}\quad \includegraphics[scale=0.95]{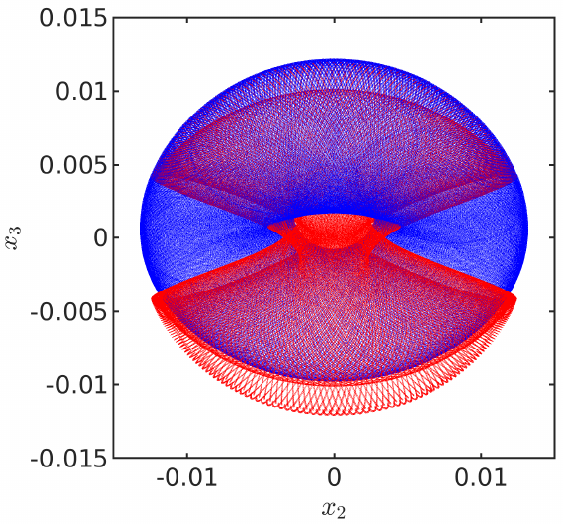}
    \caption{Largest Lissajous (red) and quasi-halo (blue) tori, and their projections, computed in the Sun-Earth ERTBP from the level set $H=-1.5002$ in the CRTBP}
    \label{fig:homoclinic_tori}
\end{figure}

 % testex.tex

\section{Conclusions}
\label{sec:conclusion}

In this work, we presented a general and efficient method to compute parameterizations of invariant tori and their bundles in quasi-periodic Hamiltonian systems with an arbitrary number of frequencies. To this end, we generalized flow map parameterization methods applicable in autonomous settings. This generalization required the development of the notions of fiberwise isotropy of invariant tori and fiberwise Lagrangian subspaces. We also obtained primitive functions of quasi-periodic time-$t$ maps and, in a more general framework, we introduced the concepts of fiberwise symplectic deformations and moment maps. All of these notions are vital for our constructions and for eventual proofs, using KAM techniques, of the existence of invariant objects and their regularity with respect to parameters.

We manage to reduce the dimension of the phase space by considering appropriate functional equations. We also reduce the dimension of our objects by using flow maps. Instead of using periods associated to the Hamiltonian we use periods associated to the internal frequencies of tori---which allow us to directly compute invariant tori in quasi-periodic Hamiltonian systems from tori in autonomous systems in an efficient and general setting.  We also provided a continuation method, under the parameterization method paradigm, for continuation of invariant tori and bundles with respect to parameters of the Hamiltonian.

We tested our method in a periodic Hamiltonian system: the Elliptic Restricted Three Body Problem (ERTBP). We computed a large grid of 3-dimensional invariant tori and their invariant bundles, gaining qualitative insight into the behavior of the ERTBP. From the numerical explorations, we observed that tori and bundles are more regular in the external phase. This implies that less coefficients are required for their parameterization. Consequently, it is advantageous to include external phases in the parameterizations. Additionally, we observed that for a large set of tori the frequencies can be in resonance; which is an obstruction to their existence. Therefore, computing tori by lifting them from an autonomous system is a more robust approach than performing continuations directly in the ERTBP.

 % conclusion.tex

\appendix
\newpage
\section{On the primitive function of $\phi_t$}\label{ap:primitive}
We dedicate this appendix to the explicit expression of the primitive function of the $t-$depending family of fiberwise exact symplectomorphism $\phi_t: U\times\T^{\ell}\to U$, for $U\subset\R^{2n}$ open and $t\in\R$, including its derivatives.

\begin{lemma}\label{lemma:primitive}
The primitive function $p_t:U\times\T^\ell\to \R$ of the fiberwise exact symplectomorphism $\phi_t:U\times\T^\ell\to U$ is given by 
\begin{equation}\label{eq:prim}
\begin{split}
	&p_t(z,\varphi)=\\
	&\int_0^t \left( a\bigl(\phi_s(z,\varphi)\bigr)^\top X_H\bigl(\phi_s(z,\varphi),R_s(\varphi)\bigr) -
        H\bigl(\phi_s(z,\varphi),R_s(\varphi)\bigr)  \right) ds,  
\end{split} 
\end{equation}
where $R_s(\varphi) = \varphi + \hat\alpha s$. That is, the primitive function $p_t$, as given above, satisfies  
\begin{equation}
{\rm D}_z p_t(z,\varphi) = a\big(\phi_t(z,\varphi)\big)^\top{\rm D}_z\phi_t(z,\varphi) - a(z)^\top. \label{eq:ap_primitiveZ}
\end{equation}
Moreover,
\begin{equation}
{\rm D}_\varphi p_t(z,\varphi) = a\big(\phi_t(z,\varphi)\big)^\top{\rm D}_\varphi \phi_t(z,\varphi) - \int^t_0 \Big({\rm D}_\varphi H\big(\phi_s(z,\varphi),R_s(\varphi)\big) \Big)ds. \label{eq:ap_primitivevarphi}
\end{equation}
\end{lemma}

\begin{proof}
In order to prove $p_t$ is as given in \eqref{eq:prim}, it suffices to show that $p_t$ satisfies  \eqref{eq:ap_primitiveZ}. Let us begin by differentiating $p_t$ with respect to $z$
\begin{equation*}
\begin{split}
{\rm D}_z p_t(z,\varphi) =& \int_0^t\Big(X_H\big(\phi_s(z,\varphi),R_s(\varphi)\big)^\top{\rm D} a\big(\phi_s(z,\varphi)\big){\rm D}_z\phi_s(z,\varphi) \\
& \phantom{\smash{\int_0^t\Big(}} + a\big(\phi_s(z,\varphi)\big)^\top{\rm D}_z X_H(\phi_s(z,\varphi),R_s(\varphi)) {\rm D}_z\phi_s(z,\varphi)\\
& \phantom{\smash{\int_0^t\Big(}}- {\rm D}_z H\big(\phi_s(z,\varphi),R_s(\varphi)\big){\rm D}_z\phi_s(z,\varphi)\Big) ds.
\end{split}
\end{equation*}
We use that ${\rm D}_zH(z,\varphi)=-X_H^\top(z,\varphi)\Omega(z)$ and that $\Omega(z) = {\rm D}a(z)^\top - {\rm D}a(z)$ to obtain
\begin{equation*}
    \begin{split}
    {\rm D}_z p_t(z,\varphi)
    &=
      \int_0^t\Big(X_H\big(\phi_s(z,\varphi),R_s(\varphi)\big)^\top {\rm D}
      a\big(\phi_s(z,\varphi)\big)^\top{\rm D}_z\phi_s(z,\varphi) \\
    & \phantom{\smash{=\int_0^t\Big(}} + a\big(\phi_s(z,\varphi)\big)^\top{\rm
       D}_z X_H\big(\phi_s(z,\varphi),R_s(\varphi)\big) {\rm
       D}_z\phi_s(z,\varphi) \Big) ds \\
    & =\int_0^t\bigg(\frac{d}{ds} \Big(a\big(\phi_s(z,\varphi)\big)\Big)^\top
        {\rm D}_z \phi_s(z,\varphi) %\right.
        \\
    &\phantom{=\smash{\int_0^t\bigg(}} %\left.
       + a\big(\phi_s(z,\varphi)\big)^\top \frac{d}{ds}\Big({\rm D}_z
       \phi_s(z,\varphi)\Big) %\right)
       \bigg) ds \\
    &  = \int_0^t\frac{d}{ds}\Big(a\big(\phi_s(z,\varphi)\big)^\top{\rm D}_z\phi_s(z,\varphi) \Big)ds\\
    &= a\big(\phi_t(z,\varphi)\big)^\top{\rm D}_z \phi_t(z,\varphi) - a(z).
    \end{split}
\end{equation*}
Let us now differentiate $p_t$ with respect to $\varphi$ 
\begin{align*}
	{\rm D}_\varphi p_t(z,\varphi)&=
	 \int_0^t\Big(X_H\big(\phi_s(z,
	 \varphi),R_s(\varphi)\big)^\top {\rm D}
	 a\big(\phi_s(z,\varphi)\big){\rm D}
	 _\varphi\phi_s(z,\varphi) \\
	&   \phantom{=\smash{\int_0^t\bigg(}} 
	+a\big(\phi_s(z,\varphi)\big)^\top{\rm D}_z
	 X_H\big(\phi_s(z,\varphi),R_s(\varphi)\big)
	 {\rm D}_\varphi\phi_s(z,\varphi) \\
	&   \phantom{=\smash{\int_0^t\bigg(}}+
	 a\big(\phi_s(z,\varphi)\big)^\top {\rm D}
	_\varphi X_H\big(\phi_s(z,
	\varphi),R_s(\varphi)\big)\\
	& \phantom{=\smash{\int_0^t\bigg(}}- {\rm D}_z 
	H\big(\phi_s(z,\varphi),R_s(\varphi)\big)
	{\rm D}_\varphi\phi_s(z,\varphi)  \\
	&	\phantom{=\smash{\int_0^t\bigg(}} - 
	{\rm D}_\varphi H\big(\phi_s(z,		
	\varphi),R_s(\varphi)\big)\Big) ds,    
\end{align*}
which we can rewrite as
\begin{align*}
    {\rm D}_\varphi p_t(z,\varphi)&=
     \int_0^t\Big(X_H\big(\phi_s(z,
     \varphi),R_s(\varphi)\big)^\top {\rm D}
     a\big(\phi_s(z,\varphi)\big)^\top{\rm D}
     _\varphi\phi_s(z,\varphi)\\
    &	 \phantom{=\smash{\int_0^t\bigg(}} 
    +a\big(\phi_s(z,\varphi)
    \big)^\top\frac{d}{ds}\Big({\rm D}_\varphi 
    \phi_s(z,\varphi)\Big)\\
    & \phantom{=\smash{\int_0^t\bigg(}} 
     - {\rm D}_\varphi 
    H\big(\phi_s(z,\varphi),R_s(\varphi)\big)
    \Big)ds \\
    &	=\int_0^t\bigg(\frac{d}{ds}
    \Big(a\big(\phi_s(z,\varphi)\big)^\top {\rm D}
    _\varphi\phi_s(z,\varphi) \Big)\\
    &	\phantom{=\smash{\int_0^t\bigg(}} 
    - {\rm D}_\varphi 
    H\big(\phi_s(z,\varphi),R_s(\varphi)\big)
    \bigg) ds \\
    &	= a\big(\phi_t(z,\varphi)\big)^\top{\rm D}
    _\varphi \phi_t(z,\varphi) - \int^t_0 {\rm D}
    _\varphi H\big(\phi_s(z,\varphi),R_s(\varphi)
    \big)\ ds.
\end{align*}

\end{proof}
\section{Fiberwise isotropy of invariant tori}
\label{ap:isotropy}
Invariant tori have geometric properties that we leverage in our method. In this appendix, we prove the fiberwise isotropy of $\mathcal{K}$.

\begin{lemma}\label{lemma:iso}
Assume $K:\T^{d-1}\times\T^{\ell}\to U$ is a parameterization of a $\phi_T$-invariant torus $\mathcal{K}$ and that the dynamics on $\mathcal{K}$ are conjugate to an ergodic rotation, that is
\[
\phi_T\big(K(\theta,\varphi),\varphi\big) = K(\theta+\omega,\varphi+\alpha),
\]
with $(\omega,\alpha)$ rationally independent. For each $\varphi \in\T^\ell$, consider parameterizations $K_\varphi:\T^{d-1}\to U$ of tori $\mathcal{K}_\varphi\subset\mathcal{K}$ such that $K_\varphi(\theta)=K(\theta,\varphi)$. Then, for each $\varphi$, $\mathcal{K}_\varphi$ is an isotropic torus and thus $\mathcal{K}$ is fiberwise isotropic.
\end{lemma}

\begin{proof}
We need to show that $K_\varphi^*\bm{\omega}=0$ for all $\varphi\in\T^\ell$. In coordinates, this condition reads 
\begin{equation}
\label{eq:ap_iso}
    {\rm D}_\theta K_\varphi(\theta)^\top\Omega\big(K_\varphi(\theta)\big){\rm D}_\theta K_\varphi(\theta) = 0,
\end{equation}
for all $\varphi$. Let us first introduce the definitions $\bar\theta:=\theta+\omega$ and $\bar\varphi := \varphi + \alpha$. From the symplecticity of $\phi_T$ we have the identity 
\begin{equation}
\label{eq:ap_symplecticity}
    \Omega\big(K_\varphi(\theta)\big) = {\rm D}_z\phi_T\big(K_\varphi(\theta),\varphi\big)^\top \Omega\big(K_{\bar\varphi}(\bar\theta)\big) {\rm D}_z \phi_T\big(K_\varphi(\theta),\varphi\big).
\end{equation}
From the invariance of $\mathcal{K}$, the tangent bundle of $\mathcal{K}_\varphi$, parameterized by the columns of ${\rm D}_\theta K_\varphi$, is transported by the differential of $\phi_T$ as
\begin{equation}
\label{eq:ap_dif_inv}
    {\rm D}_z \phi_T\big(K_\varphi(\theta),\varphi\big) {\rm D}_\theta K_\varphi(\theta) = {\rm D}_\theta K_{\bar\varphi}(\bar\theta).
\end{equation}
Then, using \eqref{eq:ap_symplecticity}, \eqref{eq:ap_dif_inv}, and rewriting $K_\varphi(\theta)$ as $K(\theta,\varphi)$, we have
\[
{\rm D}_\theta K(\theta,\varphi)^\top\Omega\big(K(\theta,\varphi)\big) {\rm D}_\theta K(\theta,\varphi) = {\rm D}_\theta K(\bar\theta,\bar\varphi)^\top\Omega\big(K(\bar\theta,\bar\varphi)\big) {\rm D}_\theta K(\bar\theta,\bar\varphi),
\]
which is constant in $\theta,\varphi$ due to the ergodicity of $(\theta,\varphi)\mapsto(\theta+\omega,\varphi+\alpha)$, coming from the fact that, for fixed $\theta,\varphi$, $\{(\theta+j\omega,\varphi+j\alpha)\}_{j\in\Z}$ is dense in $\T^{d-1}\times\T^\ell$.

Since $\Omega(z) = {\rm D}a(z)^\top - {\rm D}a(z)$, the left-hand side of \eqref{eq:ap_iso} reads
\begin{equation}
\label{eq:ap_iso2}
{\rm D}_\theta\Big( a\big(K(\theta,\varphi)\big)\Big)^\top {\rm D}_\theta K(\theta,\varphi) - {\rm D}_\theta K(\theta,\varphi)^\top {\rm D}_\theta \Big(a\big(K(\theta,\varphi)\big)\Big)
\end{equation}
and each entry $ij$ of \eqref{eq:ap_iso2} can be expressed as
\begin{equation}
\label{eq:ap_iso3}
\begin{split}
 \phantom{=}&\sum_{ k} \partial_{\theta^i}\Big( a_k\big(K(\theta,\varphi)\big)\Big)\partial_{\theta^j} K^k(\theta,\varphi) - \partial_{\theta^j}\Big(a_k\big(K(\theta,\varphi)\big)\Big) \partial_{\theta^i} K^k(\theta,\varphi) \\
=&\sum_{k} \partial_{\theta^i}\Big( a_k\big(K(\theta,\varphi)\big)\partial_{\theta^j} K^k(\theta,\varphi)\Big) - a_k\big(K(\theta,\varphi)\big) \partial^2_{\theta^i\theta^j}K^k(\theta,\varphi) \\
\phantom{=}& - \partial_{\theta^j}\Big(a_k\big(K(\theta,\varphi)\big) \partial_{\theta^i} K^k(\theta,\varphi) \Big) + a_k\big(K(\theta,\varphi)\big) \partial^2_{\theta^i\theta^j}K^k(\theta,\varphi)\\
=& \sum_{k} \partial_{\theta^i} \Big(a_k\big(K(\theta,\varphi)\big)\partial_{\theta^j} K^k(\theta,\varphi) \Big) - \partial_{\theta^j}\Big(a_k\big(K(\theta,\varphi)\big) \partial_{\theta^i} K^k(\theta,\varphi) \Big).
\end{split}
\end{equation}
The averages of each entry $ij$ of \eqref{eq:ap_iso2} are zero since we are taking averages of derivatives with respect to $\theta$. Then, since \eqref{eq:ap_iso2} is constant, it is identically zero which concludes the proof.
\end{proof}

\section{Fiberwise Lagrangian space generated by the subframe $L$} \label{ap:fiberwise_lag}
The subframe $L$, used to construct the adapted frame $P$, needs to be a fiberwise Lagrangian subframe. We dedicate this section to the following lemma.

\begin{lemma} \label{lem:lag}
For each $\varphi\in\T^\ell$, the subframe $L_{\varphi}(\theta):=L(\theta,\varphi)$
constructed as
\begin{equation*}
     L(\theta, \varphi)=\Big({\rm D}_{\theta}K(\theta,\varphi) \quad \mathcal{X}_H(\theta,\varphi) \quad W(\theta,\varphi) \Big),
\end{equation*}
where 
\begin{equation*}
\mathcal{X}_H(\theta,\varphi):=X_H\big(K(\theta,\varphi),\varphi\big) - {\rm D}_\varphi  K(\theta,\varphi)\hat{\alpha},    
\end{equation*}
generates a Lagrangian subspace and, consequently, $L$ is a fiberwise Lagrangian subframe.
\end{lemma}

\begin{proof}
If $L$ is a fiberwise Lagrangian subframe, it needs to satisfy
\begin{equation*}
\label{eq:L_lag}
L(\theta,\varphi)^{\top}\Omega\big(K(\theta,\varphi)\big)L(\theta,\varphi)=0.
\end{equation*}
Since $\Omega$ is skew-symmetric, it suffices to prove the following:
\begin{enumerate}[i.]
    \item \label{it:1}${\rm D}_\theta K(\theta,\varphi)^\top\Omega\big(K(\theta,\varphi)\big){\rm D}_\theta K(\theta,\varphi)=0$,
    \item \label{it:2}${\rm D}_\theta K(\theta,\varphi)^\top\Omega\big(K(\theta,\varphi)\big)\mathcal{X}_H(\theta,\varphi)=0$,
    \item \label{it:3}${\rm D}_\theta K(\theta,\varphi)^\top\Omega\big(K(\theta,\varphi)\big)W(\theta,\varphi)=0$,
    \item \label{it:4}$\mathcal{X}_H(\theta,\varphi)^\top \Omega\big(K(\theta,\varphi)\big)\mathcal{X}_H(\theta,\varphi)=0$,
    \item \label{it:5}$\mathcal{X}_H(\theta,\varphi)^\top \Omega\big(K(\theta,\varphi)\big) W(\theta,\varphi)=0$,
    \item \label{it:6}$W(\theta,\varphi)^\top\Omega\big(K(\theta,\varphi)\big)W(\theta,\varphi)=0$.
\end{enumerate}
From Lemma \ref{lemma:iso}, we have that $\mathcal K$ is fiberwise isotropic. Hence, \eqref{it:1} is immediately satisfied. Additionally, \eqref{it:4} and \eqref{it:6} are the symplectic product of vectors with themselves; hence trivial. Note that \eqref{it:6} is also true for stable and unstable bundles of higher rank. The proof uses a similar contraction argument to the one we use here to prove \eqref{it:3} and \eqref{it:5}, which is the following. Let $\bar \theta^k := \theta + k\omega$ and $\bar \varphi^k := \varphi + k\alpha$ with $k\in\N$. From the symplecticity of $\phi_T$ and the invariance of $\mathcal{K}$, we obtain 
\begin{equation}
\label{eq:Omegaflow}
      \Omega\big(K(\theta,\varphi)\big) = {\rm D}_z\phi_T\big(K(\theta,\varphi),\varphi\big)^\top \Omega\big(K(\bar\theta,\bar\varphi )\big) {\rm D}_z\phi_T\big(K(\theta,\varphi),\varphi\big). 
\end{equation}
If we apply \eqref{eq:Omegaflow} recursively $k$ times in \eqref{it:3} and \eqref{it:5}, and use the invariance of $D_\theta K, \mathcal{X}_H,$ and $\mathcal{W}$, we obtain the identities
\begin{align*}
{\rm D}_\theta K(\theta,\varphi)^\top\Omega\big(K(\theta,\varphi)\big)&W(\theta,\varphi) = \\
&\lambda^k{\rm D}_\theta K(\bar\theta^k,\bar\varphi^k)^\top\Omega\big(K(\bar\theta^k,\bar\varphi^k)\big)W(\bar\theta^k,\bar\varphi^k), \\
 \mathcal{X}_H(\theta,\varphi)^\top\Omega\big(K(\theta,\varphi)\big)&W(\theta,\varphi)=\\
 &\lambda^k\mathcal{X}_H(\bar\theta^k,\bar\varphi^k)^\top\Omega\big(K(\bar\theta^k,\bar\varphi^k)\big)W(\bar\theta^k,\bar\varphi^k).   
\end{align*}
Let us now assume the dynamics in $\mathcal{W}$ is contracting, i.e., $\lambda<1$. Then by taking $k\rightarrow\infty$, we conclude \eqref{it:3} and \eqref{it:5} are zero. For the case where the dynamics in $\mathcal{W}$ is expanding, we can consider the invariance of $\mathcal{K}$ under $\phi_{-T}$ and proceed analogously.

For the proof of \eqref{it:2}, we substitute \eqref{eq:Omegaflow} in the left-hand side and use that ${\rm D}_\theta K$ and $\mathcal{X}_H$ are invariant under ${\rm D}_z \phi_T$. Then, we have 
\begin{equation*}
    {\rm D}_\theta K(\theta,\varphi)^\top\Omega\big(K(\theta,\varphi)\big)\mathcal{X}_H(\theta,\varphi) = 
    {\rm D}_\theta K(\bar\theta,\bar\varphi)^\top\Omega\big(K(\bar\theta,\bar\varphi)\big)\mathcal{X}_H(\bar\theta,\bar\varphi).
\end{equation*}
From the ergodicity of $(\omega,\alpha)$, the left-hand side of \eqref{it:2} must be constant. Recall that $\langle\cdot\rangle$ denotes the average with respect to $\theta$ and $\varphi$.  Let us then take the average of the left-hand side of \eqref{it:2} after expanding $\mathcal{X}_H$. That is,
\[
\langle{\rm D}_\theta K(\theta,\varphi)^\top \Omega\big(K(\theta,\varphi)\big)\Big(X_H\big(K(\theta,\varphi),\varphi\big) - {\rm D}_\varphi K(\theta,\varphi) \hat\alpha\Big)\rangle.
\]
From the first term, using that $\Omega(z)X_H(z,\varphi)= {\rm D}_z H(z,\varphi)^\top$, we obtain
\[
\langle{\rm D}_\theta K(\theta,\varphi)^\top {\rm D}_zH\big(K(\theta,\varphi),\varphi\big)^\top\rangle = \langle {\rm D}_\theta\Big( H\big(K(\theta,\varphi),\varphi\big)\Big)^\top\rangle=0
\]
since we are taking averages of derivatives with respect to $\theta$. 
For the second term, we use that $\Omega(z) = {\rm D}a(z)^\top - {\rm D}a(z)$ and obtain
\begin{equation}
\label{eq:proof2}
\langle {\rm D}_\theta K(\theta,\varphi)^\top {\rm D}_\varphi \Big( a\big(K(\theta,\varphi)\big) \Big ) -{\rm D}_\theta \Big(a\big(K(\theta,\varphi)\big)\Big)^\top {\rm D}_\varphi K(\theta,\varphi)\rangle \hat\alpha.
\end{equation}
Then, each term $ij$ in \eqref{eq:proof2} is obtained from
\begin{equation*}
\begin{split}
&\langle \sum_{k}\partial_{\theta^i}K^k(\theta,\varphi)\partial_{\varphi^j}\Big(a_k\big(K(\theta,\varphi)\big)\Big)- \partial_{\theta^i}\Big(a_k\big(K(\theta,\varphi)\big)\Big)\partial_{\varphi^j}K^k(\theta,\varphi)\rangle\\
&=\langle \sum_{k} \partial_{\varphi^j}\Big(a_k\big(K(\theta,\varphi)\big) \partial_{\theta^i} K^k(\theta,\varphi) \Big) - a_k\big(K(\theta,\varphi)\big)\partial^2_{\theta^i\varphi^j}K^k(\theta,\varphi)\\ 
&\phantom{=} -\partial_{\theta^i}\Big(a_k\big(K(\theta,\varphi)\big) \partial_{\varphi^j} K^k(\theta,\varphi) \Big) + a_k\big(K(\theta,\varphi)\big)\partial^2_{\theta^i\varphi^j}K^k(\theta,\varphi)\rangle\\
 & = 0
 \end{split}
\end{equation*}
since we are taking averages of derivatives with respect to $\theta$ and $\varphi$. Therefore \eqref{it:2} also holds.
\end{proof}

\section{Quadratically small averages}
\label{ap:quadratiically small averages}
For the Newton step described in Section \ref{sec:nwttr} to be consistent, we need the averages of $\eta^3$ given by
\begin{equation*}
	\eta^3(\theta,\varphi) = 
	\begin{pmatrix} 
	-{\rm D}_\theta K(\bar\theta,\bar\varphi)^\top \Omega\big(K(\bar\theta,\bar\varphi)\big) E^K(\theta,\varphi) 
	\\
	-\mathcal{X}_H(\bar\theta,\bar\varphi)^\top \Omega\big(K(\bar\theta,\bar\varphi)\big) E^K(\theta,\varphi)
	\end{pmatrix} =: 
	\begin{pmatrix}
	\eta^{31}(\theta,\varphi)
	\\
	\eta^{32}(\theta,\varphi)
	\end{pmatrix}	 
\end{equation*}
to be quadratically small with respect to the error in the torus invariance $E^K$. Note that for some suitable norm, the derivatives of $E^K$ can be controlled by $\|E^K\|$ using Cauchy estimates. The following lemma provides the explicit formulas for $\langle\eta^3\rangle$.
\begin{lemma}
The averages of $\eta^{31}$ and $\eta^{32}$ are 
\begin{align*}
\langle\eta^{31}\rangle &= -\langle{\rm D}_\theta E^K(\theta,\varphi)^\top \Delta^1 a(\theta,\varphi)+ {\rm D}_\theta K (\bar\theta,\bar\varphi)^\top\Delta^2a(\theta,\varphi)\rangle\\
\langle\eta^{32}\rangle &= \langle \Delta^1a(\theta,\varphi)^\top {\rm D}_\varphi E^K(\theta,\varphi)\hat\alpha\rangle + \langle\Delta^2a(\theta,\varphi)^\top{\rm D}_\varphi K (\bar\theta,\bar\varphi)\hat\alpha\rangle \\
&\phantom{=} - \langle \Delta^2H(\theta,\varphi)\rangle,
\end{align*}
where the Taylor remainders $\Delta^i$ of order $i$ in $E^K$ are given by 
\begin{align*}
	\Delta^1 a(\theta,\varphi):&= 
	a\big(\phi_T(K(\theta,\varphi),\varphi)\big)
	 - 	a\big(K(\bar\theta,\bar\varphi)\big)
	\\
	&=
	\int_{0}^1 {\rm D} a\big(K(\bar\theta,
	\bar\varphi)+s E^K(\theta,\varphi)\big)
	 E^K(\theta,\varphi)\ ds,
\end{align*}
\begin{align*}
	&\Delta^2  a(\theta,\varphi):= 
	a\big(\phi_T(K(\theta,\varphi),\varphi)\big) - a\big(K(\bar\theta,\bar\varphi)\big) -
	{\rm D} a\big(K(\bar\theta,\bar\varphi)\big) E^K(\theta,\varphi)   \\
	&=\int_{0}^1 (1-s){\rm D}^2 a\big(K(\bar\theta,\bar\varphi)+s E^K(\theta,\varphi)\big) [E^K(\theta,\varphi),E^K(\theta,\varphi)]\ ds, \\
	&\Delta^2 H(\theta,\varphi):=H\big(\phi_T(K(\theta,\varphi),\varphi),\bar\varphi\big) - H\big(K(\bar\theta,\bar\varphi),\bar\varphi\big) \\
	&\phantom{\Delta^2 H(\theta,\varphi)=} - {\rm D}_z H\big(K(\bar\theta,\bar\varphi),\bar\varphi\big)E^K(\theta,\varphi)   \\
    &=\int_0^1(1-s){\rm D}^2_z H\big(K(\bar\theta,\bar\varphi) + sE(\theta,\varphi),\bar\varphi\big)[E^K(\theta,\varphi),E^K(\theta,\varphi)]\ ds.
\end{align*}
\end{lemma}

\begin{proof}
Let us start with $\langle\eta^{31}\rangle$ by using the exactness of the symplectic form as expressed in \eqref{eq:exactness} to obtain 
\[
\begin{split}
 \langle \eta^{31} \rangle
= &
 - \langle {\rm D}_\theta K(\bar\theta,\bar\varphi)^\top {\rm D} a\big(K(\bar\theta,\bar\varphi)\big)^\top E^K(\theta,\varphi)\rangle \\
 & 
 + \langle  {\rm D}_\theta K(\bar\theta,\bar\varphi)^\top {\rm D} a\big(K(\bar\theta,\bar\varphi)\big) E^K(\theta,\varphi) \rangle.
\end{split}
\]
We use that
\[
\begin{split}	
	0 = & \langle  {\rm D}_\theta\Big( a\big(K(\bar\theta,\bar\varphi)\big)^\top E^K(\theta,\varphi) \Big)  \rangle \\
	   = & \langle E^K(\theta,\varphi)^\top  {\rm D}_\theta \big( a\big(K(\bar\theta,\bar\varphi)\big)\big) + a\big(K(\bar\theta,\bar\varphi)\big)^\top {\rm D}_\theta E^K(\theta,\varphi)   \rangle,
\end{split}	
\]
so we have
\begin{align*}
\langle \eta^{31}\rangle &= \\
& \langle  {\rm D}_\theta E^K(\theta,\varphi)^\top a\big(K(\bar\theta,\bar\varphi)\big) 
+{\rm D}_\theta K(\bar\theta,\bar\varphi)^\top \big(\Delta^1 a(\theta,\varphi)-\Delta^2 a(\theta,\varphi)\big)\rangle.
\end{align*}
We will now use 
\begin{gather*}
{\rm D}_\theta E^K(\theta,\varphi) = {\rm D}_z \phi_T\big(K(\theta,\varphi),\varphi\big){\rm D}_\theta K(\theta,\varphi) - {\rm D}_\theta K(\bar\theta,\bar\varphi),\\
\langle {\rm D}_\theta K( \bar\theta,\bar\varphi)^\top a\big(K(\bar\theta,\bar\varphi)\big)\rangle
= \langle {\rm D}_\theta K(\theta,\varphi)^\top a\big(K(\theta,\varphi)\big)\rangle,
\end{gather*}
Eq. \eqref{eq:ap_primitiveZ}, and 
\[
\begin{split}
{\rm D}_\theta\big( &p_T(K(\theta,\varphi),\varphi)\big) = \\
&\Big(a\big(\phi_T(K(\theta,\varphi),\varphi)\big)^\top {\rm D}_z\phi_T(K(\theta,\varphi),\varphi) - a\big(K(\theta,\varphi)\big)^\top \Big){\rm D}_\theta K(\theta,\varphi)
\end{split}
\]
to rewrite $\langle\eta^{31}\rangle$ as
\[
\begin{split}
 \langle \eta^{31}\rangle 
= &
\phantom{+}
\langle {\rm D}_\theta K(\theta,\varphi)^\top 
\Big({\rm D}_z\phi_T\big(K(\theta,\varphi),\varphi\big)^\top a\big(\phi_T(K(\theta,\varphi),\varphi)\big) \\ 
\phantom{+} & - a\big(K(\theta,\varphi)\big)\Big)\rangle  - \langle{\rm D}_\theta E^K(\theta,\varphi)^\top \Delta^1 a(\theta,\varphi)\rangle \\
& - \langle{\rm D}_\theta K(\bar\theta,\bar\varphi)^\top \Delta^2 a(\theta,\varphi)
\rangle 
\\
=&  \phantom{+}
\langle  {\rm D}_\theta\big(p_T(K(\theta,\varphi),\varphi)\big)^\top  \rangle 
\\
& - \langle {\rm D}_\theta E^K(\theta,\varphi)^\top \Delta^1 a(\theta,\varphi) + {\rm D}_\theta K(\bar\theta,\bar\varphi)^\top \Delta^2 a(\theta,\varphi)\rangle
\\
=& 
- \langle {\rm D}_\theta E^K(\theta,\varphi)^\top \Delta^1 a(\theta,\varphi) + {\rm D}_\theta K(\bar\theta,\bar\varphi)^\top \Delta^2 a(\theta,\varphi)\rangle,
\end{split}
\]

For $\eta^{32}$, we expand $\mathcal{X}_H$ to obtain
\[
\begin{split}
\langle\eta^{32}\rangle = &-\langle X_H\big(K(\bar\theta,\bar\varphi),\bar\varphi\big)^\top \Omega\big(K(\bar\theta,\bar\varphi\big)E^K(\theta,\varphi)\rangle \\
&+ \langle\hat\alpha^\top{\rm D}_\varphi K(\bar\theta,\bar\varphi)^\top\Omega\big(K(\bar\theta,\bar\varphi)\big)E^K(\theta,\varphi)\rangle
\end{split}
\]
and define 
\[
\begin{split}
\eta^{32}_1(\theta,\varphi)&:=-X_H\big(K(\bar\theta,\bar\varphi,\bar\varphi\big)^\top \Omega\big(K(\bar\theta,\bar\varphi\big)E^K(\theta,\varphi), \\
\eta^{32}_2(\theta,\varphi)&:= \hat\alpha^\top{\rm D}_\varphi K(\bar\theta,\bar\varphi)^\top\Omega\big(K(\bar\theta,\bar\varphi)\big)E^K(\theta,\varphi)
\end{split}
\]
in order to inspect each term of $\eta^{32}$ separately. We first use that $ {\rm D}_zH(z,\varphi)=-X_H(z,\varphi)^\top\Omega(z)$ to express $\langle\eta^{32}_1\rangle$ as
\[
\langle \eta^{32}_1\rangle = \langle{\rm D}_z H\big(K(\bar\theta,\bar\varphi),\bar\varphi\big) E^K(\theta,\varphi)\rangle.
\]
Hence,
\[
\langle \eta^{32}_1\rangle = \langle H\big(\phi_T(K(\theta,\varphi),\varphi),\bar\varphi\big) - H\big(K(\theta,\varphi),\varphi\big)\rangle-\langle\Delta^2 H(\theta,\varphi) \rangle
\]
where we used that
\[
\langle H\big(K(\bar\theta,\bar\varphi),\bar\varphi\big) \rangle=\langle H\big(K(\theta,\varphi),\varphi\big)\rangle.
\]
Let $R_t(\varphi)=\varphi + \hat\alpha t $. We use the fundamental theorem of calculus so 
\[
\begin{split}
&H\Big(\phi_T\big(K(\theta,\varphi),\varphi\big),\bar\varphi\Big) - H\big(K(\theta,\varphi),\varphi\big)\\
&= \int^T_0 \frac{d}{dt}\bigg(H\Big(\phi_t\big(K(\theta,\varphi),\varphi\big),R_t(\varphi)\Big)\bigg) dt \\
& = \int_0^T \bigg({\rm D}_zH\Big(\phi_t\big(K(\theta,\varphi),\varphi\big),R_t(\varphi)\Big)X_H\Big(\phi_t\big(K(\theta,\varphi),\varphi\big),R_t(\varphi)\Big) \bigg) dt\\
 & \phantom{=}+ \int_0^T \bigg({\rm D}_\varphi H\Big(\phi_t\big(K(\theta,\varphi),\varphi\big),R_t(\varphi)\Big)\hat\alpha\bigg) dt,
\end{split}
\]
where the first integral varnishes because
\[
{\rm D}_z H(z,\varphi)X_H(z,\varphi) = -X_H(z,\varphi)^\top\Omega(z)X_H(z,\varphi) = 0.
\]
Consequently,
\[
\langle \eta^{32}_1\rangle = \langle  \int_0^T \bigg({\rm D}_\varphi H\Big(\phi_t\big(K(\theta,\varphi),\varphi\big),R_t(\varphi)\Big) \hat\alpha\bigg)\ dt  \rangle -\langle \Delta^2H(\theta,\varphi)\rangle.
\]
We now use \eqref{eq:ap_primitivevarphi} to express $\langle\eta^{32}_1\rangle$ as
\begin{align*}
  \langle\eta^{32}_1\rangle &= \langle a\Big(\phi_T\big(K(\theta,\varphi),\varphi\big)\Big)^\top {\rm D}_\varphi \phi_T\big(K(\theta,\varphi),\varphi\big)\hat\alpha \\ 
  &\phantom{=}-  {\rm D}_\varphi p_T\big(K(\theta,\varphi),\varphi\big)\hat\alpha \rangle -\langle \Delta^2H(\theta,\varphi)\rangle  \\
  &=\langle a\Big(\phi_T\big(K(\theta,\varphi),\varphi\big)\Big)^\top {\rm D}_\varphi \phi_T\big(K(\theta,\varphi),\varphi\big)\hat\alpha\rangle \\
  &\phantom{=} + \langle{\rm D}_z p_T\big(K(\theta,\varphi),\varphi\big) {\rm D}_\varphi K(\theta,\varphi)\hat\alpha\rangle  - \langle \Delta^2H(\theta,\varphi)\rangle,
\end{align*}
where we used 
\begin{align*}
    0 &= \langle{\rm D}_\varphi \Big(p_T\big(K(\theta,\varphi),\varphi\big) \Big) \rangle \\
    &=\langle{\rm D}_z p_T\big(K(\theta,\varphi),\varphi\big){\rm D}_\varphi K(\theta,\varphi) + {\rm D}_\varphi p_T\big(K(\theta,\varphi),\varphi\big) \rangle.
\end{align*}
We then use that
\[
\begin{split}
{\rm D}_\varphi E^K(\theta,\varphi) &=  {\rm D}_z \phi_T\big(K(\theta,\varphi),\varphi\big){\rm D}_\varphi K(\theta,\varphi) +  {\rm D}_\varphi\phi_T\big(K(\theta,\varphi),\varphi\big)\\
&\phantom{=}- {\rm D}_\varphi K(\bar\theta,\bar\varphi),
\end{split}
\]
in order to rewrite $\langle\eta^{32}_1\rangle$ as
\begin{align*}
 \langle\eta^{32}_1\rangle &=  \langle a\Big(\phi_T\big(K(\theta,\varphi),\varphi\big)\Big)^\top {\rm D}_\varphi E^K(\theta,\varphi)\hat\alpha\rangle \\
 &\phantom{=} -\langle a\Big(\phi_T\big(K(\theta,\varphi),\varphi\big)\Big)^\top {\rm D}_z \phi_T\big(K(\theta,\varphi,\varphi)\big){\rm D}_\varphi K(\theta,\varphi)\hat\alpha\rangle\\
 &\phantom{=} +\langle a\Big(\phi_T\big(K(\theta,\varphi),\varphi\big)\Big)^\top{\rm D}_\varphi K(\bar\theta,\bar\varphi)\hat\alpha\rangle\\
 &\phantom{=} + \langle {\rm D}_z p_T\big(K(\theta,\varphi),\varphi\big) {\rm D}_\varphi K(\theta,\varphi)\hat\alpha\rangle  - \langle \Delta^2H(\theta,\varphi)\rangle \\
 &= \langle a\Big(\phi_T\big(K(\theta,\varphi),\varphi\big)\Big)^\top {\rm D}_\varphi E^K(\theta,\varphi)\hat\alpha - a\big(K(\bar\theta,\bar\varphi)\big)^\top {\rm D}_\varphi K(\bar\theta,\bar\varphi) \hat\alpha\rangle \\
 &\phantom{=} +\langle a\Big(\phi_T\big(K(\theta,\varphi),\varphi\big)\Big)^\top{\rm D}_\varphi K(\bar\theta,\bar\varphi)\hat\alpha\rangle - \langle \Delta^2H(\theta,\varphi)\rangle,
 \end{align*}
where we used Eq. \eqref{eq:ap_primitiveZ} in the last equality and that
\[
\langle a\big(K(\theta,\varphi)\big)^\top {\rm D}_\varphi K(\theta,\varphi) \hat\alpha\rangle = \langle a\big(K(\bar\theta,\bar\varphi)\big)^\top {\rm D}_\varphi K(\bar\theta,\bar\varphi) \hat\alpha\rangle.
\]
We leave the term $\eta^{32}_1$ as it is. For the term $\eta^{32}_2$, we use \eqref{eq:exactness} to obtain
\[
\begin{split}
 \langle \eta^{32}_2 \rangle
= &\phantom{+}\langle \hat\alpha^\top{\rm D}_\varphi K(\bar\theta,\bar\varphi)^\top {\rm D} a\big(K(\bar\theta,\bar\varphi)\big)^\top E^K(\theta,\varphi)\rangle \\
 &- \langle \hat\alpha^\top {\rm D}_\varphi K(\bar\theta,\bar\varphi)^\top {\rm D} a\big(K(\bar\theta,\bar\varphi)\big) E^K(\theta,\varphi) \rangle.
\end{split}
\]
Then, we use that
\[
\begin{split}	
	0 = & \langle  {\rm D}_\varphi\Big( a\big(K(\bar\theta,\bar\varphi)\big)^\top E^K(\theta,\varphi) \Big)  \rangle \\
	   = & \langle E^K(\theta,\varphi)^\top  {\rm D}_\varphi \Big(a\big(K(\bar\theta,\bar\varphi)\big)\Big) + 
    a\big(K(\bar\theta,\bar\varphi)\big)^\top {\rm D}_\varphi E^K(\theta,\varphi)   \rangle
\end{split}	
\]
and the definitions for $\Delta^1 a$ and $\Delta^2 a$ to obtain
\[
\begin{split}
\langle \eta^{32}_2\rangle =
&-\langle a\big(K(\bar\theta,\bar\varphi)\big)^\top {\rm D}_\varphi E^K(\theta,\varphi)\hat\alpha \rangle \\ 
& - \langle\Big(\Delta^1 a(\theta,\varphi)^\top-\Delta^2 a(\theta,\varphi)^\top\Big){\rm D}_\varphi K(\bar\theta,\bar\varphi)\hat\alpha\rangle,
\end{split}
\]
where we used that $\eta^{32}_2 = \big(\eta^{32}_2 \big)^\top$ since $\eta^{32}_2$ is a scalar. 
Lastly, we obtain $\langle\eta^{32} \rangle$ as 
\begin{align*}
    \langle \eta^{32}\rangle & = \langle \eta^{32}_1\rangle + \langle \eta^{32}_2\rangle\\
    &= \langle a\Big(\phi_T\big(K(\theta,\varphi),\varphi\big)\Big)^\top {\rm D}_\varphi E^K(\theta,\varphi)\hat\alpha -  a\big(K(\bar\theta,\bar\varphi)\big)^\top {\rm D}_\varphi K(\bar\theta,\bar\varphi) \hat\alpha\rangle \\
 &\phantom{=} + \langle a\Big(\phi_T\big(K(\theta,\varphi),\varphi\big)\Big)^\top{\rm D}_\varphi K(\bar\theta,\bar\varphi)\hat\alpha\rangle \\
  &\phantom{=}-\langle a\big(K(\bar\theta,\bar\varphi)\big)^\top {\rm D}_\varphi E^K(\theta,\varphi)\hat\alpha - \Delta^1 a(\theta,\varphi)^\top  {\rm D}_\varphi K(\bar\theta,\bar\varphi)\hat\alpha\rangle\\ 
 &\phantom{=}  + \langle \Delta^2 a(\theta,\varphi)^\top {\rm D}_\varphi K(\bar\theta,\bar\varphi)\hat\alpha\rangle - \langle \Delta^2H(\theta,\varphi)\rangle\\
 &= \langle\Delta^1 a(\theta,\varphi)^\top  {\rm D}_\varphi E^K(\theta,\varphi)\hat\alpha\rangle + \langle \Delta^2 a(\theta,\varphi)^\top {\rm D}_\varphi K(\bar\theta,\bar\varphi)\hat\alpha\rangle\\
 &\phantom{=} - \langle \Delta^2H(\theta,\varphi)\rangle.
\end{align*}
\end{proof}

\section{Fiberwise symplectic deformations and moment maps}\label{ap:deformations}
In this section we introduce the notion of fiberwise symplectic deformations and establish their main properties. For a general exposition on symplectic deformations see \cite{GonzalezHL13}.

\begin{Def}
A fiberwise symplectic deformation in $U\subset\R^{2n}$ and with base $G\subset\R^m$ is a smooth diffeomorphism
\begin{align*}
\Phi:\quad U\times G &\longrightarrow U\times G \\
(z,\boldsymbol{t})&\longmapsto \big(\phi_{\boldsymbol{t}}(z),\tau(\bm t)\big),
\end{align*}
such that for all $\bm t\in G$, $\phi_{\bm t}:U\to U$ is symplectic and $\tau:G\to G$ is the base of the deformation. If for all $\bm t\in G$, $\phi_{\bm t}(z)=\phi(z,\bm{t})$ is exact symplectic, we will say that the fiberwise deformation is Hamiltonian.
\end{Def}

\begin{Def}
The primitive function of a fiberwise Hamiltonian deformation $\Phi:U\times G\to U\times G$ is a smooth function $p:U\times G\to \R$ such that for all $\bm t \in G$, the function $p_{\bm t}(z)=p(z,\bm t)$ is the primitive function of $\phi_{\bm t}(z)$. 
\end{Def}

\begin{Def}
Let $\Phi:U\times G \to U\times G$ be a fiberwise Hamiltonian deformation and let $p: U\times G \to \R$ be the primitive function of $\Phi$.
\begin{enumerate}[i)]
    \item The generator of $\Phi$ is the function $\mathcal{F}: U\times G\to \R^{2n\times m}$ defined as
    \begin{equation}\label{eq:generator}
    \mathcal{F}_{\bm t}(z) := {\rm D}_{\bm t} \phi\big(\phi_{\tau^{-1}(\bm t)}^{-1}(z),\tau^{-1}({\bm t})\big),
    \end{equation}
    where $\mathcal{F}_{\bm t}(z)=\mathcal{F}(z,\bm t) = \Big(\mathcal{F}^1(z,\bm t)\ \, \mathcal{F}^2(z,\bm t)\ \, \dots\ \,  \mathcal{F}^m(z,\bm t)    \Big)$.
    \item The moment map of $\Phi$ is the function $\mathcal{M}: U\times G \to \R^m$ defined as
    \[
        \mathcal{M}_{\bm t} (z)^\top := a(z)^\top  \mathcal{F}_{ \bm t}(z) - {\rm D}_{\bm t} p\big(\phi^{-1}_{\tau^{-1}(\bm t)}(z),\tau^{-1}({\bm t})\big),
    \]
    where $\mathcal{M}_{\bm t} (z) = \mathcal{M}(z,{\bm t}) = \Big(\mathcal{M}^1(z,\bm t)\ \, \mathcal{M}^2(z,\bm t) \ \,\dots \ \,\mathcal{M}^m(z,\bm t)    \Big)^\top $.
\end{enumerate}
\end{Def}

\begin{lemma} \label{lem:moment_gen}
For all ${\bm t} \in G$, the moment map $\mathcal{M}$ satisfies 
\[
\mathcal{F}_{\bm t} (z) = \Omega(z)^{-1} \big({\rm D}_z \mathcal{M}_{\bm t}(z)\big)^\top.
\]
\end{lemma}

\begin{proof}
For $i=1,2,\dots,m$, let us differentiate $\mathcal{M}^i\big(\phi_{\bm t}(z),\tau({\bm t})\big)$ with respect to $z$ as
\begin{align*}
{\rm D}_z & \Big( \mathcal{M}^i\big(\phi_{\bm t}(z),\tau({\bm t})\big)\Big)  = {\rm D}_z\Big( a\big(\phi_{\bm t}(z)\big)^\top \partial_{{t}_i}\phi_{\bm t}(z) - \partial_{{t}_i}p_{\bm t}(z)\Big)\\
&= a\big(\phi_{\bm t}(z)\big)^\top \partial_{{t}_i}{\rm D}_{z}\phi_{t}(z) + \partial_{{t}_i}\phi_{\bm t}(z)^\top {\rm D} a\big(\phi_{\bm t}(z)\big) {\rm D}_z \phi_{\bm t}(z)\\
& \phantom{=} - \partial_{{t}_i}\Big(a\big(\phi_{\bm t}(z) \big)^\top {\rm D}_z\phi_{\bm t}(z) - a(z)^\top\Big) \\
&= a\big(\phi_{\bm t}(z)\big)^\top \partial_{{t}_i}{\rm D}_{z}\phi_{\bm t}(z) + \partial_{{t}_i}\phi_{\bm t}(z)^\top {\rm D} a\big(\phi_{\bm t}(z)\big) {\rm D}_z \phi_{\bm t}(z)\\
& \phantom{=} - a\big(\phi_{\bm t}(z) \big)^\top \partial_{{t}_i}{\rm D}_{z}\phi_{\bm t}(z) - \partial_{{t}_i}\phi_{\bm t}(z)^\top {\rm D}a\big(\phi_{\bm t}(z)\big)^\top {\rm D}_z \phi_{\bm t}(z)\\
& = - \partial_{{t}_i}\phi_{\bm t}(z)^\top \Omega\big(\phi_{\bm t}(z)\big) {\rm D}_z \phi_{\bm t}(z).
\end{align*}
On the other hand, applying the chain rule we also have
\[
{\rm D}_z \Big(\mathcal{M}^i\big(\phi_{\bm t}(z),\tau({\bm t})\big)\Big) = {\rm D}_z \mathcal{M}^i\big(\phi_{\bm t}(z),\tau({\bm t})\big){\rm D}_z \phi_{\bm t}(z). 
\]
Assuming that ${\rm D}_z\phi_{\bm t}(z)$ is invertible, we obtain
\begin{equation*}
    {\rm D}_z\mathcal{M}^i\big(\phi_{\bm t}(z), \tau({\bm t})\big) = - \partial_{{t}_i}\phi_{\bm t}(z)^\top \Omega\big(\phi_{\bm t}(z)\big),
\end{equation*}
and evaluating at $\bar z=\phi^{-1}_{\tau^{-1}(\bm t)}(z)$ and $\bar {\bm t} = \tau^{-1}({\bm t})$, 
\begin{equation*}
    {\rm D}_z\mathcal{M}^i(z,\bm t) = - \partial_{{t}_i} \phi\big(\phi^{-1}_{\tau^{-1}(\bm t)}(z),\tau^{-1}(\bm t)\big)^\top \Omega(z).
\end{equation*}
Equivalently, using the definition of the generator given by \eqref{eq:generator}, we conclude
\begin{equation*}
    \mathcal{F}^{i}(z,\bm t) = \Omega(z)^{-1}\big({\rm D}_z \mathcal{M}^i(z,\bm t)\big)^\top
\end{equation*}
\end{proof}
\begin{remark} \label{ap:Moment-Ham}
Note that, for $i=1,2,...,m$, the moment map $\mathcal{M}$ gives the Hamiltonian $\mathcal{M}^{i}$ for the vector field $\mathcal{F}^ {i}$ obtained by differentiating the symplectomorphism $\phi_{\bm t}$ with respect to ${t}_i$.
\end{remark}

In the context of the present paper, let us consider flows of a quasi-periodic Hamiltonian system, defined by the function $H$, with frequencies $\hat \alpha\in\R^\ell$. As described in Section \ref{sec:setting}, for each $(z,\varphi)$ the flow $\tilde{\phi}: I_{z,\varphi}\times U\times \T^\ell\to U\times\T^\ell$, where $I_{z,\varphi}\subset\R$ is the maximal interval of existence for initial conditions $(z,\varphi)$, adopts the form
\begin{equation*}
  \tilde{\phi}_t(z,\varphi)=\begin{pmatrix}\phi_t(z,\varphi)\\ \varphi + \hat{\alpha}t \end{pmatrix},
\end{equation*}
where the evolution operator $\phi_t$ is fiberwise exact symplectic for all $t\in I_{z,\varphi}$. Let us first define the rotation operator $R_t(\varphi):=\varphi + \hat\alpha t$. We observe that for fixed $t$, and each $(z,\varphi)$, we can identify time-$t$ maps with fiberwise Hamiltonian deformations $\Phi: U \times G\to U\times G$ 
\begin{align*}
\Phi:\quad U\times G &\longrightarrow U\times G \\
(z,\boldsymbol{t})&\longmapsto \big(\phi_{t}(z,\varphi),\tau(\bm t)\big),
\end{align*}
where $G\subset \R^{\ell +1}$, ${\bm t} = (t,\varphi)$, and $\tau({\bm t}) = \big(t,R_t(\varphi) \big)$. For a quasi-periodic Hamiltonian $H_{\varepsilon}(z,\varphi)$ that depends on some parameter $\varepsilon\in\R$, we can again identify time-$t$ maps with fiberwise Hamiltonian deformations $\Phi: U \times G\to U\times G $
\begin{align*}
\Phi:\quad U\times G &\longrightarrow U\times G \\
(z,\boldsymbol{t})&\longmapsto \big(\phi_{t,\varepsilon}(z,\varphi),\tau(\bm t)\big),
\end{align*}
where $G\subset\R^{\ell +2}$, ${\bm t} = (t,\varepsilon,\varphi)$, and $\tau({\bm t}) = (t,\varepsilon,R_t(\varphi) \big)$. For this case, we will also write $\Phi_{\bm t}(z) = \Phi_{t,\varepsilon}(z,\varphi)$.

\begin{lemma}\label{lem:moment}
The moment map of $\Phi_{t,\varepsilon}(z,\varphi)$ is given by
\begin{align*}
    \mathcal{M}^t\big(\phi_{t,\varepsilon}(z,\varphi),\tau({\bm t})\big)&= H_\varepsilon\big(\phi_{t,\varepsilon}(z,\varphi),R_{t}(\varphi)\big)\\
    \mathcal{M}^\varepsilon\big(\phi_{t,\varepsilon}(z,\varphi),\tau({\bm t})\big) &= \int_0^t \partial_\varepsilon H_\varepsilon\big(\phi_{s,\varepsilon}(z,\varphi),R_s(\varphi)\big)\ ds \\
 \mathcal{M}^\varphi\big(\phi_{t,\varepsilon}(z,\varphi),\tau({\bm t})\big)   &= \int^t_0 {\rm D}_\varphi H_\varepsilon\big(\phi_{s,\varepsilon}(z,\varphi),R_s(\varphi)\big)\ ds
\end{align*}
\end{lemma}
\begin{proof}
For $\mathcal{M}^t\big(\phi_{t,\varepsilon}(z,\varphi),\tau({\bm t})\big)$, let us use the definition of the moment map and the primitive function of $\phi_t$ as given in \eqref{eq:prim}---which, for each $(t,\varepsilon,\varphi)$, coincides with the primitive function of the fiberwise Hamiltonian deformation $\Phi_{t,\varepsilon}(z,\varphi)$---to obtain 
\begin{align*}
    \mathcal{M}^t\big(\phi_{t,\varepsilon}(z,\varphi),\tau({\bm t})\big)  &= a\big(\phi_{t,\varepsilon}(z,\varphi) \big)^\top \partial_t \phi_{t,\varepsilon}(z,\varphi) - \partial_t p_{ t,\varepsilon}(z,\varphi)\\
     &= a\big(\phi_{t,\varepsilon}(z,\varphi) \big)^\top X_{H_\varepsilon}\big(\phi_{t,\varepsilon}(z,\varphi),R_t(\varphi) \big)\\
     & \phantom{=} - a\big(\phi_{t,\varepsilon}(z,\varphi)\big)^\top  X_{H_\varepsilon}\big(\phi_{t,\varepsilon}(z,\varphi),R_t(\varphi)\big) \\
     &\phantom{=}+ H_\varepsilon\big(\phi_{t,\varepsilon}(z,\varphi),R_t(\varphi)\big)\\
     &= H_\varepsilon\big(\phi_{t,\varepsilon}(z,\varphi),R_t(\varphi)\big).
\end{align*}
Note that this result could have been obtained directly from Remark \ref{ap:Moment-Ham}.

For $\mathcal{M}^\varepsilon\big(\phi_{t,\varepsilon}(z,\varphi),\tau({\bm t})\big)$, let us first differentiate the primitive function of $\Phi$ with respect to $\varepsilon$
\begin{align*}
    \partial_\varepsilon p_{t,\varepsilon}(z,\varphi) &=
    \partial_\varepsilon \Bigg(\int_0^t \bigg( a\big(\phi_{s,\varepsilon}(z,\varphi)\big)^\top  X_{H_\varepsilon}\big(\phi_{s,\varepsilon}(z,\varphi),R_s(\varphi)\big) \\
    &\phantom{=}-H_\varepsilon\big(\phi_{s,\varepsilon}(z,\varphi),R_s(\varphi)\big)  \bigg) ds \Bigg) \\
    &=\int_0^t \bigg(X_{H_\varepsilon}\big(\phi_{s,\varepsilon}(z,\varphi),R_s(\varphi)\big)^\top {\rm D}a\big(\phi_{s,\varepsilon}(z,\varphi)\big)\partial_\varepsilon\phi_{s,\varepsilon}(z,\varphi) \\
    &\phantom{=} + a\big(\phi_{s,\varepsilon}(z,\varphi)\big)^\top  \Big({\rm D}_z X_{H_\varepsilon}\big(\phi_{s,\varepsilon}(z,\varphi),R_s(\varphi)\big)     \partial_\varepsilon\phi_{s,\varepsilon}(z,\varphi)  \\
    &\phantom{=} + \partial_\varepsilon X_{H_\varepsilon}\big(\phi_{s,\varepsilon}(z,\varphi),R_s(\varphi)\big)\Big)
    \\
    &\phantom{=} -{\rm D}_z H_\varepsilon\big(\phi_{s,\varepsilon}(z,\varphi),R_s(\varphi)\big) \partial_\varepsilon\phi_{s,\varepsilon}(z,\varphi)\\
    &\phantom{=} - \partial_\varepsilon H_\varepsilon\big(\phi_{s,\varepsilon}(z,\varphi),R_s(\varphi)\big)
    \bigg)\ ds.
\end{align*}
We use that ${\rm D}_zH_\varepsilon(z,\varphi)=-X_{H_\varepsilon}^\top(z,\varphi)\Omega(z)$ and \eqref{eq:exactness} to obtain
\begin{align*}
    \partial_\varepsilon & p_{t,\varepsilon}(z,\varphi) =\\
    &\int_0^t \bigg(X_{H_\varepsilon}\big(\phi_{s,\varepsilon}(z,\varphi),R_s(\varphi)\big)^\top {\rm D}a\big(\phi_{s,\varepsilon}(z,\varphi)\big)^\top\partial_\varepsilon\phi_{s,\varepsilon}(z,\varphi)\\
    & \phantom{=} + a\big(\phi_{s,\varepsilon}(z,\varphi)\big)^\top  \Big({\rm D}_z X_{H_\varepsilon}\big(\phi_{s,\varepsilon}(z,\varphi),R_s(\varphi)\big)     \partial_\varepsilon\phi_{s,\varepsilon}(z,\varphi)  \\
    &\phantom{=} + \partial_\varepsilon X_{H_\varepsilon}\big(\phi_{s,\varepsilon}(z,\varphi),R_s(\varphi)\big)\Big)
    - \partial_\varepsilon H_\varepsilon\big(\phi_{s,\varepsilon}(z,\varphi),R_s(\varphi)\big)
    \bigg) ds.
\end{align*}
We can rewrite it as
\begin{align*}
   \partial_\varepsilon p_{t,\varepsilon}(z,\varphi) &= \int_0^t \bigg( \frac{d}{ds}\Big( a\big(\phi_{s,\varepsilon}(z,\varphi) \big)\Big)^\top \partial_\varepsilon\phi_{s,\varepsilon}(z,\varphi)\\
    &\phantom{=} + a\big(\phi_{s,\varepsilon}(z,\varphi) \big)^\top \frac{d}{ds}      \big(\partial_\varepsilon \phi_{s,\varepsilon}(z,\varphi) \big)  \\
    &\phantom{=} - \partial_\varepsilon H_\varepsilon \big(\phi_{s,\varepsilon}(z,\varphi),R_s(\varphi)\big) \bigg)ds 
\end{align*}    
and obtain  
\begin{align*}    
    \partial_\varepsilon p_{t,\varepsilon}(z,\varphi) &= \int^t_0 \bigg(\frac{d}{ds}\left(a\big(\phi_{s,\varepsilon}(z,\varphi)\big)^\top\partial_\varepsilon\phi_{s,\varepsilon}(z,\varphi)\right)\\
    &\phantom{=} -\partial_\varepsilon H_\varepsilon\big(\phi_{s,\varepsilon}(z,\varphi),R_s(\varphi) \big) \bigg)ds\\
    &= a\big(\phi_{t,\varepsilon}(z,\varphi)\big)^\top \partial_\varepsilon \phi_{t,\varepsilon}(z,\varphi)- \int^t_0 \partial_\varepsilon H_\varepsilon\big(\phi_{s,\varepsilon}(z,\varphi),R_s(\varphi) \big) ds.
    \end{align*}
Then, by definition, we have
\begin{align*}
\mathcal{M}^\varepsilon\big(\phi_{t,\varepsilon}(z,\varphi),\tau({\bm t})\big) &= a\big(\phi_{t,\varepsilon}(z,\varphi)\big)^\top \partial_\varepsilon\phi_{t,\varepsilon}(z,\varphi)-\partial_\varepsilon p_{t,\varepsilon}(z,\varphi)\\
&= \int^t_0 \partial_\varepsilon H_\varepsilon\big(\phi_{s,\varepsilon}(z,\varphi),R_s(\varphi) \big)\ ds.
\end{align*}
For $\mathcal{M}^\varphi\big(\phi_{t,\varepsilon}(z,\varphi),\tau(\bm t)\big)$, we will use the the expression of ${\rm D}_\varphi p_t$ given by \eqref{eq:ap_primitivevarphi}. Then for the primitive function of $\Phi$, we have
\begin{align*}
{\rm D}_\varphi p_{t,\varepsilon}&(z,\varphi) = \\ 
& a\big(\phi_{t,\varepsilon}(z,\varphi)\big)^\top{\rm D}_\varphi \phi_{t,\varepsilon}(z,\varphi) - \int^t_0 \Big({\rm D}_\varphi H_\varepsilon\big(\phi_{s,\varepsilon}(z,\varphi),R_s(\varphi)\big) \Big)ds.
\end{align*}
Hence, 
\begin{align*}
\mathcal{M}^\varphi\big(\phi_{t,\varepsilon}(z,\varphi),\tau({\bm t})\big) &= a\big(\phi_{t,\varepsilon}(z,\varphi)\big)^\top {\rm D}_\varphi \phi_{t,\varepsilon}(z,\varphi)-{\rm D}_\varphi p_{t,\varepsilon}(z,\varphi)\\
&= \int^t_0 {\rm D}_\varphi H_\varepsilon\big(\phi_{s,\varepsilon}(z,\varphi),R_s(\varphi) \big)\ ds.
\end{align*}
Note that it is possible to obtain expressions for $\mathcal{M}(z,\bf{t})$, but we provide here the expressions for $\mathcal{M}(\phi_{t,\varepsilon}(z,\varphi),\tau({\bm t})\big)$ since they will be useful in the next section.

\end{proof}

\section{Zero averages}\label{ap:zero_average}
For the continuation of an invariant torus $\mathcal{K}$ with respect to parameters of the Hamiltonian we need the averages of $\eta^3$ given by 
\begin{equation*}
	\eta^3(\theta,\varphi) = 
	\begin{pmatrix} 
	-{\rm D}_\theta K(\bar\theta,\bar\varphi)^\top \Omega\big(K(\bar\theta,\bar\varphi)\big) \partial_\varepsilon \phi_{T,\varepsilon}\big(K(\theta,\varphi),\varphi\big)
	\\
	-\mathcal{X}_H(\bar\theta,\bar\varphi)^\top \Omega\big(K(\bar\theta,\bar\varphi)\big) \partial_\varepsilon \phi_{T,\varepsilon}\big(K(\theta,\varphi),\varphi\big)
	\end{pmatrix}
\end{equation*}
to be zero in order for the small divisors cohomological equations from Section \ref{sec:cnte} to be solvable. Recall the definitions $\bar\theta := \theta + \omega$ and $\bar\varphi=\varphi + \alpha$, and let us define
\begin{equation*}
	\begin{pmatrix}
	\eta^{31}(\theta,\varphi)
	\\
	\eta^{32}(\theta,\varphi)
	\end{pmatrix}	:= \begin{pmatrix} 
	-{\rm D}_\theta K(\bar\theta,\bar\varphi)^\top \Omega\big(K(\bar\theta,\bar\varphi)\big) \partial_\varepsilon \phi_{T,\varepsilon}\big(K(\theta,\varphi),\varphi\big)
	\\
	-\mathcal{X}_H(\bar\theta,\bar\varphi)^\top \Omega\big(K(\bar\theta,\bar\varphi)\big) \partial_\varepsilon \phi_{T,\varepsilon}\big(K(\theta,\varphi),\varphi\big)
	\end{pmatrix}. 
\end{equation*}
\begin{lemma}
The averages of $\eta^{31}$ and $\eta^{32}$ are zero.
\end{lemma}
\begin{proof}
Let us consider the following fiberwise Hamiltonian deformation $\Phi_{T,\varepsilon}:U \times G \to U\times G$
\begin{align*}
\Phi:\quad U\times G &\longrightarrow U\times G \\
(z,\boldsymbol{t})&\longmapsto \big(\phi_{T,\varepsilon}(z,\varphi),\tau(\bm t)\big),
\end{align*}
with $G\subset\R^{\ell +2}$, $\bm t = (T,\varepsilon,\varphi)$, and $\tau(\bm t) = \big(T,\varepsilon, \bar\varphi\big)$, see Appendix  \ref{ap:deformations}. We can then use the definition of the generator $\mathcal{F}^{\varepsilon}
$ of the fiberwise Hamiltonian deformation as given in \eqref{eq:generator} and obtain for $\langle\eta^{31}\rangle$
\begin{equation*}
\langle\eta^{31} \rangle = -\langle {\rm D}_\theta K(\bar\theta,\bar\varphi)^\top \Omega\big(K(\bar\theta,\bar\varphi)\big)\mathcal{F}_{T,\varepsilon}^\varepsilon\Big(\phi_{T,\varepsilon}\big(K(\theta,\varphi),\varphi\big),\bar\varphi\Big) \rangle.
\end{equation*}
Using the invariance of $\mathcal{K}$ and Lemma \ref{lem:moment_gen}, we have
\begin{align*}
\langle \eta^{31}\rangle &=-\langle {\rm D}_\theta K(\bar\theta,\bar\varphi)^\top \Omega\big(K(\bar\theta,\bar\varphi)\big)\mathcal{F}_{T,\varepsilon}^\varepsilon\big(K(\bar\theta,\bar\varphi),\bar\varphi\big) \rangle\\
&=-\langle {\rm D}_\theta K(\bar\theta,\bar\varphi)^\top{\rm D}_z\mathcal{M}^\varepsilon_{T,\varepsilon}\big(K(\bar\theta,\bar\varphi),\bar\varphi \big)^\top    \rangle\\
&=-\langle {\rm D}_\theta\Big(\mathcal{M}^\varepsilon_{T,\varepsilon}\big(K(\bar\theta,\bar\varphi),\bar\varphi \big) \Big)^\top \rangle = 0
\end{align*}
since we are taking averages of derivatives with respect to $\theta$.

For $\langle \eta^{32}\rangle$, let us use Lemma \ref{lem:moment_gen} and that 
\begin{equation*}
{\rm D}_z\phi_{T,\varepsilon}\big(K(\theta,\varphi),\varphi\big)\mathcal{X}_H(\theta,\varphi)=\mathcal{X}_H(\bar{\theta},\bar{\varphi})
\end{equation*}
to rewrite $\langle \eta^{32}\rangle$ as
\begin{equation*}
\langle \eta^{32}\rangle  =
 -\langle \mathcal{X}_H(\theta,\varphi)^\top {\rm D}_z \bigg(\mathcal{M}^\varepsilon_{T,\varepsilon}\Big(\phi_{T,\varepsilon}\big(K(\theta,\varphi),\varphi\big),\bar\varphi\Big)     \bigg)^\top  \rangle.
\end{equation*}
Let us transpose the previous expression and use the explicit form of $\mathcal{M}^\varepsilon_{T,\varepsilon}\Big(\phi_{T,\varepsilon}\big(K(\theta,\varphi),\varphi\big),\bar\varphi \Big)$ as given in Lemma \ref{lem:moment}. Hence, using the rotation operator $R_s(\varphi) = \varphi + \hat\alpha s$, we have
\begin{align*}
 \langle  \eta^{32}\rangle = -\langle   \int_0^T\bigg( {\rm D}_z \partial_\varepsilon &H_\varepsilon\big(\phi_{s,\varepsilon}\big(K(\theta,\varphi),\varphi\big),R_s(\varphi)\Big)\\
 & {\rm D}_z \phi_{s,\varepsilon}\big(K(\theta,\varphi),\varphi \big)\mathcal{X}_H(\theta,\varphi) \bigg) \ ds \rangle.
\end{align*}
We then expand $\mathcal{X}_H$ to rewrite $\langle\eta^{32}\rangle$ as
\begin{align*}
\langle\eta^{32} \rangle & = -\langle \int_0^T\bigg( {\rm D}_z \partial_\varepsilon H_\varepsilon\Big(\phi_{s,\varepsilon}\big(K(\theta,\varphi),\varphi\big),R_s(\varphi) \Big) {\rm D}_z \phi_{s,\varepsilon}\big(K(\theta,\varphi),\varphi \big)\\
&\phantom{=} \Big(X_{H_\varepsilon}\big(K(\theta,\varphi),\varphi\big) - {\rm D}_\varphi K(\theta,\varphi)\hat\alpha \Big) \bigg)\ ds \rangle
\end{align*}
and we use that 
\[
X_{H_\varepsilon}\big(\phi_{t,\varepsilon}(z,\varphi),R_t(\varphi) \big) = {\rm D}_z\phi_{t,\varepsilon}(z,\varphi) X_{H_\varepsilon}(z,\varphi) + {\rm D}_\varphi \phi_{t,\varepsilon}(z,\varphi) \hat\alpha,
\]
see Section \ref{sec:frm}, to obtain
\begin{align*}
\langle&\eta^{32}\rangle =\\
& \langle \int_0^T \bigg(-{\rm D}_z \partial_\varepsilon H_\varepsilon\Big(\phi_{s,\varepsilon}\big(K(\theta,\varphi),\varphi\big),R_s(\varphi) \Big) \frac{d}{ds}\Big(\phi_{s,\varepsilon}\big(K(\theta,\varphi),\varphi \big) \Big)  \\
& +  {\rm D}_z \partial_\varepsilon H_\varepsilon\Big(\phi_{s,\varepsilon}\big(K(\theta,\varphi),\varphi\big),R_s(\varphi) \Big)  {\rm D}_z \phi_{s,\varepsilon}\big(K(\theta,\varphi),\varphi \big) {\rm D}_\varphi K(\theta,\varphi)\hat\alpha \\ 
&+  {\rm D}_z \partial_\varepsilon H_\varepsilon\Big(\phi_{s,\varepsilon}\big(K(\theta,\varphi),\varphi\big),R_s(\varphi) \Big)  {\rm D}_\varphi \phi_{s,\varepsilon}\big(K(\theta,\varphi),\varphi \big)\hat\alpha
     \bigg) ds\rangle. 
\end{align*}
We can then rewrite it as

\begin{align*}
 \langle&\eta^{32}\rangle = \\
 & \langle \int_0^T \bigg(-{\rm D}_z \partial_\varepsilon H_\varepsilon\Big(\phi_{s,\varepsilon}\big(K(\theta,\varphi),\varphi\big),R_s(\varphi) \Big)\frac{d}{ds}\Big(\phi_{s,\varepsilon}\big(K(\theta,\varphi),\varphi \big) \Big)      \\
 & - {\rm D}_\varphi \partial_\varepsilon H_\varepsilon\Big(\phi_{s,\varepsilon}\big(K(\theta,\varphi),\varphi\big),R_s(\varphi)\Big)\hat\alpha \\
 & + {\rm D}_\varphi \partial_\varepsilon H_\varepsilon\Big(\phi_{s,\varepsilon}\big(K(\theta,\varphi),\varphi\big),R_s(\varphi) \Big)\hat\alpha \\ 
& +  {\rm D}_z \partial_\varepsilon H_\varepsilon\Big(\phi_{s,\varepsilon}\big(K(\theta,\varphi),\varphi\big),R_s(\varphi) \Big)  {\rm D}_z \phi_{s,\varepsilon}\big(K(\theta,\varphi),\varphi \big) {\rm D}_\varphi K(\theta,\varphi)\hat\alpha \\ 
&+  {\rm D}_z \partial_\varepsilon H_\varepsilon\Big(\phi_{s,\varepsilon}\big(K(\theta,\varphi),\varphi\big),R_s(\varphi) \Big)  {\rm D}_\varphi \phi_{s,\varepsilon}\big(K(\theta,\varphi),\varphi \big)\hat\alpha
     \bigg) ds\rangle
\end{align*}

to finally obtain
\begin{align*}
\langle\eta^{32}\rangle&= -\langle \int_0^T \frac{d}{ds}\bigg( \partial_\varepsilon H_\varepsilon \Big(\phi_{s,\varepsilon}\big(K(\theta,\varphi),\varphi\big),R_s(\varphi) \Big)\bigg) \ ds   \rangle \\
&\phantom{=} + \langle \int_0^T {\rm D}_\varphi \bigg(\partial_\varepsilon H_\varepsilon \Big(\phi_{s,\varepsilon}\big(K(\theta,\varphi),\varphi\big),R_s(\varphi) \Big)  \bigg)\ ds\rangle\hat\alpha \\
&= -\langle \partial_\varepsilon H_\varepsilon \big(K(\bar\theta,\bar\varphi),\bar\varphi \big) - \partial_\varepsilon H_\varepsilon \big(K(\theta,\varphi),\varphi \big)   \rangle \\
&\phantom{=} + \langle {\rm D}_\varphi\bigg( \mathcal{M}^\varepsilon_{T,\varepsilon}\Big(\phi_{T,\varepsilon}\big(K(\theta,\varphi),\varphi\big),\bar\varphi\Big)\bigg)  \rangle \hat\alpha = 0,
\end{align*}
where we used that $\partial_\varepsilon H_\varepsilon\big(K(\theta,\varphi),\varphi\big)$ and $\partial_\varepsilon H_\varepsilon\big(K(\bar\theta,\bar\varphi),\bar\varphi\big)$ have the same average and that the average of derivatives with respect to $\varphi$ is zero.

\end{proof}
 % appendix.tex

\bibliographystyle{plain}
\bibliography{references} % references.bib

\end{document}